\documentclass[a4paper]{article}%
\usepackage{amssymb}
\usepackage{amsmath}
\usepackage{amsfonts}
\usepackage{graphicx}
\usepackage{amstext}
\usepackage{amscd}%
\providecommand{\U}[1]{\protect\rule{.1in}{.1in}}
\newtheorem{theorem}{Theorem}

\newtheorem{corollary}[theorem]{Corollary}

\newtheorem{definition}[theorem]{Definition}

\newtheorem{lemma}[theorem]{Lemma}

\newtheorem{proposition}[theorem]{Proposition}
\newtheorem{remark}[theorem]{Remark}

\newenvironment{proof}[1][Proof]{\noindent\textbf{#1.} }{\ \rule{0.5em}{0.5em}}
\begin{document}

\title{Stochastic stability of the classical Lorenz flow under impulsive type forcing}
\author{Michele Gianfelice$%
%TCIMACRO{\U{b0}}%
%BeginExpansion
{{}^\circ}%
%EndExpansion
$, Sandro Vaienti$^{\S }$\\{\small $%
%TCIMACRO{\U{b0}}%
%BeginExpansion
{{}^\circ}%
%EndExpansion
$Dipartimento di Matematica e Informatica, Universit\`{a} della Calabria} \\{\small Ponte Pietro Bucci, cubo 30B, I-87036 Arcavacata di Rende (CS)} \\{\small \texttt{gianfelice@mat.unical.it}} \\{\small $^{\S }$Aix Marseille Universit\'{e}, Universit\'{e} de Toulon, CNRS,
CPT, Marseille, France} \\{\small \texttt{vaienti@cpt.univ-mrs.fr}}}
\maketitle

\begin{abstract}
We introduce a novel type of random perturbation for the classical Lorenz flow
in order to better model phenomena slowly varying in time such as
anthropogenic forcing in climatology and prove stochastic stability for the
unperturbed flow. The perturbation acts on the system in an impulsive way,
hence is not of diffusive type as those already discussed in \cite{Ki},
\cite{Ke}, \cite{Me}. Namely, given a cross-section $\mathcal{M}$ for the
unperturbed flow, each time the trajectory of the system crosses $\mathcal{M}$
the phase velocity field is changed with a new one sampled at random from a
suitable neighborhood of the unperturbed one. The resulting random evolution
is therefore described by a piecewise deterministic Markov process. The proof
of the stochastic stability for the umperturbed flow is then carryed on
working either in the framework of the Random Dynamical Systems or in that of
semi-Markov processes.

\end{abstract}
\tableofcontents

\bigskip

\begin{description}
\item[AMS\ subject classification:] {\small 34F05, 93E15. }

\item[Keywords and phrases:] {\small Random perturbations of dynamical
systems, classical Lorenz model, random dynamical systems, semi-Markov random
evolutions, piecewise deterministic Markov processes, Lorenz'63 model,
anthropogenic forcing. }

\item[Acknowledgement:] {\small M. Gianfelice was partially supported by LIA
LYSM AMU-CNRS-ECM-INdAM. S. Vaienti was supported by the Leverhulme Trust for
support thorough the Network Grant IN-2014-021 and by the project APEX
\textit{Syst\`{e}mes dynamiques: Probabilit\'{e}s et Approximation
Diophantienne PAD} funded by the R\'{e}gion PACA (France). S.Vaienti warmly
thanks the Laboratoire International Associ\'{e} LIA LYSM, the LabEx
Archim\`{e}de (AMU University, Marseille), the INdAM (Italy) }and the UMI-CNRS
3483 Laboratoire Fibonacci (Pisa) where this work has been completed under a
CNRS delegation.
\end{description}

\bigskip

\part{Introduction, notations and results}

\section{The classical Lorenz flow}

The physical behaviour of turbulent systems such the atmosphere are usually
modeled by flows exhibiting a sensitive dependence on the initial conditions.
The behaviour of the trajectories of the system in the phase space for large
times is usually numerically very hard to compute and consequently the same
computational difficulty affects also the computation of the phase averages of
physically relevant observables. A way to overcome this problem is to select a
few of these relevant observables under the hypothesis that the statistical
properties of the smaller system defined by the evolution of such quantities
can capture the important features of the statistical behaviour of the
original system \cite{NVKDF}.

As a matter of fact this turns out to be the case when considering
\emph{classical Lorenz model}, a.k.a. \emph{Lorenz'63 model} in the physics
literature, i.e. the system of equation
\begin{equation}
\left\{
\begin{array}
[c]{l}%
\dot{x}_{1}=-\zeta x_{1}+\zeta x_{2}\\
\dot{x}_{2}=-x_{1}x_{3}+\gamma x_{1}-x_{2}\\
\dot{x}_{3}=x_{1}x_{2}-\beta x_{3}%
\end{array}
\right.  \ , \label{L}%
\end{equation}
which was introduced by Lorenz in his celebrated paper \cite{Lo} as a
simplified yet non trivial model for thermal convection of the atmosphere and
since then it has been pointed out as the typical real example of a
non-hyperbolic three-dimensional flow whose trajectories show a sensitive
dependence on initial conditions. In fact, the classical Lorenz flow, for
$\zeta=10,\gamma=28,\beta=8/3,$ has been proved in \cite{Tu}, and more
recently in \cite{AM}, to show the same dynamical features of its ideal
counterpart the so called \emph{geometric Lorenz flow}, introduced in
\cite{ABS} and in \cite{GW}, which represents the prototype of a
three-dimensional flow exhibiting a partially hyperbolic attractor \cite{AP}.
The Lorenz'63 model, indeed, has the interesting feature that it can be
rewritten as
\begin{equation}
\left\{
\begin{array}
[c]{l}%
\dot{y}_{1}=-\zeta y_{1}+\zeta y_{2}\\
\dot{y}_{2}=-y_{1}y_{3}-\gamma y_{1}-y_{2}\\
\dot{y}_{3}=y_{1}y_{2}-\beta y_{3}-\beta\left(  \gamma+\zeta\right)
\end{array}
\right.  \ , \label{L1}%
\end{equation}
showing the corresponding flow to be generated by the sum of a Hamiltonian
$SO\left(  3\right)  $-invariant field and a gradient field (we refer the
reader to \cite{GMPV} and references therein). Therefore, as it has been
proved in \cite{GMPV}, the invariant measure of the classical Lorenz flow can
be constructed starting from the invariant measure of the one-dimensional
system describing the evolution of the extrema of the first integrals of the
associated Hamiltonian flow.

\subsection{Stability of the invariant measure of the Lorenz'63 flow}

Since $C^{1}$ perturbations of the classical Lorenz vector field admit a
$C^{1+\epsilon}$ stable foliation \cite{AM} and since the geometric Lorenz
attractor is robust in the $C^{1}$ topology \cite{AP}, it is natural to
discuss the statistical and the stochastic stability of the classical Lorenz
flow under this kind of perturbations.

Indeed, in applications to climate dynamics, when considering the Lorenz'63
flow as a model for the atmospheric circulation, the analysis of the stability
of the statistical properties of the unperturbed flow under perturbations of
the velocity phase field of this kind can turn out\ to be a useful tool in the
study of the so called \emph{anthropogenic climate change} \cite{CMP}.

\subsubsection{Statistical stability}

For what concerns the statistical stability, in \cite{GMPV} it has been shown
that the effect of an additive constant perturbation term to the classical
Lorenz vector field results into a particular kind of perturbation of the map
of the interval describing the evolution of the maxima of the Casimir function
for the (+) Lie-Poisson brackets associated to the $so\left(  3\right)  $
algebra. Moreover, it has been proved that the invariant measures for the
perturbed and for the unperturbed 1-$d$ maps of this kind have Lipschitz
continuous density and that the unperturbed invariant measure is strongly
statistically stable. Since the SRB measure of the classical Lorenz flow can
be constructed starting from the invariant measure of the one-dimensional map
obtained through reduction to the quotient leaf space of the Poincar\'{e} map
on a two-dimensional manifold transverse to the flow \cite{AP}, the
statistical stability for the invariant measure of this map implies that of
the SRB measure of the unperturbed flow. Other results in this direction are
given in \cite{AS}, \cite{BR} and \cite{GL} where strong statistical stability
of the geometric Lorenz flow is analysed.

\subsubsection{Random perturbations}

Random perturbations of the classical Lorenz flow have been studied in the
framework of stochastic differential equations \cite{Sc}, \cite{CSG},
\cite{Ke} (see also \cite{Ar} and reference therein). The main interest of
these studies was bifurcation theory and the existence and the
characterization of the random attractor. The existence of the stationary
measure for this stochastic version of the system of equations given in
(\ref{L1}) is proved in \cite{Ke}.

Stochastic stability under diffusive type perturbations has been studied in
\cite{Ki} for the geometric Lorenz flow and in \cite{Me} for the contracting
Lorenz flow.

\section{Physical motivation}

The analysis of the stability of the statistical properties of the classical
Lorenz flow can provide a theoretical framework for the study of climate
changes, in particular those induced by the anthropogenic influence on climate dynamics.

A possible way to study this problem is to add a weak perturbing term to the
phase vector field generating the atmospheric flow which model the atmospheric
circulation: the so called \emph{anthropogenic forcing}. Assuming that the
atmospheric circulation is described by a model exhibiting a robust singular
hyperbolic attractor, as it is the case for the classical Lorenz flow, it has
been shown empirically that the effect of the perturbation can possibly affect
just the statistical properties of the system \cite{Pa}, \cite{CMP}.
Therefore, because of its very weak nature (small intensity and slow
variability in time), a practical way to measure the impact of the
anthropogenic forcing on climate statistics is to look at the extreme value
statistics of those particular observables whose evolution may be more
sensitive to it \cite{Su}. In the particular case these observables are given
by bounded (real valued) functions on the phase space, an effective way to
look at their extreme value statistics is to look first at the statistics of
their extrema and then eventually to the extreme value statistics of these.

We stress that the result presented in \cite{GMPV} fit indeed in this
framework since, starting from the assumption made in \cite{Pa} and \cite{CMP}
that, taking the classical Lorenz flow as a model for the atmospheric
circulation, the effect of the anthropogenic influence on climate dynamics can
be modeled by the addition of a small constant term to the unperturbed phase
vector field, it has been shown that the statistics of the extrema of the
first integrals of the Hamiltonian flow underlying the classical Lorenz one,
which are global observables for this system, are very sensitive to this kind
of perturbation (see e.g. Example 8 in \cite{GMPV}).

Of course, a more realistic model for the anthropogenic forcing should take
into account random perturbations of the phase vector field rather than
deterministic ones. Anyway it seems unlikely that the resulting process can be
a diffusion, since in this case the driving process fluctuates faster than
what it is assumed to do in principle a perturbing term of the type just described.

\subsection{Modeling random perturbations of impulsive type}

We introduce a random perturbation of the Lorenz'63 flow which, being of
impulsive nature, differ from diffusion-type perturbations.

For any realization of the noise $\eta\in\left[  -\varepsilon,\varepsilon
\right]  ,$ we consider a flow $\left(  \Phi_{\eta}^{t},t\geq0\right)  $
generated by the phase vector field $\phi_{\eta}$ belonging to a sufficiently
small neighborhood of the classical Lorenz one in the $C^{1}$ topology. For
$\varepsilon$ small enough, the realizations of the perturbed phase vector
field $\phi_{\eta}$ can be chosen such that there exists an open neighborhood
$U$ of the unperturbed attractor in $\mathbb{R}^{3},$ independent of the noise
parameter $\eta,$ containing the attractor of any realization of $\phi_{\eta}$
and, moreover, such that a given Poincar\'{e} section $\mathcal{M}$ for the
unperturbed flow is also transversal to any realization of the perturbed one.
Thus, given $\mathcal{M},$ the random process describing the perturbation is
constructed selecting at random, in an independent way, the value of
$\phi_{\eta}$ at the crossing of $\mathcal{M}$ by the phase trajectory.

This procedure defines a semi-Markov random evolution \cite{KS}, in fact a
piecewise deterministic Markov process (PDMP) \cite{Da}.

Therefore, the major object of this paper will be to show the existence of a
stationary measure for the imbedded Markov chain driving the random process
just described as well as to prove that the stationary process weakly
converges, as $\varepsilon$ tends to $0,$ to the physical measure of the
unperturbed one.

More specifically, let $\hat{\tau}_{\eta}:U\rightarrow\mathcal{M}$ and
$\tau_{\eta}:\mathcal{M}\circlearrowleft$ be respectively the hitting time of
$\mathcal{M}$ and the return time map on $\mathcal{M}$ for $\left(  \Phi
_{\eta}^{t},t\geq0\right)  .$ If $\eta$ is sampled according to a given law
$\lambda_{\varepsilon}$ supported on $\left[  -\varepsilon,\varepsilon\right]
,$ the sequence $\left\{  \mathfrak{x}_{i}\right\}  _{i\geq0}$ such that
$\mathfrak{x}_{0}\in\mathcal{M}$ and, for $i\geq0,\mathfrak{x}_{i+1}%
:=\Phi_{\eta}^{\tau_{\eta}\left(  \mathfrak{x}_{i}\right)  }\left(
\mathfrak{x}_{i}\right)  $ is a homogeneous Markov chain on $\mathcal{M}$ with
transition probability measure
\begin{equation}
\mathbb{P}\left\{  \mathfrak{x}_{1}\in dz|\mathfrak{x}_{0}\right\}
=\lambda_{\varepsilon}\left\{  \eta\in\left[  -1,1\right]  :R_{\eta}\left(
\mathfrak{x}_{0}\right)  \in dz\right\}  \ .
\end{equation}
Considering the collection of sequences of i.i.d.r.v's $\left\{  \eta
_{i}\right\}  _{i\geq0}$ distributed according to $\lambda_{\varepsilon},$ we
define the random sequence $\left\{  \sigma_{n}\right\}  _{n\geq1}%
\in\mathbb{R}$ such that $\sigma_{n}:=\sum_{i=0}^{n-1}\tau_{\eta_{i-1}}\left(
\mathfrak{x}_{i-1}\right)  ,n\geq1.$ Then, it is easily checked that the
sequence $\left\{  \left(  \mathfrak{x}_{n},\mathbf{t}_{n}\right)  \right\}
_{n\geq0}$ such that $\mathbf{t}_{0}:=\sigma_{1}$ and, for $n\geq
0,\mathbf{t}_{n}:=\sigma_{n+1}-\sigma_{n}$ is a Markov renewal process (MRP)
\cite{As}, \cite{KS}. Therefore, denoting by $\left(  \mathbf{N}_{t}%
,t\geq0\right)  ,$ such that $\mathbf{N}_{0}:=0$ and $\mathbf{N}_{t}%
:=\sum_{n\geq0}\mathbf{1}_{\left[  0,t\right]  }\left(  \sigma_{n}\right)  ,$
the associated counting process and defining:

\begin{itemize}
\item $\left(  \mathfrak{x}_{t},t\geq0\right)  ,$ such that $\mathfrak{x}%
_{t}:=\mathfrak{x}_{\mathbf{N}_{t}},$ the associated semi-Markov process;

\item $\left(  \mathfrak{l}_{t},t\geq0\right)  ,$ such that $\mathfrak{l}%
_{t}:=t-\sigma_{\mathbf{N}_{t}},$ the \emph{age} (\emph{residual life}) of the MRP;

\item $\left(  \eta_{t},t\geq0\right)  $ such that $\eta_{t}:=\eta
_{\mathbf{N}_{t}},$
\end{itemize}

setting $\sigma_{0}:=\hat{\tau}_{\eta},$ we introduce the random process
$\left(  \mathfrak{u}_{t},t\geq0\right)  ,$ such that
\begin{equation}
\mathfrak{u}_{t}\left(  y_{0}\right)  :=\left\{
\begin{array}
[c]{ll}%
\Phi_{\eta}^{t}\left(  y_{0}\right)  \mathbf{1}_{[0,\sigma_{0}\left(
y_{0}\right)  )}\left(  t\right)  +\mathbf{1}_{\left\{  \Phi_{\eta}%
^{\sigma_{0}\left(  y_{0}\right)  }\left(  y_{0}\right)  \right\}  }\left(
\mathfrak{x}_{0}\right)  \Phi_{\eta_{\mathbf{N}_{t-\sigma_{0}\left(
y_{0}\right)  }}}^{\mathfrak{l}_{t-\sigma_{0}\left(  y_{0}\right)  }}%
\circ\mathfrak{x}_{t-\sigma_{0}\left(  y_{0}\right)  } & y_{0}\in
U\backslash\mathcal{M}\\
\mathbf{1}_{\left\{  y_{0}\right\}  }\left(  \mathfrak{x}_{0}\right)
\Phi_{\eta_{\mathbf{N}_{t}}}^{\mathfrak{l}_{t}}\circ\mathfrak{x}_{t} &
y_{0}\in\mathcal{M}%
\end{array}
\right.  \;;\;y_{0}\in U,t\geq0\ , \label{u_t}%
\end{equation}
describes the system evolution started at $y_{0}.$ We prove

\begin{theorem}
\label{main}There exists a measure $\mu_{\varepsilon}$ on the measurable space
$\left(  U,\mathcal{B}\left(  U\right)  \right)  ,$ with $\mathcal{B}\left(
U\right)  $ the trace $\sigma$algebra of the Borel $\sigma$algebra of
$\mathbb{R}^{3},$ such that, for any bounded real-valued measurable function
$f$ on $U,$%
\begin{equation}
\lim_{t\rightarrow\infty}\frac{1}{T}\int_{0}^{T}f\circ\mathfrak{u}_{t}%
=\mu_{\varepsilon}\left(  f\right)
\end{equation}
and
\begin{equation}
\lim_{\varepsilon\downarrow0}\mu_{\varepsilon}\left(  f\right)  =\mu
_{0}\left(  f\right)
\end{equation}
where $\mu_{0}$ is the physical measure of the classical Lorenz flow.
\end{theorem}

A more precise definition of the quantities involved in the construction of
$\left(  \mathfrak{u}_{t},t\geq0\right)  $ is given in the second part of the
paper where we also present a different characterization of this random
process, which follows from the representation of the Markov chain $\left\{
\mathfrak{x}_{n}\right\}  _{n\geq0}$ as Random Dynamical System (RDS), and
study its asymptotic stationary properties. In the third part of the paper we
present the construction of $\left(  \mathfrak{u}_{t},t\geq0\right)  $ just
given in a more rigorous way and reprhase the analysis carried on in the
second part of the paper in the framework of PDMP's.

One may argue that the perturbation should act modifying the phase velocity
field of the system at any point of $U$ and not just at the crossing of a
given cross-section. In fact, let $\left\{  \mathfrak{t}_{n}\right\}
_{n\geq0}$ be the sequence of i.i.d.r.v's representing the jump times of this
process, which we choose independent of the noise parameter $\eta.\left\{
\mathfrak{S}_{n}\right\}  _{n\geq1},$ such that $\mathfrak{S}_{n}:=\sum
_{k=0}^{n-1}\mathfrak{t}_{n},$ is the associated renewal process and $\left(
\mathfrak{n}_{t},t\geq0\right)  ,$ such that $\mathfrak{n}_{t}:=\sum_{n\geq
0}\mathbf{1}_{\left[  0,t\right]  }\left(  \mathfrak{S}_{n}\right)  ,$ is the
associated counting process. The sequence $\left\{  \mathfrak{z}_{n}\right\}
_{n\geq0},$ such that $\mathfrak{z}_{0}\in U$ and for $n\geq0,\mathfrak{z}%
_{n+1}:=\Phi_{\eta}^{\mathfrak{t}\left(  \mathfrak{z}_{n}\right)  }\left(
\mathfrak{z}_{n}\right)  $ is a homogeneous Markov chain and now the system
evolution is given by the random process $\left(  \mathbf{u}_{t}%
,t\geq0\right)  $ such that, when started at $y_{0}\in U,\mathbf{u}_{t}\left(
y_{0}\right)  $ has the form (\ref{u_t}) with $\left\{  \left(  \mathfrak{x}%
_{n},\mathbf{t}_{n}\right)  \right\}  _{n\geq0}$ replaced by $\left\{  \left(
\mathfrak{z}_{n},\mathfrak{t}_{n}\right)  \right\}  _{n\geq0}.$\ Let now
$\sigma_{0}$ be the hitting time of $\mathcal{M}$ for $\left\{  \mathfrak{z}%
_{n}\right\}  _{n\geq0}.$ Under the reasonable assumptions on the renewal
process that, for any $t>0,\mathfrak{z}_{0}\in U,\mathbb{E}\left[
\mathfrak{n}_{t}|\mathfrak{z}_{0}\right]  <\infty$ and for any $\mathfrak{z}%
_{0}\notin\mathcal{M},\lim_{n\uparrow\infty}\mathbb{P}\left\{  \sigma
_{0}>\mathfrak{S}_{n}|\mathfrak{z}_{0}\right\}  =0,$ this case can be reduced
to the one treated in this article. Indeed, if $\sigma_{\mathcal{M}}$ is the
hitting time for $\left(  \mathbf{u}_{t},t\geq0\right)  $ of the cross-section
$\mathcal{M},$ since by definition $\mathbf{u}_{0}=\mathfrak{z}_{0}$ for any
$\mathbf{u}_{0}\in U,\mathbb{P}\left\{  \sigma_{\mathcal{M}}>\mathfrak{S}%
_{n}|\mathbf{u}_{0}\right\}  \leq\mathbb{P}\left\{  \sigma_{0}>\mathfrak{S}%
_{n}|\mathfrak{z}_{0}\right\}  .$ Hence, $\mathbb{P}\left\{  \sigma
_{\mathcal{M}}=\infty|\mathbf{u}_{0}\right\}  =\lim_{n\rightarrow\infty
}\mathbb{P}\left\{  \sigma_{\mathcal{M}}>\mathfrak{S}_{n}|\mathbf{u}%
_{0}\right\}  =0.$ Therefore we can analyze the trajectories of the system by
looking at the sequence of return times to $\mathcal{M},$ that is we can
reduce ourselves to study a random evolution of the kind given in (\ref{u_t}).

\section{Structure of the paper and results}

The paper is divided into four parts.

The first part, together with the introduction, contains the notations used
throughout the paper as well as the definition of the unperturbed dynamical
system and of its perturbation for given realizations of the noise.

In the second part we set up the problem of the stochastic stability of the
classical Lorenz flow under the stochastic perturbation scheme just descrided
in the framework of RDS. In order to simplify the exposition, which contains
many technical details and requires the introduction of several quantities, we
will list here the main steps we will go through to get to the proof deferring
the reader to the next sections for a detailed and precise description.

We consider a Poincar\'{e} section $\mathcal{M}$ for the unperturbed flow
$\left(  \Phi_{0}^{t},t\geq0\right)  $ associated to the smooth vector field
$\phi_{0}.$ This cross-section is transverse to the flows generated by smooth
perturbation $\phi_{\eta}$ of the original vector field if $\eta$ is chosen at
random in $\left[  -\varepsilon,\varepsilon\right]  $ according to some
probability measure $\lambda_{\varepsilon}$ for sufficiently small
$\varepsilon.$

\begin{itemize}
\item[\textbf{Step 1}] For any $\eta\in\left[  -\varepsilon,\varepsilon
\right]  ,$ the perturbed phase field $\phi_{\eta}$ is such that the
associated flows $\left(  \Phi_{\eta}^{t},t\geq0\right)  $ admit a $C^{1}$
stable foliation in a neighborhood of the corresponding attractor. In order to
study the RDS defined by the composition of the maps $R_{\eta}:=\Phi_{\eta
}^{\tau_{\eta}}:\mathcal{M}\circlearrowleft,$ with $\tau_{\eta}:\mathcal{M}%
\circlearrowleft$ the return time map on $\mathcal{M}$ for $\left(  \Phi
_{\eta}^{t},t\geq0\right)  ,$ we show that we can restrict ourselves to study
a RDS given by the composition of maps $\bar{R}_{\eta}:\mathcal{M}%
\circlearrowleft,$ conjugated to the maps $R_{\eta}$ via a diffeomorphism
$\kappa_{\eta}:\mathcal{M}\circlearrowleft,$ leaving invariant the unperturbed
stable foliation for any realization of the noise. Namely, we can reduce the
cross-section to a unit square foliated by vertical stable leaves, as for the
geometric Lorenz flow. By collapsing these leaves on their base points via the
diffeomorphism $q,$ we conjugate the first return map $\bar{R}_{\eta}$ on
$\mathcal{M}$ to a piecewise map $\bar{T}_{\eta}$ of the interval $I.$ This
one-dimensional quotient map is expanding with the first derivative blowing up
to infinity at some point.

\item[\textbf{Step 2}] We introduce the random perturbations of the
unperturbed quotient map $T_{0}.$ Suppose $\omega=(\eta_{0},\eta_{1}%
,\cdots,\eta_{k},\cdots)$ is a sequence of values in $\left[  -\varepsilon
,\varepsilon\right]  $ each chosen independently of the others according to
the probability $\lambda_{\varepsilon}.$ We construct the concatenation
$\bar{T}_{\eta_{k}}\circ\cdots\circ\bar{T}_{\eta_{0}}$ and prove that there
exists a stationary measure $\nu_{1}^{\varepsilon},$ i.e. such that for any
bounded measurable function $g$ and $k\geq0,\int g(\bar{T}_{\eta_{k}}%
\circ\cdots\circ\bar{T}_{\eta_{0}})(x)\nu_{1}^{\varepsilon}\left(  dx\right)
\lambda_{\varepsilon}^{\otimes k}(d\eta)=\int gd\nu_{1}^{\varepsilon}.$
Clearly, $\mu_{\mathbf{T}}^{\varepsilon}:=\nu_{1}^{\varepsilon}\otimes
\mathbb{P}_{\varepsilon},$ with $\mathbb{P}_{\varepsilon}$ the probability
measure on the i.i.d. random sequences $\omega,$ is an invariant measure for
the associated RDS (see (\ref{defTT})).

\item[\textbf{Step 3}] We lift the random process just defined to a Markov
process on the Poincar\'{e} surface $\mathcal{M}$ given by the sequences
$\bar{R}_{\eta_{k}}\circ\cdots\circ\bar{R}_{\eta_{0}}$ and show that the
stationary measure $\nu_{2}^{\varepsilon}$ for this process can be constructed
from $\nu_{1}^{\varepsilon}.$ We set $\mu_{\overline{\mathbf{R}}}%
^{\varepsilon}:=\bar{\nu}_{2}^{\varepsilon}\otimes\mathbb{P}_{\varepsilon}$
the corresponding invariant measure for the RDS (see (\ref{defRR})).

We remark that, by construction, the conjugation property linking $R_{\eta}$
with $\bar{R}_{\eta}$ lifts to the associated RDS's. This allows us to recover
from $\mu_{\overline{\mathbf{R}}}^{\varepsilon}$ the invariant measure
$\mu_{\mathbf{R}}^{\varepsilon}$ for the RDS generated by composing the
$R_{\eta}$'s.

\item[\textbf{Step 4}] Let $\mathbf{R}:\mathcal{M}\times\Omega\circlearrowleft
$ be the map defining the RDS corresponding to the compositions of the
realizations of $R_{\eta}$ (see (\ref{RK=KR})). We identify the set
\begin{equation}
(\mathcal{M}\times\Omega)_{\mathbf{t}}:=\{(x,\omega,s)\in\mathcal{M}%
\times\Omega\times\mathbb{R}^{+}:s\in\lbrack0,\mathbf{t}(x,\omega))\}\ ,
\end{equation}
where $\Omega:=\left[  -\varepsilon,\varepsilon\right]  ^{\mathbb{N}%
},\mathbf{t}(x,\omega):=\tau_{\pi(\omega)}(x)$ is the \emph{random roof
function} and $\pi(\omega):=\eta_{0}$ is the first coordinate of $\omega,$
with the set $\mathfrak{V}$ of equivalence classes of points $\left(
x,\omega,t\right)  $ in $\mathcal{M}\times\Omega\times\mathbb{R}^{+}$ such
that $t=s+\sum_{k=0}^{n-1}\mathbf{t}\left(  \mathbf{R}^{k}\left(
x,\omega\right)  \right)  $ for some $s\in\lbrack0,\mathbf{t}(x,\omega
)),n\geq1.$ Then, if $\hat{\pi}:\mathcal{M}\times\Omega\times\mathbb{R}%
^{+}\longrightarrow\mathfrak{V}$ is the canonical projection and, for any
$t>0,N_{t}:=\max\left\{  n\in\mathbb{Z}^{+}:\sum_{k=0}^{n-1}\mathbf{t}%
\circ\mathbf{R}^{k}\leq t\right\}  ,$ we define the \emph{random suspension
semi-flow}
\begin{equation}
(\mathcal{M}\times\Omega)_{\mathbf{t}}\ni\left(  x,\omega,s\right)
\longmapsto\mathbf{S}^{t}(x,\omega,s):=\hat{\pi}(\mathbf{R}^{N_{s+t}}\left(
x,\omega\right)  ,s+t)\in(\mathcal{M}\times\Omega)_{\mathbf{t}}\ .
\end{equation}
In particular, for instance, if $\mathbf{s}_{2}(x,\omega)=\tau_{\eta_{0}%
}(x)+\tau_{\eta_{1}}(R_{\eta_{1}}(x))\leq s+t,$ we have
\begin{equation}
\mathbf{S}^{t}(x,\omega,s)=((R_{\eta_{1}}\circ R_{\eta_{0}}(x)),\theta
^{2}\omega,s+t-\mathbf{s}_{2}(x,\omega))\ ,
\end{equation}
where $\theta:\Omega\ni\omega=(\eta_{0},\eta_{1},\cdots,\eta_{k}%
,\cdots)\longmapsto\theta\omega:=(\eta_{1},\eta_{2},\cdots,\eta_{k+1}%
,\cdots)\in\Omega$ is the left shift.

\item[\textbf{Step 5}] We build up a conjugation between the random suspension
semi-flow and a semi-flow on $U\times\Omega,$ which we will call $\left(
X^{t},t\geq0\right)  ,$ such that its projection on $U$ is a representation of
(\ref{u_t}). The rough idea is that each time the orbit crosses the
Poincar\'{e} section $\mathcal{M},$ the vector fields will change randomly.
Therefore, we start by fixing the \emph{initial condition} $(y,\omega)$ with
$y\in U$ yet not necessarily on $\mathcal{M}.$ We now begin to define the
\emph{random flow} $\left(  X^{t},t\geq0\right)  .$ Let $\pi:\Omega
\mapsto\left[  -\varepsilon,\varepsilon\right]  $ be the projection of
$\omega=(\eta_{0},\eta_{1},\cdots,\eta_{k},\cdots)$ onto the first coordinate
and call $t_{\eta_{0}}\left(  y\right)  =t_{\pi\left(  \omega\right)  }\left(
y\right)  $ the time the orbit $\Phi_{\eta_{0}}^{t}(y)=\Phi_{\pi\left(
\omega\right)  }^{t}(y)$ takes to meet $\mathcal{M}$ and set $y_{1}%
:=\Phi_{\eta_{0}}^{t_{\eta_{0}}\left(  y\right)  }(y)=\Phi_{\pi\left(
\omega\right)  }^{t_{\pi\left(  \omega\right)  }\left(  y\right)  }(y).$Then,
since $\forall\omega\in\Omega,n\geq0,\pi\left(  \theta^{n}\omega\right)
=\eta_{n},$%
\begin{align}
X^{t}(y,\omega)  &  :=\left(  \Phi_{\pi\left(  \omega\right)  }^{t}%
(y),\omega\right)  ,\ 0\leq t\leq t_{\eta_{0}}\left(  y\right)  \ ;\label{X_t}%
\\
X^{t}(y,\omega)  &  =\left(  \Phi_{\pi\left(  \theta\omega\right)  }%
^{t-t_{\pi\left(  \omega\right)  }\left(  y\right)  }(y_{1}),\theta
\omega\right)  ,\ t_{\eta_{0}}\left(  y\right)  <t\leq t_{\eta_{0}}\left(
y\right)  +\tau_{\eta_{1}}(y_{1})\ ;\nonumber\\
X^{t}(y,\omega)  &  =\left(  \Phi_{\pi\left(  \theta^{2}\omega\right)
}^{t-t_{\pi\left(  \omega\right)  }\left(  y\right)  -\tau_{\pi\left(
\theta\omega\right)  }(y_{1})}(R_{\pi\left(  \theta\omega\right)  }%
(y_{1})),\theta^{2}\omega\right)  ,\ t_{\eta_{0}}\left(  y\right)  +\tau
_{\eta_{1}}(y_{1})<t\leq t_{\eta_{0}}\left(  y\right)  +\tau_{\eta_{1}}%
(y_{1})+\tau_{\eta_{2}}(R_{\eta_{1}}(y_{1}))\ ,\nonumber
\end{align}
where $R_{\pi\left(  \theta\omega\right)  }(y_{1})=R_{\eta_{1}}(y_{1}),$ and
so on.

\item[\textbf{Step 6}] We are now ready to define the conjugation
$\mathbf{V}:\mathcal{M}\times\Omega\times\mathbb{R}^{+}\rightarrow
\mathbb{R}^{3}\times\Omega$ in the following way:
\begin{align}
\mathbf{V}(x,\omega,s)  &  =\left(  \Phi_{\pi\left(  \omega\right)  }%
^{s}(x),\omega\right)  ,\ x\in\mathcal{M};\ \omega=(\eta_{0},\eta_{1}%
,\cdots,\eta_{k},\cdots)\in\Omega;\ 0\leq s<\tau_{\eta_{0}}(x)\\
\mathbf{V}(x,\omega,s)  &  =\left(  \Phi_{\pi\left(  \theta\omega\right)
}^{s-\tau_{\pi\left(  \omega\right)  }(x)}(R_{\pi\left(  \omega\right)
}(x)),\theta\omega\right)  ;\ \tau_{\eta_{0}}(x)\leq s<\tau_{\eta_{0}}%
(x)+\tau_{\eta_{1}}(R_{\eta_{0}}(x))\ ,\nonumber
\end{align}
where $R_{\pi\left(  \omega\right)  }(x)=R_{\eta_{0}}(x),$ and so on. By
collecting the expressions given above it is not difficult to check that
$\left(  X^{t},t\geq0\right)  $ must satisfy the equation
\begin{equation}
\mathbf{V}\circ\mathbf{S}^{t}=X^{t}\circ\mathbf{V}\ .
\end{equation}
For instance, if $s+t<\tau_{\eta_{0}}(x),$ we have $X^{t}\circ\mathbf{V}%
(x,\omega,s)=\left(  X^{t}(\Phi_{\eta_{0}}^{s}(x)),\omega\right)  =\left(
\Phi_{\eta_{0}}^{t}(\Phi_{\eta_{0}}^{s}(x)),\omega\right)  =\left(  \Phi
_{\eta_{0}}^{s+t}(x),\omega\right)  ,$ while $\mathbf{V}\circ\mathbf{S}%
^{t}(x,\omega,s)=\mathbf{V}(x,\omega,s+t)=\left(  \Phi_{\eta_{0}}%
^{s+t}(x),\omega\right)  .$

\item[\textbf{Step 7}] We lift the measure $\mu_{\mathbf{R}}^{\varepsilon}$ on
the random suspension in order to get an invariant measure for $\left(
\mathbf{S}^{t},t\geq0\right)  .$ Under the assumption that the random roof
function $\mathbf{t}$ is $\mu_{\mathbf{R}}^{\varepsilon}$-summable, the
invariant measure $\mu_{\mathbf{S}}^{\varepsilon}$ for the random suspension
semi-flow acts on bounded real functions $f$ as
\begin{equation}
\int d\mu_{\mathbf{S}}^{\varepsilon}f=\left(  \int d\mu_{\mathbf{R}%
}^{\varepsilon}\mathbf{t}\right)  ^{-1}\int d\mu_{\mathbf{R}}^{\varepsilon
}\left(  \int_{0}^{\mathbf{t}}f\circ\mathbf{S}^{t}dt\right)  \ .
\end{equation}

The invariant measure for the random flow $\left(  X^{t},t\geq0\right)  $ will
then be push forward $\mu_{\mathbf{S}}^{\varepsilon}$ under the conjugacy
$\mathbf{V},$ i.e.%
\begin{equation}
\mu_{\mathbf{V}}^{\varepsilon}=\mu_{\mathbf{S}}^{\varepsilon}\circ
\mathbf{V}^{-1}\ .
\end{equation}

\item[\textbf{Step 8}] We show that the correspondence $\mu_{\mathbf{T}%
}^{\varepsilon}\longrightarrow\mu_{\mathbf{R}}^{\varepsilon}\longrightarrow
\mu_{\mathbf{V}}^{\varepsilon}$ is injective and so that the stochastic
stability of $T_{0}$ (which in fact we prove to hold in the $L^{1}\left(
I,dx\right)  $ topology) implies that of the physical measure $\mu_{0}$ of the
unperturbed flow. More precisely, we lift the evolutions defined by the
unperturbed maps $T_{0}$ and $R_{0},$ as well as that represented by the
unperturbed suspension semi-flow $\left(  S_{0}^{t},t\geq0\right)  ,$ to
evolutions defined respectively on $I\times\Omega,\mathcal{M}\times\Omega$ and
on $(\mathcal{M}\times\Omega)_{\tau_{0}}:=\{(x,\omega,s)\in\mathcal{M}%
\times\Omega\times\mathbb{R}^{+}:s\in\lbrack0,\tau_{0}(x))\}.$ By
construction, the invariant measures for these evolutions are $\mu_{T_{0}%
}\otimes\delta_{\bar{0}},\mu_{R_{0}}\otimes\delta_{\bar{0}},\mu_{S_{0}}%
\otimes\delta_{\bar{0}},$ where $\bar{0}$ denotes the sequence in $\Omega$
whose entries are all equal to $0,\delta_{\bar{0}}$ is the Dirac mass at
$\bar{0}$ and $\mu_{T_{0}},\mu_{R_{0}},\mu_{S_{0}}$ are respectively the
invariant measures for $T_{0},R_{0}$ and $S_{0}.$ Then, we prove the weak
convergence, as $\varepsilon\downarrow0,$ of $\mu_{\mathbf{T}}^{\varepsilon}$
to $\mu_{T_{0}}\otimes\delta_{\bar{0}}$ and consequently the weak convergence
of $\mu_{\mathbf{R}}^{\varepsilon}$ to $\mu_{T_{0}}\otimes\delta_{\bar{0}}.$
This will imply the weak convergence of $\mu_{\mathbf{S}}^{\varepsilon}$ to
$\mu_{S_{0}}\otimes\delta_{\bar{0}}$ and therefore the weak convergence of
$\mu_{\mathbf{V}}^{\varepsilon}$ to $\mu_{0}\otimes\delta_{\bar{0}}$ proving
Theorem \ref{main}.
\end{itemize}

In the third part we will take a more probabilistic point of view and
formulate the question about the stochastic stability for the unperturbed flow
in the framework of PDMP. More precisely, we will show that we can recover the
physical measure of the unperturbed flow as weak limit, as the intensity of
the perturbation vanishes, of the measure on the phase space of the system
obtained by looking at the law of large numbers for cumulative processes
defined as the integral over $\left[  0,t\right]  $ of functionals on the path
space of the stationary process representing the perturbed system's dynamics.
Therefore, we will be reduced to prove that the imbedded Markov chain driving
the random process that describes the evolution of the system is stationary,
that its stationary (invariant) measure is unique and that it will converge
weakly to the invariant measure of the unperturbed Poincar\'{e} map
corresponding to $\mathcal{M}.$ To prove existence and uniqueness of the
stationary initial distribution of a Markov chain with uncountable state space
is not an easy task in general (we refer the reader to \cite{MT} for an
account on this subject). To overcome this difficulty we will make use of the
skew-product structure of the first return maps $R_{\eta}$ as it will be
outlined more precisely in the next section. However, if the perturbation of
the phase velocity field is given by the addition to the unperturbed one of a
small constant term, namely $\phi_{\eta}:=\phi_{0}+\eta H,H\in\mathbb{S}^{2},$
the proof of the stochastic stability of invariant measure for the unperturbed
Poincar\'{e} map will follow a more direct strategy; we refer the reader to
Section \ref{CF}.

The fourth part of the paper contains an appendix where we give examples of
the Poincar\'{e} section $\mathcal{M}$ and therefore of the maps $R_{\eta}$
and $T_{\eta},$ as well as we take the chance to comment on some results
achieved in our previous paper \cite{GMPV} about the statistical stability of
the classical Lorenz flow which will be recalled along the present work.

\section{Notations}

If $\mathfrak{X}$ is a Borel space we denote by $\mathcal{B}\left(
\mathfrak{X}\right)  $ its Borel $\sigma$algebra and by $M_{b}\left(
\mathfrak{X}\right)  $ the Banach space of bounded $\mathcal{B}\left(
\mathfrak{X}\right)  $-measurable functions on $\mathfrak{X}$ equipped with
the uniform norm. Moreover, we denote by $\mathfrak{M}\left(  \mathfrak{X}%
\right)  $ the Banach space of finite Radon measures on $\left(
\mathfrak{X},\mathcal{B}\left(  \mathfrak{X}\right)  \right)  $ such that, for
any $\mu\in\mathfrak{M}\left(  \mathfrak{X}\right)  ,\left\Vert \mu\right\Vert
:=\sup_{g\in C\left(  \mathfrak{X}\right)  :\left\Vert g\right\Vert _{\infty
}=1}\left\vert \mu\left(  g\right)  \right\vert =\left\vert \mu\right\vert
\left(  \mathfrak{X}\right)  ,$ where $\left\vert \mu\right\vert :=\mu_{+}%
+\mu_{-}$ with $\mu_{\pm}$ the elements of the canonical decomposition of
$\mu.$ Furthermore, $\mathfrak{P}\left(  \mathfrak{X}\right)  $ denotes the
set of probability measures on $\left(  \mathfrak{X},\mathcal{B}\left(
\mathfrak{X}\right)  \right)  $ and, if $\mu\in\mathfrak{P}\left(
\mathfrak{X}\right)  ,spt\mu\subseteq\mathfrak{X}$ denotes its support.
Finally, if $\mu\in\mathfrak{M}\left(  \mathfrak{X}\right)  $ is positive, we
denote by $\hat{\mu}:=\frac{\mu}{\mu\left(  \mathfrak{X}\right)  }$ its
associated probability measure.

We denote by $\left\langle \cdot,\cdot\right\rangle $ the Euclidean scalar
product in $\mathbb{R}^{d},$ by $\left\Vert \cdot\right\Vert $ the associated
norm and by $\lambda^{d}$ the Lebesgue measure on $\mathbb{R}^{d}.$ We set
$\lambda^{1}:=\lambda.$

Let $\varepsilon>0$ and $\lambda_{\varepsilon}$ a probability measure on the
measurable space $\left(  \left[  -1,1\right]  ,\mathcal{B}\left(  \left[
-1,1\right]  \right)  \right)  $ such that in the limit of $\varepsilon$
tending to zero, $\lambda_{\varepsilon}$ weakly converges to the atomic mass
at $0.$

\subsection{Metric Dynamical System associated with the noise}

Consider the measurable space $\left(  \Omega,\mathcal{F}\right)  $ where
$\Omega:=\left[  -1,1\right]  ^{\mathbb{Z}^{+}},\mathcal{F}$ is the $\sigma
$algebra generated by the cylinder sets $\mathcal{C}_{n}\left(  A\right)
:=\left\{  \omega\in\Omega:\left(  \eta_{1},...,\eta_{n}\right)  \in
A\right\}  ,$ with $A\in\mathcal{B}\left(  \left[  -1,1\right]  ^{n}\right)
,n\geq1.$ In fact, we can consider $\Omega$ endowed with the metric
$\Omega\times\Omega\ni\left(  \omega_{1},\omega_{2}\right)  \longmapsto
\rho\left(  \omega_{1},\omega_{2}\right)  :=\sum_{n\geq1}\frac{1}{2^{n}}%
\frac{\left\vert \eta_{n}^{\left(  1\right)  }-\eta_{n}^{\left(  2\right)
}\right\vert }{1+\left\vert \eta_{n}^{\left(  1\right)  }-\eta_{n}^{\left(
2\right)  }\right\vert }\in\left[  0,1\right]  $ so that, denoting again by
$\Omega,$ with abuse of notation, the metric space $\left(  \Omega
,\rho\right)  ,\mathcal{F}$ coincides with $\mathcal{B}\left(  \Omega\right)
.$ If $\varrho$ is a probability measure on $\left(  \left[  -1,1\right]
,\mathcal{B}\left(  \left[  -1,1\right]  \right)  \right)  ,$ we denote by
$\mathbb{P}_{\varrho}$ the probability measure on $\left(  \Omega
,\mathcal{F}\right)  $ such that $\mathbb{P}_{\varrho}\left(  \mathcal{C}%
_{n}\left(  A\right)  \right)  :=\int_{A}\prod\limits_{i=0}^{n-1}%
\varrho\left(  d\eta_{i}\right)  $ and set $\mathbb{P}_{\varepsilon
}:=\mathbb{P}_{\lambda_{\varepsilon}}.$ In the following, to ease the
notation, we will omit to note the subscript denoting the dependence of the
probability distribution on $\left(  \Omega,\mathcal{F}\right)  $ from that on
$\left(  \left[  -1,1\right]  ,\mathcal{B}\left(  \left[  -1,1\right]
\right)  \right)  $ unless differently specified.

Let $\theta$ be the left shift operator on $\Omega.$ We denote by $\left(
\Omega,\mathcal{F},\theta,\mathbb{P}\right)  $ the corresponding metric
dynamical system. Moreover, we set
\begin{equation}
\Omega\ni\omega\longmapsto\pi\left(  \omega\right)  :=\eta_{1}\in
spt\lambda_{\varepsilon}\ .
\end{equation}

\subsection{Random Dynamical System\label{RDS}}

If $\Xi$ is a Polish space, let $\mathbb{M}\left(  \Xi\right)  $ the set of
the measurable maps $\vartheta:\Xi\circlearrowleft.$ We denote by
$\vartheta^{\#}$ the pull-back of $\vartheta$ (or Koopman operator), namely
$\vartheta^{\#}\varphi:=\varphi\circ\vartheta$ for any real valued measurable
function $\varphi$ on $\Xi,$ and by $\vartheta_{\#}$ the push-forward of
$\vartheta$ i.e. the corresponding transfer operator acting on $L^{1}\left(
\Xi\right)  $ being the adjoint of $\vartheta^{\#}$ considered as an operator
acting on $L^{\infty}\left(  \Xi\right)  .$

Given $\left\{  \vartheta_{\eta}\right\}  _{\eta\in spt\lambda_{\varepsilon}%
}\subset\mathbb{M}\left(  \Xi\right)  ,$ the skew product
\begin{equation}
\Xi\times\Omega\ni\left(  x,\omega\right)  \longmapsto\Theta\left(
x,\omega\right)  :=\left(  \vartheta_{\pi\left(  \omega\right)  },\theta
\omega\right)  \in\Xi\times\Omega\
\end{equation}
defines a random dynamical system (RDS) on $\left(  \Xi,\mathcal{B}\left(
\Xi\right)  \right)  $ over the metric dynamical system $\left(
\Omega,\mathcal{F},\theta,\mathbb{P}\right)  $ (see \cite{Ar} Section 1.1.1).
We set:

\begin{itemize}
\item $\mathfrak{P}_{\mathbb{P}}\left(  \Xi\times\Omega\right)  $ to be the
set of probability measures $\mu$ on $\left(  \Xi\times\Omega,\mathcal{B}%
\left(  \Xi\right)  \otimes\mathcal{F}\right)  $ with marginal $\mathbb{P}$ on
$\left(  \Omega,\mathcal{F}\right)  $ and denote by $\mu\left(  \cdot
|\omega\right)  :=\frac{d\mu\left(  \cdot,\omega\right)  }{d\mathbb{P}\left(
\omega\right)  };$

\item $\mathfrak{I}_{\mathbb{P}}\left(  \Theta\right)  :=\left\{  \mu
\in\mathfrak{P}_{\mathbb{P}}\left(  \Xi\times\Omega\right)  :\Theta_{\#}%
\mu=\mu\right\}  ;$
\end{itemize}

(see \cite{Ar} Definition 1.4.1). We also define
\begin{equation}
\Xi\times\Omega\ni\left(  x,\omega\right)  \longmapsto p\left(  x,\omega
\right)  :=x\in\Xi\ .
\end{equation}

\subsection{Path space representation of a stochastic process}

Let us denote by $\mathbb{D}\left(  \mathbb{R}^{+},\Xi\right)  $\ the Skorohod
space of $\Xi$-valued functions on $\mathbb{R}^{+}$ and by $\mathfrak{B}%
\left(  \Xi\right)  $ its Borel $\sigma$algebra. Then, $\forall t\in
\mathbb{R}^{+},$ the evaluation map $\mathbb{D}\left(  \mathbb{R}^{+}%
,\Xi\right)  \ni\mathbf{Y}\longmapsto\xi_{t}\left(  \mathbf{Y}\right)
:=Y_{t}\in\Xi$ is a random element on $\left(  \mathbb{D}\left(
\mathbb{R}^{+},\Xi\right)  ,\mathfrak{B}\left(  \Xi\right)  \right)  $ with
values in $\Xi.$ We also denote by $\mathbb{D}_{y}\left(  \mathbb{R}^{+}%
,\Xi\right)  $ the Skorohod space of $\Xi$-valued functions on $\mathbb{R}%
^{+}$ started at $y\in\Xi.$

Let $\left\{  \mathfrak{F}_{t}^{\xi}\right\}  _{t\geq0},$ such that, $\forall
t\geq0,\mathfrak{F}_{t}^{\xi}:=\bigvee\limits_{s\leq t}\xi_{t}^{-1}\left(
\mathcal{B}\left(  \Xi\right)  \right)  ,$ be the natural filtration
associated to the stochastic process $\left(  \xi_{t},t\geq0\right)  .$ Then,
since $\Xi$ is Polish it is separable and so $\lim_{t\rightarrow\infty
}\mathfrak{F}_{t}^{\xi}=\bigvee\limits_{t\geq0}\mathfrak{F}_{t}^{\xi
}=\mathfrak{B}\left(  \Xi\right)  .$

Given $y\in\Xi,$ if $\left(  \mathfrak{y}_{t},t\geq0\right)  $ is a $\Xi
$-valued random process on $\left(  \Omega,\mathcal{F},\mathbb{P}\right)  $
such that,\linebreak$\forall B\in\mathcal{B}\left(  \Xi\right)  ,\mathbb{P}%
\left\{  \omega\in\Omega:\mathfrak{y}_{0}\left(  \omega\right)  \in B\right\}
=\mathbf{1}_{B}\left(  y\right)  ,$ let $\mathcal{Y}_{y}$ be the
$\mathbb{D}\left(  \mathbb{R}^{+},\Xi\right)  $-valued random element on
$\left(  \Omega,\mathcal{F}\right)  $ such that, $\forall\omega\in\Omega
,t\geq0,\xi_{t}\left(  \mathcal{Y}_{y}\left(  \omega\right)  \right)
=\mathfrak{y}_{t}\left(  \omega\right)  .$ We then set $\mathbb{Q}%
_{y}^{\mathfrak{y}}:=\mathbb{P\circ}\mathcal{Y}_{y}^{-1}.$ If $\Xi\ni
y\longmapsto\mathbb{Q}_{y}^{\mathfrak{y}}\in\mathfrak{P}\left(  \mathbb{D}%
\left(  \mathbb{R}^{+},\Xi\right)  \right)  $ is $\mathcal{B}\left(
\Xi\right)  $-measurable, it is a probability kernel from $\left(
\Xi,\mathcal{B}\left(  \Xi\right)  \right)  $ to $\left(  \mathbb{D}\left(
\mathbb{R}^{+},\Xi\right)  ,\mathfrak{B}\left(  \Xi\right)  \right)  $ such
that $\mathfrak{P}\left(  \Xi\right)  \ni\mu\longmapsto\mathbb{Q}_{\mu
}^{\mathfrak{y}}:=\mu\left(  \mathbb{Q}_{\cdot}^{\mathfrak{y}}\right)
\in\mathfrak{P}\left(  \mathbb{D}\left(  \mathbb{R}^{+},\Xi\right)  \right)
.$ Hence, denoting by $\mathfrak{F}_{t}^{\mathfrak{y}}\left(  \mu\right)  ,$
for any $t\geq0,$ the completion of $\mathfrak{F}_{t}^{\xi}$ with all the
$\mathbb{Q}_{\mu}^{\mathfrak{y}}$-null sets of $\mathfrak{B}\left(
\Xi\right)  ,$ we set $\mathfrak{F}_{t}^{\mathfrak{y}}:=\bigcap\limits_{\mu
\mathbb{\in}\mathfrak{P}\left(  \Xi\right)  }\mathfrak{F}_{t}^{\mathfrak{y}%
}\left(  \mu\right)  .$

If $\mathbb{Q}_{\cdot}^{\mathfrak{y}}$ is a probability kernel, $\forall
A\in\mathcal{F},$ the conditional probability $\mathbb{P}\left(
A|\mathfrak{y}_{0}\right)  $ admits a regular version which we denote by
$\mathbb{P}^{\mathfrak{y}}\left(  A|\cdot\right)  .$ Hence we set $\forall
t\geq0,\mathcal{F}_{t}^{\mathfrak{y}}:=\bigvee\limits_{s\leq t}\mathfrak{y}%
_{t}^{-1}\left(  \mathcal{B}\left(  \Xi\right)  \right)  ,$ denote by
$\mathcal{F}_{t}^{\mathfrak{y}}\left(  \mu\right)  $ the completion of
$\mathcal{F}_{t}^{\mathfrak{y}}$ with all the $\int_{\Xi}\mu\left(  dy\right)
\mathbb{P}^{\mathfrak{y}}\left(  \cdot|y\right)  $-null sets of $\mathcal{F}$
and set\linebreak$\overline{\mathcal{F}}_{t}^{\mathfrak{y}}:=\bigcap
\limits_{\mu\mathbb{\in}\mathfrak{P}\left(  \Xi\right)  }\mathcal{F}%
_{t}^{\mathfrak{y}}\left(  \mu\right)  .$

\section{The perturbed phase vector fields and the associated suspension
semiflows}

Given $\varepsilon>0$ sufficiently small, for any realization of the noise
$\eta\in spt\lambda_{\varepsilon},$ let $\phi_{\eta}$ be a phase field in
$\mathbb{R}^{3}$ and let $\left(  \Phi_{\eta}^{t},t\geq0\right)  $ be the
associated flow.

\subsection{The perturbed phase vector field $\phi_{\eta}$}

We assume that $\phi_{\eta}\in C^{r}\left(  \mathbb{R}^{3},\mathbb{R}%
^{3}\right)  $ for some $r\geq2$ independent of $\eta.$ In particular, we
denote by $\phi_{0}$ the Lorenz'63 vector field given in (\ref{L1}) and by
$\mathcal{M}$ be a Poincar\'{e} section for the associated flow $\left(
\Phi_{0}^{t},t\geq0\right)  .$

We further assume that, for any realization of the noise $\eta\in
spt\lambda_{\varepsilon},\phi_{\eta}$ belongs to a small neighborhood
$\mathfrak{U}$ of the unperturbed phase field $\phi_{0}$ in the $C^{1}$
topology such that there exists an open neighborhood $U$ in $\mathbb{R}^{3}$
containing the attractor $\Lambda$ of $\phi_{0}$ which also contains\linebreak%
$\Lambda_{\eta}:=\bigcap\limits_{t\geq0}\Phi_{\eta}^{t}\left(  U\right)  ,$
where the set $\Lambda_{\eta}$ is invariant for $\left(  \Phi_{\eta}^{t}%
,t\geq0\right)  ,$ is transitive and contains a hyperbolic singularity. We
choose $\mathfrak{U}$ small enough such that $\mathcal{M}$ is a Poincar\'{e}
section for any realization of the flow $\left(  \Phi_{\eta}^{t}%
,t\geq0\right)  $ (see e.g. \cite{HS} chapter 16, paragraph 2) and there
exists a stable foliation $\mathcal{I}_{\eta}$ of $\mathcal{M}$ that is at
least $C^{1+\epsilon},$ for some $\epsilon>0$ independent of $\eta,$ which can
be associated to the points of a transversal curve $I_{\eta}$ inside
$\mathcal{M}$ (see \cite{APPV} sections 5.2 and 5.3).

A good example for $\phi_{\eta}$ to keep in mind is
\begin{equation}
\phi_{\eta}:=\phi_{0}+\eta Hg_{\mathcal{M}}\ , \label{ex}%
\end{equation}
where $H\in\mathbb{S}^{2}$ and $g_{\mathcal{M}}$ is a sufficiently smooth
approximation of $\mathbf{1}_{\mathcal{M}}$ supported on $\mathcal{M}.$
Indeed, in this case, the existence and smoothness of the stable foliation can
be proved following the argument given in \cite{AM} section 4.

\subsection{The Poincar\'{e} map $R_{\eta}$}

Given $\eta\in spt\lambda_{\varepsilon},$ let $\Gamma_{\eta}$ be the leaf of
the invariant foliation of $\mathcal{M}$ corresponding to points $x$ whose
orbit $\left(  \Phi_{\eta}^{t}\left(  x\right)  ,t>0\right)  $ falls into the
local stable manifold of the hyperbolic singularity of $\phi_{\eta}.$ Then
\begin{equation}
\mathcal{M}\backslash\Gamma_{\eta}\ni x\longmapsto\tau_{\eta}\left(  x\right)
\in\mathbb{R}^{+}%
\end{equation}
is the return time map on $\mathcal{M}$ for $\left(  \Phi_{\eta}^{t}%
,t\geq0\right)  $ and
\begin{equation}
\mathcal{M}\backslash\Gamma_{\eta}\ni x\longmapsto R_{\eta}\left(  x\right)
:=\Phi_{\eta}^{\tau_{\eta}\left(  x\right)  }\left(  x\right)  \in\mathcal{M}
\label{defR}%
\end{equation}
is the Poincar\'{e} return map on $\mathcal{M}.$

Identifying $\mathcal{I}_{\eta}$ with $I_{\eta},$ let%

\begin{equation}
\mathcal{M}\ni x\longmapsto u:=q_{\eta}\left(  x\right)  \in I_{\eta}%
\end{equation}
be the canonical projection along the leaves of the foliation $\mathcal{I}%
_{\eta}.$ The assumption we made on $\phi_{\eta}$ imply that $\mathcal{I}%
_{\eta}$ is invariant and contracting, which means that there exists a map
$T_{\eta}:I_{\eta}^{\prime}\longrightarrow I_{\eta},$ with $I_{\eta}^{\prime
}\subseteq I_{\eta},$ such that \ for any $x$ in the domain of $R_{\eta}$
\begin{equation}
q_{\eta}\circ R_{\eta}\left(  x\right)  =T_{\eta}\circ q_{\eta}\left(
x\right)  \label{defT}%
\end{equation}
and if $u\in I_{\eta}$ is in the domain of $T_{\eta}$ the diameter of
$R_{\eta}^{n}\left(  q_{\eta}^{-1}\left(  u\right)  \right)  $ tends to zero
as $n$ tends to infinity.

\subsubsection{The conjugated map $\bar{R}_{\eta}$}

Since for any $\eta\in spt\lambda_{\varepsilon}$ the leaves of the stable
foliation $\mathcal{I}_{\eta}$ of $\mathcal{M}$ are rectifiables, arguing as
in \cite{APPV} sections 5.2 and 5.3 (see also Remark 3.15 in \cite{AP} and
\cite{AM}) we can construct two $C^{1}$ diffeomorphisms $\kappa_{\eta
}:\mathcal{M}\circlearrowleft$ and $\iota_{\eta}:\mathcal{I}_{\eta
}\longrightarrow\mathcal{I}:=\mathcal{I}_{0},$ such that
\begin{equation}
\iota_{\eta}\circ q_{\eta}=q\circ\kappa_{\eta}\ , \label{iq=qk}%
\end{equation}
where $q:=q_{0}$ (see Figure 1).

As a consequence, we can define $\bar{T}_{\eta}:I\circlearrowleft,$ where
$I:=I_{0},$ such that
\begin{equation}
\bar{T}_{\eta}\circ q\circ\kappa_{\eta}=\iota_{\eta}\circ T_{\eta}\circ
q_{\eta}%
\end{equation}
which, by (\ref{iq=qk}) implies
\begin{equation}
\bar{T}_{\eta}\circ\iota_{\eta}=\iota_{\eta}\circ T_{\eta}\ . \label{Ti=iT}%
\end{equation}
Defining $\bar{R}_{\eta}:\mathcal{M}\circlearrowleft$ such that
\begin{equation}
\bar{R}_{\eta}\circ\kappa_{\eta}=\kappa_{\eta}\circ R_{\eta}\ , \label{Rk=kR}%
\end{equation}
we get
\begin{equation}
\bar{T}_{\eta}\circ q=q\circ\bar{R}_{\eta}\ . \label{defT'}%
\end{equation}
We remark that the diffeomorfism $q$ does not depend on $\eta$ anymore.

Since
\begin{align}
\bar{T}_{\eta}\circ q\circ\kappa_{\eta}  &  =\bar{T}_{\eta}\circ\iota_{\eta
}\circ q_{\eta}=\iota_{\eta}\circ T_{\eta}\circ q_{\eta}\\
&  =\iota_{\eta}\circ q_{\eta}\circ R_{\eta}=q\circ\kappa_{\eta}\circ R_{\eta
}=q\circ\bar{R}_{\eta}\circ\kappa_{\eta}\ .\nonumber
\end{align}
Therefore, $\forall u\in I_{\eta},$ since $R_{\eta}\left(  q_{\eta}%
^{-1}\left(  u\right)  \right)  \subset q_{\eta}^{-1}\left(  T_{\eta}\left(
u\right)  \right)  ,$ by (\ref{iq=qk}), (\ref{Ti=iT}), (\ref{Rk=kR}) and
(\ref{defT'}) we obtain
\begin{equation}
\kappa_{\eta}^{-1}\circ\bar{R}_{\eta}\circ\kappa_{\eta}\left(  \kappa_{\eta
}^{-1}\circ q^{-1}\circ\iota_{\eta}\left(  u\right)  \right)  \subset
\kappa_{\eta}^{-1}\circ q^{-1}\circ\iota_{\eta}\left(  \iota_{\eta}^{-1}%
\circ\bar{T}_{\eta}\circ\iota_{\eta}\left(  u\right)  \right)  \ ,
\end{equation}
that is
\begin{equation}
\kappa_{\eta}^{-1}\circ\bar{R}_{\eta}\left(  q^{-1}\circ\iota_{\eta}\left(
u\right)  \right)  \subset\kappa_{\eta}^{-1}\circ q^{-1}\left(  \bar{T}_{\eta
}\circ\iota_{\eta}\left(  u\right)  \right)  \ ,
\end{equation}
which, because by definition $\kappa_{\eta}$ maps a leaf of the foliation
$\mathcal{I}_{\eta}$ to a leaf of the foliation $\mathcal{I},$ implies
\begin{equation}
\bar{R}_{\eta}\circ q^{-1}\left(  \iota_{\eta}\left(  u\right)  \right)
\subset q^{-1}\left(  \bar{T}_{\eta}\circ\iota_{\eta}\left(  u\right)
\right)
\end{equation}
and so, $\forall u\in I,$%
\begin{equation}
\bar{R}_{\eta}\circ q^{-1}\left(  u\right)  \subset q^{-1}\left(  \bar
{T}_{\eta}\left(  u\right)  \right)  \ . \label{ass1}%
\end{equation}

\begin{figure}[ptbh]
\centering
\resizebox{0.75\textwidth}{!}{\includegraphics{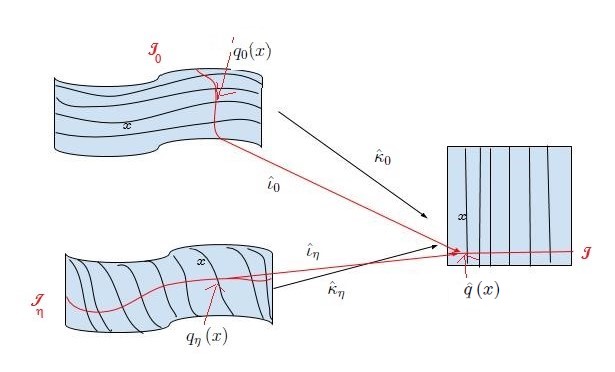}}
%If not, use
%\vspace{5cm}       % Give the correct figure height in cm
%Give a unique label
\caption{$\hat{\iota}_{0} \circ q_{0} = \hat{q} \circ\hat{\kappa}_{0}\, ,\,
\hat{\iota}_{\eta} \circ q_{\eta} = \hat{q} \circ\hat{\kappa}_{\eta}.$
Therefore, $\iota_{\eta} := \hat{\iota}_{\eta} \circ\hat{\iota}_{0}^{-1}\, ,\,
\kappa_{\eta} := \hat{\kappa}_{\eta} \circ\hat{\kappa}_{0}^{-1}$ implies
$\iota_{\eta} \circ q_{\eta} = q_{0} \circ\kappa_{\eta} .$}%
\label{fig:1}%
\end{figure}

\subsection{The suspension semi-flow}

Let us set
\begin{equation}
\mathcal{M}\backslash\Gamma_{\eta}\ni x\longmapsto\sigma_{\eta}^{n}\left(
x\right)  :=\sum_{k=0}^{n-1}\tau_{\eta}\left(  R_{\eta}^{k}\left(  x\right)
\right)  \in\mathbb{R}^{+}\ ,\;n\geq1\ ,
\end{equation}
and, $\forall x\in\mathcal{M}\backslash\Gamma_{\eta},$%
\begin{equation}
\mathbb{R}^{+}\ni t\longmapsto n_{\eta}\left(  x,t\right)  :=\max\left\{
n\in\mathbb{Z}^{+}:\sigma_{\eta}^{n}\left(  x\right)  \leq t\right\}
\in\mathbb{Z}^{+}\ .
\end{equation}
If
\begin{equation}
\mathcal{M}_{\tau_{\eta}}:=\left\{  \left(  x,s\right)  \in\mathcal{M}%
\times\mathbb{R}^{+}:s\in\lbrack0,\tau_{\eta}\left(  x\right)  )\right\}
\subset\mathbb{R}^{3}\ ,
\end{equation}
we define the suspension semiflow $\left(  S_{\eta}^{t},t\geq0\right)  $ as
\begin{equation}
\mathcal{M}_{\tau_{\eta}}\ni\left(  x,s\right)  \longmapsto S_{\eta}%
^{t}\left(  x,s\right)  :=\left(  R_{\eta}^{n_{\eta}\left(  x,t+s\right)
}\left(  x\right)  ,t+s-\sigma_{\eta}^{n_{\eta}\left(  x,s+t\right)  }\left(
x\right)  \right)  \in\mathcal{M}_{\tau_{\eta}}\ ,\;t\geq0\ .
\label{ssf_omega}%
\end{equation}

Let $\sim_{\eta}$ be the equivalence relation on $\mathcal{M}\times
\mathbb{R}^{+}$ such that any two points $\left(  x,s\right)  ,\left(
y,t\right)  $ in $\mathcal{M}\times\mathbb{R}^{+}$ belong to the same
equivalence class if there exist $\left(  x_{0},s_{0}\right)  \in
\mathcal{M}_{\tau_{\eta}},s^{\prime},s^{\prime\prime}>0$ such that $\Phi
_{\eta,\tau_{\eta}}^{s^{\prime}}\left(  x_{0},s_{0}\right)  =\left(
x,s\right)  ,\Phi_{\eta,\tau_{\eta}}^{s^{\prime\prime}}\left(  x_{0}%
,s_{0}\right)  =\left(  y,t\right)  $ and $n_{\eta}\left(  x_{0}%
,s^{\prime\prime}\vee s^{\prime}+s_{0}\right)  -n_{\eta}\left(  x_{0}%
,s^{\prime\prime}\wedge s^{\prime}+s_{0}\right)  \in\mathbb{N}.$ We denote by
$\mathcal{V}_{\eta}:=\mathcal{M}\times\mathbb{R}^{+}/\sim_{\eta}$ the
corresponding quotient space and by $\tilde{\pi}_{\eta}:\mathcal{M}%
\times\mathbb{R}^{+}\longrightarrow\mathcal{V}_{\eta}$ the canonical
projection which induces a topology and consequently a Borel $\sigma$algebra
on $\mathcal{V}_{\eta}.$ Therefore,
\begin{equation}
\mathcal{M}\times\mathbb{R}^{+}\ni\left(  x,s\right)  \longmapsto S_{\eta}%
^{t}\circ\tilde{\pi}_{\eta}\left(  x,s\right)  =\tilde{\pi}_{\eta}\left(
x,s+t\right)  \in\mathcal{V}_{\eta}\ ,\;t>0\ . \label{ssf_omega-1}%
\end{equation}

Let us define $\bar{\tau}_{\eta}:\mathcal{M}\backslash\Gamma_{0}%
\longrightarrow\mathbb{R}^{+}$ such that
\begin{equation}
\bar{\tau}_{\eta}\circ\kappa_{\eta}=\tau_{\eta}\ ,
\end{equation}
and consequently
\begin{equation}
\mathcal{M}_{\bar{\tau}_{\eta}}:=\left\{  \left(  x,s\right)  \in
\mathcal{M}\times\mathbb{R}^{+}:s\in\lbrack0,\bar{\tau}_{\eta}\left(
x\right)  )\right\}  \subset\mathbb{R}^{3}\ .
\end{equation}
Setting $\bar{\sigma}_{\eta}^{n},n\in\mathbb{Z}^{+},$ and $\bar{n}_{\eta}$
such that
\begin{equation}
\bar{\sigma}_{\eta}^{n}\circ\kappa_{\eta}=\sigma_{\eta}^{n}\;;\;\bar{n}_{\eta
}\circ\kappa_{\eta}=n_{\eta}%
\end{equation}
and
\begin{equation}
\mathcal{M}_{\bar{\tau}_{\eta}}\ni\left(  x,s\right)  \longmapsto\overline
{S}_{\eta}^{t}\left(  x,s\right)  :=\left(  \bar{R}_{\eta}^{\bar{n}_{\eta
}\left(  x,t+s\right)  }\left(  x\right)  ,t+s-\bar{\sigma}_{\eta}^{\bar
{n}_{\eta}\left(  x,s+t\right)  }\left(  x\right)  \right)  \in\mathcal{M}%
_{\bar{\tau}_{\eta}}\ ,\;t\geq0\ , \label{ssf_omega-2}%
\end{equation}
we can lift of the diffeomorphism $\kappa_{\eta}$ defined in (\ref{iq=qk}) to
the diffeomorphism
\begin{equation}
\mathcal{M}_{\tau_{\eta}}\ni\left(  x,s\right)  \longmapsto\bar{\kappa}_{\eta
}\left(  x,s\right)  :=\left(  \kappa_{\eta}\left(  x\right)  ,\frac{\bar
{\tau}_{\eta}\circ\kappa_{\eta}\left(  x\right)  }{\tau_{\eta}\left(
x\right)  }s\right)  =\left(  \kappa_{\eta}\left(  x\right)  ,s\right)
\in\mathcal{M}_{\bar{\tau}_{\eta}}\ , \label{kappabar}%
\end{equation}
so that, by (\ref{Rk=kR}),
\begin{equation}
\bar{\kappa}_{\eta}\circ S_{\eta}^{t}=\overline{S}_{\eta}^{t}\circ\bar{\kappa
}_{\eta}\ .
\end{equation}
Let $\approx_{\eta}$ to be the equivalence relation on $\mathcal{M}%
\times\mathbb{R}^{+}$ such that any two points $\left(  x,s\right)  ,\left(
y,t\right)  $ in $\mathcal{M}\times\mathbb{R}^{+}$ belong to the same
equivalence class if there exist $\left(  x_{0},s_{0}\right)  \in
\mathcal{M}_{\bar{\tau}_{\eta}},s^{\prime},s^{\prime\prime}>0$ such that
$\bar{\Phi}_{\eta,\bar{\tau}_{\eta}}^{s^{\prime}}\left(  x_{0},s_{0}\right)
=\left(  x,s\right)  ,\bar{\Phi}_{\eta,\bar{\tau}_{\eta}}^{s^{\prime\prime}%
}\left(  x_{0},s_{0}\right)  =\left(  y,t\right)  $ and $\bar{n}_{\eta}\left(
x_{0},s^{\prime\prime}\vee s^{\prime}+s_{0}\right)  -\bar{n}_{\eta}\left(
x_{0},s^{\prime\prime}\wedge s^{\prime}+s_{0}\right)  \in\mathbb{N}.$ Denoting
by $\overline{\mathcal{V}}_{\eta}:=\mathcal{M}\times\mathbb{R}^{+}%
/\approx_{\eta}$ the corresponding quotient space and by $\breve{\pi}_{\eta
}:\mathcal{M}\times\mathbb{R}^{+}\longrightarrow\overline{\mathcal{V}}_{\eta}$
the canonical projection such that
\begin{equation}
\mathcal{M}\times\mathbb{R}^{+}\ni\left(  x,s\right)  \longmapsto\overline
{S}_{\eta}^{t}\circ\breve{\pi}_{\eta}\left(  x,s\right)  =\breve{\pi}_{\eta
}\left(  x,s+t\right)  \in\overline{\mathcal{V}}_{\eta}\ ,\;t>0
\label{ssf_omega-3}%
\end{equation}
by (\ref{kappabar}) we can define a diffeomorphism $\tilde{\kappa}_{\eta
}:\mathcal{V}_{\eta}\longrightarrow\overline{\mathcal{V}}_{\eta}$ such that
\begin{equation}
\tilde{\kappa}_{\eta}\circ\tilde{\pi}_{\eta}=\breve{\pi}_{\eta}\circ
\tilde{\kappa}_{\eta}\ . \label{kpi=pik}%
\end{equation}

\part{Stochastic stability for impulsive type forcing}

As already anticipated in the introduction, in this section we will study the
weak convergence of the invariant measure of the semi-Markov random evolution
describing the random perturbations of $\left(  \Phi_{0}^{t},t\geq0\right)  $
in a neighborhood of the unperturbed attractor to the unperturbed physical measure.

To this purpose we will first consider the RDS defined by the composition of
the maps $\bar{R}_{\eta}$ given in (\ref{Rk=kR}) which, by construction,
preserve the unperturbed invariant foliation. Then, we give an explicit
representation for the invariant measure of the original process in terms of
the invariant measure for this auxiliary process which, in turn, can be
defined starting from the invariant measure for the RDS defined by the
composition of the maps $\bar{T}_{\eta}.$

Finally, we will prove that the stochastic stability of the unperturbed
physical measure follows from the stochastic stability of the invariant
measure for the one-dimensional dynamical system defined by the map $T_{0}.$

\section{The associated Random Dynamical System\label{LRDS}}

In this section we present the construction of the auxiliary random processes
needed to build up a representation of the random evolution given in
(\ref{u_t}) in the framework of RDSs. We refer the reader to \cite{Ar} Section
1.1.1 for an account on the definition of a RDS in a more general setup.

\subsection{Random maps}

\begin{enumerate}
\item
\begin{equation}
I\times\Omega\ni\left(  u,\omega\right)  \longmapsto\mathbf{T}\left(
u,\omega\right)  :=\left(  \bar{T}_{\pi\left(  \omega\right)  }\left(
u\right)  ,\theta\omega\right)  \in I\times\Omega\ , \label{defTT}%
\end{equation}
with $\mathbf{T}^{0}$ the identity operator on $I\times\Omega,$ defines a
measurable random dynamical system on $\left(  I,\mathcal{B}\left(  I\right)
\right)  $ over the metric dynamical system $\left(  \Omega,\mathcal{F}%
,\mathbb{P},\theta\right)  ;$

\item setting $\mathcal{\tilde{M}}:=\mathcal{M}\backslash\Gamma_{0},$%
\begin{equation}
\mathcal{\tilde{M}}\times\Omega\ni\left(  x,\omega\right)  \longmapsto
\overline{\mathbf{R}}\left(  x,\omega\right)  \in\left(  \bar{R}_{\pi\left(
\omega\right)  }\left(  x\right)  ,\theta\omega\right)  \in\mathcal{M}%
\times\Omega\ , \label{defRR}%
\end{equation}
with $\overline{\mathbf{R}}^{0}$ the identity operator on $\mathcal{M}%
\times\Omega,$ define two measurable random dynamical systems on $\left(
\mathcal{M},\mathcal{B}\left(  \mathcal{M}\right)  \right)  $ over the metric
dynamical system $\left(  \Omega,\mathcal{F},\mathbb{P},\theta\right)  .$

Let
\begin{equation}
\mathcal{M}\times\Omega\ni\left(  x,\omega\right)  \longmapsto Q\left(
x,\omega\right)  :=\left(  q\left(  x\right)  ,\omega\right)  \in
I\times\Omega\ . \label{defQ}%
\end{equation}
Then, $\forall\left(  x,\omega\right)  \in\mathcal{\tilde{M}}\times\Omega,$%
\begin{align}
\left(  Q\circ\overline{\mathbf{R}}\right)  \left(  x,\omega\right)   &
=Q\left(  \bar{R}_{\pi\left(  \omega\right)  }\left(  x\right)  ,\theta
\omega\right)  =\left(  q\left(  \bar{R}_{\pi\left(  \omega\right)  }\left(
x\right)  \right)  ,\theta\omega\right) \\
&  =\left(  \bar{T}_{\pi\left(  \omega\right)  }\left(  q\left(  x\right)
\right)  ,\theta\omega\right)  =\left(  \mathbf{T}\circ Q\right)  \left(
x,\omega\right) \nonumber
\end{align}
that is
\begin{equation}
Q\circ\overline{\mathbf{R}}=\mathbf{T}\circ Q\ . \label{QR=TQ}%
\end{equation}

Defining the map%
\begin{equation}
\mathcal{M}\times\Omega\ni\left(  x,\omega\right)  \longmapsto\mathbf{K}%
\left(  x,\omega\right)  :=\left(  \kappa_{\pi\left(  \omega\right)  }\left(
x\right)  ,\omega\right)  \in\mathcal{M}\times\Omega\ , \label{defK}%
\end{equation}
for any $\left(  x,\omega\right)  \in\widetilde{\mathcal{M}\times\Omega
}:=\left(  \mathcal{M}\times\Omega\right)  \backslash\left\{  \left(
x,\omega\right)  \in\mathcal{M}\times\Omega:x\in\Gamma_{\pi\left(
\omega\right)  }\right\}  ,$ we define $\mathbf{R}:\widetilde{\mathcal{M}%
\times\Omega}\longrightarrow\mathcal{M}\times\Omega$ such that
\begin{equation}
\overline{\mathbf{R}}\circ\mathbf{K}\left(  x,\omega\right)  =\mathbf{K}%
\left(  x,\omega\right)  \circ\mathbf{R}\ , \label{RK=KR}%
\end{equation}
that is
\begin{equation}
\widetilde{\mathcal{M}\times\Omega}\ni\left(  x,\omega\right)  \longmapsto
\left(  \bar{R}_{\pi\left(  \omega\right)  }\left(  x\right)  \circ\kappa
_{\pi\left(  \omega\right)  },\theta\omega\right)  =\left(  \kappa_{\pi\left(
\omega\right)  }\circ R_{\pi\left(  \omega\right)  }\left(  x\right)
,\theta\omega\right)  \in\mathcal{M}\times\Omega\ .
\end{equation}

\end{enumerate}

\subsection{The random suspension semi-flow\label{RSSF}}

Let
\begin{equation}
\mathcal{M}\times\Omega\ni\left(  x,\omega\right)  \longmapsto\mathbf{t}%
\left(  x,\omega\right)  :=\tau_{\pi\left(  \omega\right)  }\left(  x\right)
\in\overline{\mathbb{R}^{+}}\ . \label{def_tSS}%
\end{equation}

Then, $\forall n\geq1,$ we define
\begin{equation}
\mathcal{M}\times\Omega\ni\left(  x,\omega\right)  \longmapsto\mathbf{s}%
_{n}\left(  x,\omega\right)  :=\sum_{k=0}^{n-1}\mathbf{t}\left(
\mathbf{R}^{k}\left(  x,\omega\right)  \right)  \in\overline{\mathbb{R}^{+}%
}\ ,\;n\geq1\ , \label{sSS}%
\end{equation}
and denote, $\forall t>0,$%
\begin{equation}
\mathcal{M}\times\Omega\ni\left(  x,\omega\right)  \longmapsto N_{t}\left(
x,\omega\right)  :=\max\left\{  n\in\mathbb{Z}^{+}:\mathbf{s}_{n}\left(
x,\omega\right)  \leq t\right\}  \in\mathbb{Z}^{+}\mathbb{\ }.
\end{equation}
We now proceed as in the definition of standard suspension flow given in
(\ref{ssf_omega}). We define
\begin{equation}
\left(  \mathcal{M}\times\Omega\right)  _{\mathbf{t}}:=\left\{  \left(
x,\omega,s\right)  \in\widetilde{\mathcal{M}\times\Omega}\times\mathbb{R}%
^{+}:s\in\lbrack0,\mathbf{t}\left(  x,\omega\right)  )\right\}
\label{MOmegat}%
\end{equation}
and consequently the semiflow $\left(  \mathbf{S}^{t},t\geq0\right)  ,$ which
we will call \emph{random suspension semi-flow,} where
\begin{equation}
\left(  \mathcal{M}\times\Omega\right)  _{\mathbf{t}}\ni\left(  x,\omega
,s\right)  \longmapsto\mathbf{S}^{t}\left(  x,\omega,s\right)  :=\left(
\mathbf{R}^{N_{s+t}\left(  x,\omega\right)  }\left(  x,\omega\right)
,s+t-\mathbf{s}_{N_{s+t}\left(  x,\omega\right)  }\left(  x,\omega\right)
\right)  \in\left(  \mathcal{M}\times\Omega\right)  _{\mathbf{t}}\ .
\label{defSS}%
\end{equation}

Let $\sim$ be the equivalence relation on $\mathcal{M}\times\Omega
\times\mathbb{R}^{+}$ such that any two points $\left(  x,\omega,s\right)
,\left(  y,\omega^{\prime},t\right)  $ in $\mathcal{M}\times\Omega
\times\mathbb{R}^{+}$ belong to the same equivalence class if there exist
$\left(  x_{0},\omega_{0},s_{0}\right)  \in\left(  \mathcal{M}\times
\Omega\right)  _{\mathbf{t}}$ and $t^{\prime},t^{\prime\prime}>0$ such that
$\mathbf{S}^{t^{\prime}}\left(  x_{0},\omega_{0},s_{0}\right)  =\left(
x,\omega,s\right)  ,\mathbf{S}^{t^{\prime\prime}}\left(  x_{0},\omega
_{0},s_{0}\right)  =\left(  y,\omega^{\prime},t\right)  $ and $N_{t^{\prime
\prime}\vee t^{\prime}+s_{0}}\left(  x_{0},\omega_{0}\right)  -N_{t^{\prime
\prime}\wedge t^{\prime}+s_{0}}\left(  x_{0},\omega_{0}\right)  \in
\mathbb{N}.$ We denote by $\mathfrak{V}:=\mathcal{M}\times\Omega
\times\mathbb{R}^{+}/\sim$ the corresponding quotient space and by $\hat{\pi
}:\mathcal{M}\times\Omega\times\mathbb{R}^{+}\longrightarrow\mathfrak{V}$ the
canonical projection which induces a topology and consequently a Borel
$\sigma$algebra on $\mathfrak{V}.$ Therefore,
\begin{equation}
\mathcal{M}\times\Omega\times\mathbb{R}^{+}\ni\left(  x,\omega,s\right)
\longmapsto\mathbf{S}^{t}\circ\hat{\pi}\left(  x,\omega,s\right)  =\hat{\pi
}\left(  x,\omega,s+t\right)  \in\mathfrak{V}\ ,\;t>0\ . \label{defSS1}%
\end{equation}

Let us define $\mathbf{\bar{t}}:\widetilde{\mathcal{M}\times\Omega
}\longrightarrow\mathbb{R}^{+}$ such that
\begin{equation}
\mathbf{\bar{t}}\circ\mathbf{K}=\mathbf{t} \label{tK=t}%
\end{equation}
and consequently
\begin{equation}
\left(  \mathcal{M}\times\Omega\right)  _{\mathbf{\bar{t}}}:=\left\{  \left(
x,\omega,s\right)  \in\mathcal{M}\times\Omega\times\mathbb{R}^{+}:s\in
\lbrack0,\mathbf{\bar{t}}\left(  x,\omega\right)  )\right\}  \ .
\end{equation}
Setting $\mathbf{\bar{s}}_{n},n\in\mathbb{N}$ and $\overline{N}$ such that
\begin{equation}
\mathbf{\bar{s}}_{n}\circ\mathbf{K=s}_{n}\;;\;\overline{N}\circ\mathbf{K}=N
\end{equation}
and
\begin{equation}
\left(  \mathcal{M}\times\Omega\right)  _{\mathbf{t}}\ni\left(  x,\omega
,s\right)  \longmapsto\overline{\mathbf{S}}^{t}\left(  x,\omega,s\right)
:=\left(  \overline{\mathbf{R}}^{\overline{N}_{s+t}\left(  x,\omega\right)
}\left(  x,\omega\right)  ,s+t-\mathbf{\bar{s}}_{\overline{N}_{s+t}\left(
x,\omega\right)  }\left(  x,\omega\right)  \right)  \in\left(  \mathcal{M}%
\times\Omega\right)  _{\mathbf{\bar{t}}}\ , \label{defSSbar}%
\end{equation}
we can lift the map defined in (\ref{defK}), as we did to get (\ref{kappabar}%
), to the map
\begin{equation}
\left(  \mathcal{M}\times\Omega\right)  _{\mathbf{t}}\ni\left(  x,\omega
,s\right)  \longmapsto\overline{\mathbf{K}}\left(  x,\omega,s\right)
:=\left(  \mathbf{K}\left(  x,\omega\right)  ,s\right)  \in\left(
\mathcal{M}\times\Omega\right)  _{\mathbf{\bar{t}}} \label{Kbar}%
\end{equation}
so that
\begin{equation}
\overline{\mathbf{K}}\circ\mathbf{S}^{t}=\overline{\mathbf{S}}^{t}%
\circ\overline{\mathbf{K}}\ . \label{SK=KS}%
\end{equation}
Let $\approx$ be the equivalence relation on $\mathcal{M}\times\Omega
\times\mathbb{R}^{+}$ such that any two points $\left(  x,\omega,s\right)
,\left(  y,\omega^{\prime},t\right)  $ in $\mathcal{M}\times\Omega
\times\mathbb{R}^{+}$ belong to the same equivalence class if there exist
$\left(  x_{0},\omega_{0},s_{0}\right)  \in\left(  \mathcal{M}\times
\Omega\right)  _{\mathbf{\bar{t}}}$ and $t^{\prime},t^{\prime\prime}>0$ such
that $\overline{\mathbf{S}}^{t^{\prime}}\left(  x_{0},\omega_{0},s_{0}\right)
=\left(  x,\omega,s\right)  ,\overline{\mathbf{S}}^{t^{\prime\prime}}\left(
x_{0},\omega_{0},s_{0}\right)  =\left(  y,\omega^{\prime},t\right)  $ and
$\overline{N}_{t^{\prime\prime}\vee t^{\prime}+s_{0}}\left(  x_{0},\omega
_{0}\right)  -\overline{N}_{t^{\prime\prime}\wedge t^{\prime}+s_{0}}\left(
x_{0},\omega_{0}\right)  \in\mathbb{N}.$ We denote by $\overline{\mathfrak{V}%
}:=\mathcal{M}\times\Omega\times\mathbb{R}^{+}/\approx$ the corresponding
quotient space and by $\check{\pi}:\mathcal{M}\times\Omega\times\mathbb{R}%
^{+}\longrightarrow\overline{\mathfrak{V}}$ the canonical projection such
that
\begin{equation}
\mathcal{M}\times\Omega\times\mathbb{R}^{+}\ni\left(  x,\omega,s\right)
\longmapsto\overline{\mathbf{S}}^{t}\circ\check{\pi}\left(  x,\omega,s\right)
=\check{\pi}\left(  x,\omega,s+t\right)  \in\overline{\mathfrak{V}}\ ,\;t>0\ ,
\label{defSS2}%
\end{equation}
by (\ref{Kbar}) we can define a map $\mathbf{\tilde{K}}:\mathfrak{V}%
\longrightarrow\overline{\mathfrak{V}}$ such that
\begin{equation}
\mathbf{\tilde{K}}\circ\hat{\pi}=\check{\pi}\circ\mathbf{\tilde{K}\ .}
\label{Kpi=piK}%
\end{equation}

\section{The invariant measures}

\subsection{The invariant measure for the RDS's $\overline{\mathbf{R}}$ and
$\mathbf{R}$ on $\left(  \mathcal{M},\mathcal{B}\left(  \mathcal{M}\right)
\right)  $}

Let us assume $\mu_{\mathbf{T}}\in\mathfrak{I}_{\mathbb{P}}\left(
\mathbf{T}\right)  $ to be the invariant measure for $\mathbf{T.}$

The results in \cite{AP} Section 7.3.4.1 applies almost verbatim to
$\mathbf{T}$ and $\overline{\mathbf{R}}$ (see in particular Lemma 7.21 and
Corollary 7.22). Hence the proof of the following result is deferred to the appendix.

\begin{proposition}
\label{Prop1}Let $\mu_{\mathbf{T}}$ be the invariant measure for $\mathbf{T}.$
There exists a measure $\mu_{\overline{\mathbf{R}}}$ on \linebreak$\left(
\mathcal{M}\times\Omega,\mathcal{B}\left(  \mathcal{M}\right)  \otimes
\mathcal{F}\right)  ,$ invariant under $\overline{\mathbf{R}},$ such that,
$\forall\psi\in L_{\mathbb{P}}^{1}\left(  \Omega,C_{b}\left(  \mathcal{M}%
\right)  \right)  ,$%
\begin{equation}
\mu_{\overline{\mathbf{R}}}\left(  \psi\right)  :=\lim_{n\rightarrow\infty
}\int\mu_{\mathbf{T}}\left(  du,d\omega\right)  \inf_{x\in q^{-1}\left(
u\right)  }\psi\circ\overline{\mathbf{R}}^{n}\left(  x,\omega\right)
\label{mu_R}%
\end{equation}
and the correspondence $\mu_{\mathbf{T}}\longmapsto\mu_{\overline{\mathbf{R}}%
}$ is injective. Moreover, if $\mu_{\mathbf{T}}$ is ergodic, then
$\mu_{\overline{\mathbf{R}}}$ is also ergodic.
\end{proposition}

\begin{remark}
\label{Rem1}If $\mu_{\mathbf{T}}\in\mathfrak{I}_{\mathbb{P}}\left(
\mathbf{T}\right)  $ then $\mu_{\overline{\mathbf{R}}}\in\mathfrak{I}%
_{\mathbb{P}}\left(  \overline{\mathbf{R}}\right)  $ and, by \cite{Ar}
Proposition 1.4.3, the correspondence $\mu_{\mathbf{T}}\left(  \cdot
|\omega\right)  \longmapsto\mu_{\overline{\mathbf{R}}}\left(  \cdot
|\omega\right)  $ is injective.

Moreover, if $\mu_{\mathbf{T}}$ admits the disintegration $\mu_{\mathbf{T}%
}\left(  du,d\omega\right)  =\nu_{1}\left(  du\right)  \mathbb{P}\left(
d\omega\right)  ,$ by \cite{Ar} Theorem 2.1.7, $\nu_{1}$ is the stationary
measure for the Markov chain with transition operator
\begin{equation}
C_{b}\left(  I\right)  \ni\varphi\longmapsto P_{T}\varphi:=\mathbb{E}\left[
\varphi\circ\mathbf{q}\circ\mathbf{T}\right]  \in M_{b}\left(  I\right)  \ ,
\end{equation}
where
\begin{equation}
I\times\Omega\ni\left(  u,\omega\right)  \longmapsto\mathbf{q}\left(
u,\omega\right)  :=u\in I\ . \label{defqq}%
\end{equation}
Therefore, there exists a stationary measure $\mu_{\overline{\mathbf{R}}}$ for
the Markov chain with transition operator
\begin{equation}
C_{b}\left(  \mathcal{M}\right)  \ni\psi\longmapsto P_{\overline{R}}%
\psi:=\mathbb{E}\left[  \psi\circ p\circ\overline{\mathbf{R}}\right]  \in
M_{b}\left(  \mathcal{M}\right)  \ , \label{P_Rbar}%
\end{equation}
such that $\mu_{\overline{\mathbf{R}}}=\bar{\nu}_{2}\otimes\mathbb{P}.$
Indeed, by (\ref{mu_R}),
\begin{align}
\bar{\nu}_{2}\left(  \psi\right)   &  =\lim_{n\longrightarrow\infty}\int
\nu_{1}\left(  du\right)  \mathbb{E}\left[  \inf_{x\in q^{-1}\left(  u\right)
}\left[  \psi\circ p\circ\overline{\mathbf{R}}^{n}\right]  \left(
x,\cdot\right)  \right] \\
&  =\lim_{n\longrightarrow\infty}\int\nu_{1}\left(  du\right)  \inf_{x\in
q^{-1}\left(  u\right)  }\left(  P_{\overline{R}}^{n}\psi\right)  \left(
x\right) \nonumber
\end{align}
and, by (\ref{invmu_R})\footnote{By (\ref{P_Rbar}),
\begin{align*}
\left(  P_{\overline{R}}^{2}\psi\right)  \left(  x\right)   &  =\mathbb{E}%
\left[  \left(  P_{\overline{R}}\psi\right)  \circ p\circ\overline{\mathbf{R}%
}\right]  \left(  x\right)  =\mathbb{E}\left[  \mathbb{E}\left[  \left(
\psi\circ p\circ\overline{\mathbf{R}}\right)  \circ p\circ\overline
{\mathbf{R}}\right]  \right] \\
&  =\int d\mathbb{P}\left(  \omega\right)  \int d\mathbb{P}\left(
\omega^{\prime}\right)  \left(  \psi\circ p\right)  \left(  \bar{R}%
_{\pi\left(  \omega^{\prime}\right)  }\circ\bar{R}_{\pi\left(  \omega\right)
}x,\theta\omega^{\prime}\right) \\
&  =\int d\mathbb{P}\left(  \theta\omega\right)  \left(  \psi\circ p\right)
\left(  \bar{R}_{\pi\left(  \theta\omega\right)  }\circ\bar{R}_{\pi\left(
\omega\right)  }x,\theta^{2}\omega\right) \\
&  =\mathbb{E}\left[  \psi\circ p\circ\overline{\mathbf{R}}^{2}\right]  .
\end{align*}
},
\begin{align}
\bar{\nu}_{2}\left(  P_{\overline{R}}\psi\right)   &  =\lim_{n\longrightarrow
\infty}\int\nu_{1}\left(  du\right)  \inf_{x\in q^{-1}\left(  u\right)
}\left(  P_{\overline{R}}^{n+1}\psi\right)  \left(  x\right) \\
&  =\lim_{n\longrightarrow\infty}\int\nu_{1}\left(  du\right)  \mathbb{E}%
\left[  \inf_{x\in q^{-1}\left(  u\right)  }\left[  \psi\circ p\circ
\overline{\mathbf{R}}^{n+1}\right]  \left(  x,\cdot\right)  \right]  =\bar
{\nu}_{2}\left(  \psi\right)  \ .\nonumber
\end{align}
Moreover, for any $\varphi\in C_{b}\left(  I\right)  ,\varphi\circ q\in
C_{b}\left(  \mathcal{M}\right)  ;$ thus, by (\ref{QR=TQ}),
\begin{align}
\bar{\nu}_{2}\left(  \varphi\circ q\right)   &  =\lim_{n\longrightarrow\infty
}\int\nu_{1}\left(  du\right)  \mathbb{E}\left[  \inf_{x\in q^{-1}\left(
u\right)  }\left[  \varphi\circ q\circ p\circ\overline{\mathbf{R}}^{n}\right]
\left(  x,\cdot\right)  \right] \\
&  =\lim_{n\longrightarrow\infty}\int\nu_{1}\left(  du\right)  \mathbb{E}%
\left[  \inf_{x\in q^{-1}\left(  u\right)  }\left[  \varphi\circ
\mathbf{q}\circ Q\circ\overline{\mathbf{R}}^{n}\right]  \left(  x,\cdot
\right)  \right] \nonumber\\
&  =\lim_{n\longrightarrow\infty}\int\nu_{1}\left(  du\right)  \mathbb{E}%
\left[  \inf_{x\in q^{-1}\left(  u\right)  }\left[  \varphi\circ
\mathbf{q}\circ\mathbf{T}^{n}\circ Q\right]  \left(  x,\cdot\right)  \right]
\nonumber\\
&  =\lim_{n\longrightarrow\infty}\int\nu_{1}\left(  du\right)  \mathbb{E}%
\left[  \left[  \varphi\circ\mathbf{q}\circ\mathbf{T}^{n}\right]  \left(
u,\cdot\right)  \right] \nonumber\\
&  =\lim_{n\longrightarrow\infty}\int\nu_{1}\left(  du\right)  P_{T}%
^{n}\varphi\left(  u\right)  =\nu_{1}\left[  \varphi\right]  \ .\nonumber
\end{align}
Since $\mathcal{B}_{I}:=q^{-1}\left(  \mathcal{B}\left(  I\right)  \right)  $
is a sub-$\sigma$algebra of $\mathcal{B}\left(  \mathcal{M}\right)  $ and
since $\bar{\nu}_{2}\left(  \varphi|\mathcal{B}_{I}\right)  $ is constant on
the leaves of the invariant foliation, we get $\bar{\nu}_{2}\left(
\varphi\right)  =\bar{\nu}_{2}\left(  \bar{\nu}_{2}\left(  \varphi
|\mathcal{B}_{I}\right)  \right)  =\nu_{1}\left[  \bar{\nu}_{2}\left(
\varphi|\mathcal{B}_{I}\right)  \right]  .$ Hence, since by definition
$\forall u\in I,\omega\in\Omega,$%
\begin{equation}
\lim_{n\rightarrow\infty}\operatorname*{diam}p\left(  \overline{\mathbf{R}%
}^{n}\left(  Q^{-1}\left(  u,\omega\right)  \right)  \right)  =0\ ,
\end{equation}
$\bar{\nu}_{2}$ is singular w.r.t. the Lebesgue measure on $\left(
\mathcal{M},\mathcal{B}\left(  \mathcal{M}\right)  \right)  ,$ while the
marginal of $\bar{\nu}_{2}$ on $\left(  I,\mathcal{B}\left(  I\right)
\right)  $ coincides with $\nu_{1}.$
\end{remark}

\begin{corollary}
If $\mu_{\overline{\mathbf{R}}}\in\mathfrak{I}_{\mathbb{P}}\left(
\overline{\mathbf{R}}\right)  $ then $\mu_{\mathbf{R}}:=\mathbf{K}_{\#}%
^{-1}\mu_{\overline{\mathbf{R}}}=\mu_{\overline{\mathbf{R}}}\circ\mathbf{K}%
\in\mathfrak{I}_{\mathbb{P}}\left(  \mathbf{R}\right)  ,$ with, by
(\ref{RK=KR}),
\begin{equation}
\mathcal{M}\times\Omega\ni\left(  x,\omega\right)  \longmapsto\mathbf{K}%
^{-1}\left(  x,\omega\right)  :=\left(  \kappa_{\pi\left(  \omega\right)
}^{-1}\left(  x\right)  ,\omega\right)  \in\mathcal{M}\times\Omega\ .
\end{equation}

\end{corollary}

\begin{proof}
By (\ref{RK=KR}), for any $A\in\mathcal{B}\left(  \mathcal{M}\right)
\otimes\mathcal{F}$ we get
\begin{align}
\mu_{\mathbf{R}}\left(  \mathbf{R}^{-1}\left(  A\right)  \right)   &
=\mu_{\overline{\mathbf{R}}}\circ\mathbf{K}\left(  \mathbf{R}^{-1}\left(
A\right)  \right)  =\mu_{\overline{\mathbf{R}}}\circ\mathbf{K}\left(  \left(
\mathbf{R}^{-1}\circ\mathbf{K}^{-1}\right)  \left(  \mathbf{K}\left(
A\right)  \right)  \right) \\
&  =\mu_{\overline{\mathbf{R}}}\circ\mathbf{K}\left(  \left(  \mathbf{K}%
^{-1}\circ\overline{\mathbf{R}}^{-1}\right)  \left(  \mathbf{K}\left(
A\right)  \right)  \right)  =\mu_{\overline{\mathbf{R}}}\left(  \overline
{\mathbf{R}}^{-1}\left(  \mathbf{K}\left(  A\right)  \right)  \right)
=\mu_{\overline{\mathbf{R}}}\circ\mathbf{K}\left(  A\right)  \ .\nonumber
\end{align}

\end{proof}

\subsection{The invariant measure for the random semi-flow $\left(
\mathbf{S}^{t},t\geq0\right)  $}

Lemmata 7.28 and 7.29 as well as Corollary 7.30 in Section 7.3.6 of \cite{AP}
applies verbatim to the semi-flow (\ref{defSSbar}). We summarize these
statements in the following Lemma.

\begin{lemma}
Assume that the return time $\mathbf{\bar{t}}$ in (\ref{def_tSS}) is bounded
away from zero and integrable w.r.t. $\mu_{\overline{\mathbf{R}}}.$ Then the
measure on $\left(  \overline{\mathfrak{V}},\mathcal{B}\left(  \overline
{\mathfrak{V}}\right)  \right)  $ such that, for any bounded measurable
function $f:\overline{\mathfrak{V}}\longrightarrow\mathbb{R},$%
\begin{equation}
\mu_{\overline{\mathbf{S}}}\left(  f\right)  :=\frac{1}{\mu_{\overline
{\mathbf{R}}}\left(  \mathbf{\bar{t}}\right)  }\int\mu_{\overline{\mathbf{R}}%
}\left(  dx,d\omega\right)  \int_{0}^{\mathbf{\bar{t}}\left(  x,\omega\right)
}dtf\circ\check{\pi}\left(  x,\omega,t\right)  \label{mu_Sbar}%
\end{equation}
is invariant under the semi-flow defined by (\ref{defSS2}) on $\overline
{\mathfrak{V}}.$

Moreover, the correspondence $\mu_{\overline{\mathbf{R}}}\longmapsto
\mu_{\overline{\mathbf{S}}}$ (and so $\mu_{\mathbf{T}}\longmapsto
\mu_{\overline{\mathbf{R}}}\longmapsto\mu_{\overline{\mathbf{S}}}$) is injective.

Furthermore, if $\mu_{\overline{\mathbf{R}}}$ is invariant under
$\overline{\mathbf{R}},$ then
\begin{equation}
\lim_{T\longrightarrow\infty}\frac{1}{T}\int_{0}^{T}dtf\circ\check{\pi}\left(
x,\omega,t\right)  =\mu_{\overline{\mathbf{S}}}\left(  f\right)  \ .
\end{equation}
As a byproduct, if $\mu_{\overline{\mathbf{R}}}$ is ergodic $\mu
_{\overline{\mathbf{S}}}$ is also ergodic.
\end{lemma}

\begin{proof}
The proof of the invariance of $\mu_{\overline{\mathbf{S}}}$ under $\left(
\overline{\mathbf{S}}^{t},t\geq0\right)  $ on $\overline{\mathfrak{V}}$
follows word by word that of Lemma 7.28 in Section 7.3.6 of \cite{AP}. The
injectivity of the correspondence $\mu_{\overline{\mathbf{R}}}\longmapsto
\mu_{\overline{\mathbf{S}}}$ follows from that of the correspondence
$\psi\longmapsto f$ associating to any bounded measurable function
$\psi:\mathcal{M}\times\Omega\longrightarrow\mathbb{R}$ the bounded measurable
function
\begin{equation}
\mathfrak{V}\ni\left(  x,\omega,t\right)  \longmapsto f\left(  x,\omega
,t\right)  :=\mu_{\overline{\mathbf{R}}}\left(  \mathbf{\bar{t}}\right)
\frac{\psi\left(  x,\omega\right)  }{\mathbf{\bar{t}}\left(  x,\omega\right)
}\mathbf{1}_{[0,\mathbf{\bar{t}}\left(  x,\omega\right)  )}\left(  t\right)
\in\mathbb{R} \label{f_V}%
\end{equation}
such that $\mu_{\overline{\mathbf{S}}}\left(  f\right)  =\mu_{\overline
{\mathbf{R}}}\left(  \psi\right)  .$ The proof of the last result as well as
ergodicity of $\mu_{\overline{\mathbf{S}}}$ under the hypothesis of ergodicity
of $\mu_{\overline{\mathbf{R}}}$ are identical respectively to that of Lemma
7.28 and Corollary 7.30 in Section 7.3.6 of \cite{AP}.
\end{proof}

\begin{proposition}
Under the hypothesis of the preceding lemma, the measure on $\left(
\mathfrak{V},\mathcal{B}\left(  \mathfrak{V}\right)  \right)  $ such that, for
any bounded measurable function $f:\mathfrak{V}\longrightarrow\mathbb{R},$%
\begin{equation}
\mu_{\mathbf{S}}\left(  f\right)  :=\frac{1}{\mu_{\mathbf{R}}\left(
\mathbf{t}\right)  }\int\mu_{\mathbf{R}}\left(  dx,d\omega\right)  \int
_{0}^{\mathbf{t}\left(  x,\omega\right)  }dtf\circ\hat{\pi}\left(
x,\omega,t\right)  \label{mu_S}%
\end{equation}
is invariant under the semi-flow defined by (\ref{defSS1}) on $\mathfrak{V}.$

Moreover, the correspondence $\mu_{\mathbf{T}}\longmapsto\mu_{\mathbf{R}%
}\longmapsto\mu_{\mathbf{S}}$) is injective.

Furthermore, if $\mu_{\overline{\mathbf{R}}}$ is invariant under
$\overline{\mathbf{R}},$ then
\begin{equation}
\lim_{T\longrightarrow\infty}\frac{1}{T}\int_{0}^{T}dtf\circ\hat{\pi}\left(
x,\omega,t\right)  =\mu_{\mathbf{S}}\left(  f\right)  \ .
\end{equation}
As a byproduct, if $\mu_{\overline{\mathbf{R}}}$ is ergodic $\mu_{\mathbf{S}}$
is also ergodic.
\end{proposition}

\begin{proof}
If $\mathbf{t}\in L_{\mu_{\mathbf{R}}}^{1}$the proof of the invariance of
$\mu_{\mathbf{S}}$ under $\left(  \mathbf{S}^{t},t\geq0\right)  $ on
$\mathfrak{V}$ is identical to that given in the previous lemma. Moreover, the
proof of the ergodicity of $\mu_{\mathbf{S}}$ under the hypothesis of
ergodicity of $\mu_{\mathbf{R}}$ follows the same lines of that of the
corresponing statements involving $\mu_{\overline{\mathbf{S}}}$ and
$\mu_{\overline{\mathbf{R}}}$ in view of the previous corollary and the fact
that, by (\ref{tK=t}),
\begin{equation}
\mu_{\overline{\mathbf{R}}}\left(  \mathbf{\bar{t}}\right)  =\mathbf{K}%
_{\#}\mu_{\mathbf{R}}\left(  \mathbf{\bar{t}}\right)  =\mu_{\mathbf{R}}\left(
\mathbf{\bar{t}}\circ\mathbf{K}\right)  =\mu_{\mathbf{R}}\left(
\mathbf{t}\right)  \ ,
\end{equation}
which, by (\ref{Kpi=piK}), $\forall f:\overline{\mathfrak{V}}\longrightarrow
\mathbb{R},$ imply
\begin{align}
\mu_{\overline{\mathbf{S}}}\left(  f\right)   &  =\mathbf{\tilde{K}}_{\#}%
\mu_{\mathbf{S}}\left(  f\right)  =\mu_{\mathbf{S}}\left(  f\circ
\mathbf{\tilde{K}}\right)  =\frac{1}{\mu_{\mathbf{R}}\left(  \mathbf{t}%
\right)  }\mu_{\mathbf{R}}\otimes\lambda\left[  \mathbf{1}_{\left[
0,\mathbf{t}\right]  }f\circ\mathbf{\tilde{K}}\circ\hat{\pi}\right] \\
&  =\frac{1}{\mu_{\overline{\mathbf{R}}}\left(  \mathbf{\bar{t}}\right)  }%
\mu_{\mathbf{R}}\otimes\lambda\left[  \mathbf{1}_{\left[  0,\mathbf{\bar{t}%
}\circ\mathbf{K}\right]  }f\circ\check{\pi}\circ\mathbf{\tilde{K}}\right]
\nonumber\\
&  =\frac{1}{\mu_{\overline{\mathbf{R}}}\left(  \mathbf{\bar{t}}\right)  }%
\int\mu_{\mathbf{R}}\left(  dx,d\omega\right)  \int_{0}^{\left(
\mathbf{\bar{t}}\circ\mathbf{K}\right)  \left(  x,\omega\right)  }%
dtf\circ\check{\pi}\left(  \mathbf{K}\left(  x,\omega\right)  ,s\right)
\nonumber
\end{align}
i.e., since $\mu_{\overline{\mathbf{R}}}=\mathbf{K}_{\#}\mu_{\mathbf{R}},$ the
r.h.s. of (\ref{mu_Sbar}). Then, the injectivity of the correspondence
$\mu_{\mathbf{T}}\longmapsto\mu_{\mathbf{R}}\longmapsto\mu_{\mathbf{S}}$
readily follows.
\end{proof}

By the assumption we made on $\phi_{\eta},$ it has been proven in \cite{AMV}
Lemma 2.1 (see also \cite{HM} Proposition 2.6.) that there exists a positive
constant $C_{1}$ such that, for any $x\in\mathcal{M},$%
\begin{equation}
\bar{\tau}_{\eta}\left(  x\right)  \leq C_{1}\log\frac{1}{\left\vert q\left(
x\right)  -\hat{u}_{0}\right\vert }\ ,
\end{equation}
where $\hat{u}_{0}$ is the image under $q$ of the intersection of
$\mathcal{M}$ with the stable manifold of the hyperbolic fixed point. For
example, by what stated in Section \ref{assT}, $\hat{u}_{0}$ equal to $0$ if
$\mathcal{M}=\mathcal{M}^{\prime}$ or $\left\vert \hat{u}_{0}\right\vert
\in\left(  0,1\right)  $ if $\mathcal{M}=\mathcal{M}^{\prime\prime}.$ The
integrability of $\mathbf{\bar{t}}$ w.r.t. $\mu_{\overline{\mathbf{R}}}$ then
readily follows.

\begin{lemma}
If $\mu_{\mathbf{T}}$ is a.c. w.r.t. $\lambda\otimes\mathbb{P}_{\varepsilon}$
with density bounded $\lambda\otimes\mathbb{P}_{\varepsilon}$-a.s., then
$\mathbf{\bar{t}}$ is integrable w.r.t. $\mu_{\overline{\mathbf{R}}}.$
\end{lemma}

\begin{proof}
The proof is analogous to that of Lemma 3.7 in \cite{BR}. The sequence
$\left\{  \mathbf{\bar{t}}^{M}\right\}  _{M\in\mathbb{N}}$ such that
$\mathbf{\bar{t}}^{M}:=\mathbf{\bar{t}}\wedge M$ is monotone increasing an
converging $\mu_{\overline{\mathbf{R}}}$-a.s. to $\mathbf{\bar{t}.}$ So for
the monotone convergence theorem is enough to prove that $\mu_{\overline
{\mathbf{R}}}\left(  \mathbf{\bar{t}}^{M}\right)  $ is uniformly bounded in
$M.$ By (\ref{Prop1}),(\ref{def_tSS}) and (\ref{tK=t}) we get
\begin{align}
\mu_{\overline{\mathbf{R}}}\left(  \mathbf{\bar{t}}^{M}\right)   &
=\underline{\lim}_{n}\int\mu_{\mathbf{T}}\left(  du,d\omega\right)  \sup_{x\in
q^{-1}\left(  u\right)  }\mathbf{\bar{t}}^{M}\circ\overline{\mathbf{R}}%
^{n}\left(  x,\omega\right)  =\lim_{n\rightarrow\infty}\int\mu_{\mathbf{T}%
}\left(  du,d\omega\right)  \sup_{\left(  x,\omega^{\prime}\right)  \in
Q^{-1}\left(  u,\omega\right)  }\mathbf{\bar{t}}^{M}\circ\overline{\mathbf{R}%
}^{n}\left(  x,\omega^{\prime}\right) \\
&  \leq\lim_{n\rightarrow\infty}\int\mu_{\mathbf{T}}\left(  du,d\omega\right)
\sup_{\left(  x,\omega^{\prime}\right)  \in\left\{  \left(  y,\omega
^{\prime\prime}\right)  \in\mathcal{M}\times\Omega\ :\ Q\left(  y,\omega
^{\prime\prime}\right)  =\mathbf{T}^{n}\left(  u,\omega\right)  \right\}
}\mathbf{\bar{t}}^{M}\left(  x,\omega^{\prime}\right) \nonumber\\
&  =\int\mu_{\mathbf{T}}\left(  du,d\omega\right)  \sup_{\left(
x,\omega^{\prime}\right)  \in Q^{-1}\left(  u,\omega\right)  }\mathbf{\bar{t}%
}^{M}\left(  x,\omega^{\prime}\right)  \leq\int\mu_{\mathbf{T}}\left(
du,d\omega\right)  \sup_{x\in q^{-1}\left(  u\right)  }\bar{\tau}_{\pi\left(
\omega\right)  }\left(  x\right)  \wedge M\nonumber\\
&  \leq\left\Vert \frac{d\mu_{\mathbf{T}}}{d\left(  \lambda\otimes
\mathbb{P}_{\varepsilon}\right)  }\right\Vert _{\infty}C_{1}\int_{I}%
du\log\left\vert u-\hat{u}_{0}\right\vert <\infty\ .\nonumber
\end{align}

\end{proof}

\section{Stochastic stability\label{SS}}

Given $\eta\in spt\lambda_{\varepsilon},$ let $\bar{\eta}\in\Omega$ be such
that $\forall m\geq0,\pi\left(  \theta^{m}\bar{\eta}\right)  =\eta.$

If $\mu_{\bar{T}_{\eta}}$ denotes the measure on $\left(  I,\mathcal{B}\left(
I\right)  \right)  $ invariant under the dynamics defined by the map $\bar
{T}_{\eta}$ given in (\ref{Ti=iT}), we can lift the metric dynamical system
$\left(  I,\mathcal{B}\left(  I\right)  ,\mu_{T_{\eta}},\bar{T}_{\eta}\right)
$ to the metric dynamical system $\left(  I\times\Omega,\mathcal{B}\left(
I\right)  \otimes\mathcal{F},\mu_{\mathbf{T}_{\eta}},\mathbf{T}_{\eta}\right)
,$ where
\begin{equation}
I\times\Omega\ni\left(  u,\omega\right)  \longmapsto\mathbf{T}_{\eta}\left(
u,\omega\right)  :=\left(  \bar{T}_{\eta}\left(  u\right)  ,\theta
\omega\right)  \in I\times\Omega
\end{equation}
and $\mu_{\mathbf{T}_{\eta}}:=\mu_{\bar{T}_{\eta}}\otimes\delta_{\bar{\eta}},$
with $\delta_{\bar{\eta}}$ the Dirac mass at $\bar{\eta}.$

In the same fashion, denoting by $\mu_{R_{\eta}}$ the measure on $\left(
\mathcal{M},\mathcal{B}\left(  \mathcal{M}\right)  \right)  $ invariant under
the dynamics defined by the map $R_{\eta}$ given in (\ref{defR}), we define
the metric dynamical system \linebreak$\left(  \mathcal{M}\times
\Omega,\mathcal{B}\left(  \mathcal{M}\right)  \otimes\mathcal{F}%
,\mu_{\mathbf{R}_{\eta}},\mathbf{R}_{\eta}\right)  ,$ where
\begin{equation}
\left(  \mathcal{M}\backslash\Gamma_{\eta}\right)  \times\Omega\ni\left(
x,\omega\right)  \longmapsto\mathbf{R}_{\eta}\left(  x,\omega\right)
\in\left(  R_{\eta}\left(  x\right)  ,\theta\omega\right)  \in\mathcal{M}%
\times\Omega
\end{equation}
and $\mu_{\mathbf{R}_{\eta}}:=\mu_{R_{\eta}}\otimes\delta_{\bar{\eta}}.$

Moreover, setting
\begin{equation}
\left(  \mathcal{M}\backslash\Gamma_{\eta}\right)  \times\Omega\ni\left(
x,\omega\right)  \longmapsto\mathbf{t}_{\eta}\left(  x,\omega\right)
:=\mathbf{t}\left(  x,\bar{\eta}\right)  =\tau_{\eta}\left(  x\right)
\in\mathbb{R}^{+}\ , \label{deft_e}%
\end{equation}
we define semi-flow $\left(  \mathbf{S}_{\eta}^{t},t\geq0\right)  $ on
$\left(  \mathcal{M}\times\Omega\right)  _{\mathbf{t}_{\eta}}=\mathcal{M}%
_{\tau_{\eta}}\times\Omega$ as in (\ref{defSS}) and consequently, setting
\begin{equation}
\mathcal{M}\times\Omega\times\mathbb{R}^{+}\ni\left(  x,\omega,s\right)
\longmapsto\hat{\pi}_{\eta}\left(  x,\omega,s\right)  :=\left(  \tilde{\pi
}_{\eta}\left(  x,s\right)  ,\omega\right)  \in\mathcal{V}_{\eta}\times
\Omega\ , \label{Pp}%
\end{equation}
the semi-flow%
\begin{equation}
\mathcal{M}\times\Omega\times\mathbb{R}^{+}\ni\left(  x,\omega,s\right)
\longmapsto\mathbf{S}_{\eta}^{t}\circ\hat{\pi}_{\eta}\left(  x,\omega
,s\right)  =\hat{\pi}_{\eta}\left(  x,\omega,s+t\right)  \in\mathcal{V}_{\eta
}\times\Omega\ ,\;t>0\ ,
\end{equation}
as in (\ref{defSS1}). Furthermore, we denote by $\mu_{\mathbf{S}_{\eta}}%
:=\mu_{S_{\eta}}\otimes\delta_{\bar{\eta}},$ where $\mu_{S_{\eta}}\left(
dt,dx\right)  :=\frac{\mathbf{1}_{\left[  0,\tau_{\eta}\left(  x\right)
\right]  }\left(  t\right)  \mu_{R_{\eta}}\left(  dx\right)  }{\mu_{R_{\eta
}\left(  \tau_{\eta}\right)  }},$ the measure on $\left(  \mathcal{V}_{\eta
}\times\Omega,\mathcal{B}\left(  \mathcal{V}_{\eta}\right)  \otimes
\mathcal{F}\right)  $ invariant under $\left(  \mathbf{S}_{\eta}^{t}%
,t\geq0\right)  .$

Since, by the definition of $\lambda_{\varepsilon},$ as $\varepsilon$ tends to
$0,\mathbb{P}_{\mathbb{\varepsilon}}$ weakly converges to the Dirac mass
supported on the realization $\bar{0}\in\Omega$ whose components are all equal
to $0,$ in the following we make explicit the dependence of $\mu_{\mathbf{T}%
},\mu_{\mathbf{R}},\mu_{\mathbf{S}},$ on $\varepsilon,$ that is we set
$\mu_{\mathbf{T}}^{\varepsilon}:=\mu_{\mathbf{T}},\mu_{\mathbf{R}%
}^{\varepsilon}:=\mu_{\mathbf{R}},\mu_{\mathbf{S}}^{\varepsilon}%
:=\mu_{\mathbf{S}}.$

\begin{definition}
\label{defSS_TR}We will say that $\mu_{T_{0}},\mu_{R_{0}}$ are
\emph{stochastically stable} if, respectively, $\mu_{\mathbf{T}}^{\varepsilon
}$ weakly converges to $\mu_{\mathbf{T}_{0}},\mu_{\mathbf{R}}^{\varepsilon}$
weakly converges to $\mu_{\mathbf{R}_{0}},$ as $\varepsilon$ tends to $0.$
\end{definition}

\begin{remark}
\label{Rem2}We remark that the definition just given of stochastic stability
of $\mu_{T_{0}},\mu_{R_{0}}$ is weaker than the one usually taken into
consideration (see e.g. \cite{Vi}). Indeed, if $\mu_{\mathbf{T}}^{\varepsilon
}\in\mathfrak{I}_{\mathbb{P}_{\varepsilon}}\left(  \mathbf{T}\right)  $ admits
the disintegration $\nu_{1}^{\varepsilon}\otimes\mathbb{P}_{\varepsilon},$
which implies, by Remark \ref{Rem1},\ $\mu_{\overline{\mathbf{R}}%
}^{\varepsilon}=\bar{\nu}_{2}^{\varepsilon}\otimes\mathbb{P}_{\varepsilon},$
and $\mu_{\mathbf{R}}^{\varepsilon}\in\mathfrak{I}_{\mathbb{P}_{\varepsilon}%
}\left(  \mathbf{R}\right)  $ admits the disintegration $\nu_{2}^{\varepsilon
}\otimes\mathbb{P}_{\varepsilon},$ where $\nu_{2}^{\varepsilon}$ is the
stationary measure for the Markov chain with transition operator
\begin{equation}
C_{b}\left(  \mathcal{M}\right)  \ni\psi\longmapsto P_{R}\psi:=\mathbb{E}%
\left[  \psi\circ p\circ\mathbf{R}\right]  \in M_{b}\left(  \mathcal{M}%
\right)  \ , \label{P_R}%
\end{equation}
then the (weak) stochastic stability of $\mu_{T_{0}},\mu_{R_{0}}$ is usually
defined as the weak convergence of $\nu_{1}^{\varepsilon},\nu_{2}%
^{\varepsilon}$ respectively to $\mu_{T_{0}}$ and $\mu_{R_{0}}$ as
$\varepsilon$ tends to $0,$ which of course implies that $\mu_{\mathbf{T}_{0}%
}$ and $\mu_{\mathbf{R}_{0}}$ are the weak limit of respectively
$\mu_{\mathbf{T}}^{\varepsilon}$ and $\mu_{\mathbf{R}}^{\varepsilon}.$
Moreover, if and $\nu_{1}^{\varepsilon}$ and $\mu_{T_{0}}$ are a.c. w.r.t. the
Lebesgue measure, the convergence in $L_{\lambda}^{1}\left(  I\right)  $ of
the density of $\nu_{1}^{\varepsilon}$ to that of $\mu_{T_{0}},$ which is
equivalent to the convergence of $\nu_{1}^{\varepsilon}$ to $\mu_{T_{0}}$ in
the total variation distance, is referred to as the strong stochastic
stability of $\mu_{T_{0}}.$
\end{remark}

\begin{definition}
We will say that $\mu_{S_{0}}$ is \emph{stochastically stable} if, $\forall
f\in C_{b}\left(  \mathfrak{V}\right)  ,\mu_{\mathbf{S}}^{\varepsilon}\left(
f\right)  $ converges to $\mu_{\mathbf{S}_{0}}\left(  f\right)  ,$ as
$\varepsilon$ tends to $0.$
\end{definition}

We will now show that, since the correspondence $\mu_{\mathbf{T}}%
^{\varepsilon}\longmapsto\mu_{\overline{\mathbf{R}}}^{\varepsilon}%
\longmapsto\mu_{\overline{\mathbf{S}}}^{\varepsilon}$ is injective, the
stochastic stability of $\mu_{T_{0}}$ imply the weak convergence of
$\mu_{\overline{\mathbf{S}}}^{\varepsilon}$ to $\mu_{S_{0}}.$ Furthermore, we
will prove that if $\mu_{T_{0}}$ is stochastically stable, the injectivity of
the correspondence $\mu_{\mathbf{T}}^{\varepsilon}\longmapsto\mu_{\mathbf{R}%
}^{\varepsilon}\longmapsto\mu_{\mathbf{S}}^{\varepsilon},$ together with the
hypothesis of $R_{\eta}$ being continuous for any $\eta\in spt\lambda
_{\varepsilon},$ imply the stochastic stability of the physical measure for
the unperturbed flow that is what stated in Theorem \ref{main}. We will also
show that, in order to prove Theorem \ref{main}, we can drop the hypothesis on
the continuity of the $R_{\eta}$'s if we assume the strong stochastic
stability of $\mu_{T_{0}}.$

In the rest of the section we will always assume $\mu_{T_{0}}$ to be
stochastically stable. As an example, in Section \ref{SST} we will prove that
this is the case for the invariant measure of the Lorenz-like cusp map and for
the classical Lorenz map introduced in Section \ref{assT}.

\subsection{Stochastic stability of $\mu_{R_{0}}$}

The following result refers for example to the case where one considers the
first return maps on the Poincar\'{e} section $\mathcal{M}$ given in the
appendix in Section \ref{M"}.

\begin{theorem}
If for any $\eta\in\left[  0,\varepsilon\right]  ,R_{\eta}%
:\mathcal{M\circlearrowleft}$ is continuous and $\mu_{T_{0}}$ is
stochastically stable, then $\mu_{\overline{\mathbf{R}}}^{\varepsilon}$ weakly
converges to $\mu_{\mathbf{R}_{0}}.$
\end{theorem}

\begin{proof}
Let $\left\{  \varepsilon_{m}\right\}  _{m\geq1}$ be any sequence in $[0,1)$
converging to $0$ and set $\mu_{\mathbf{T}}^{m}:=\mu_{\mathbf{T}}%
^{\varepsilon_{m}},\mu_{\overline{\mathbf{R}}}^{m}:=\mu_{\overline{\mathbf{R}%
}}^{\varepsilon_{m}}.$

For any $\psi\in L_{\mathbb{P}_{\lambda}}^{1}\left(  \Omega,C_{b}\left(
\mathcal{M}\right)  \right)  ,$ we set
\begin{align}
I\times\Omega &  \ni\left(  u,\omega\right)  \longmapsto\psi_{+}\left(
u,\omega\right)  :=\sup_{x\in q^{-1}\left(  u\right)  }\psi\left(
x,\omega\right)  =\sup_{\left(  x,\omega^{\prime}\right)  \in Q^{-1}\left(
u,\omega\right)  }\psi\left(  x,\omega^{\prime}\right)  \ ,\\
I\times\Omega &  \ni\left(  u,\omega\right)  \longmapsto\psi_{-}\left(
u,\omega\right)  :=\inf_{x\in q^{-1}\left(  u\right)  }\psi\left(
x,\omega\right)  =\inf_{\left(  x,\omega^{\prime}\right)  \in Q^{-1}\left(
u,\omega\right)  }\psi\left(  x,\omega^{\prime}\right)  \ .
\end{align}
Suppose first that $\psi\geq0.$ Given $m\geq1,$ by Proposition \ref{Prop1},
since $\left\{  \mu_{\mathbf{T}}^{m}\left(  \psi\circ\overline{\mathbf{R}}%
^{n}\right)  _{+}\right\}  _{n\geq1}$ is decreasing,
\begin{equation}
0\leq\mu_{\overline{\mathbf{R}}}^{m}\left(  \psi\right)  =\lim_{n\rightarrow
\infty}\mu_{\mathbf{T}}^{m}\left[  \left(  \psi\circ\overline{\mathbf{R}}%
^{n}\right)  _{+}\right]  =\underline{\lim}_{n}\mu_{\mathbf{T}}^{m}\left[
\left(  \psi\circ\overline{\mathbf{R}}^{n}\right)  _{+}\right]  \ .
\end{equation}
On the other hand, since $\left\{  \mu_{\mathbf{T}}^{m}\left(  \psi
\circ\overline{\mathbf{R}}^{n}\right)  _{-}\right\}  _{n\geq1}$ is
increasing,
\begin{equation}
\mu_{\overline{\mathbf{R}}}^{m}\left(  \psi\right)  =\lim_{n\rightarrow\infty
}\mu_{\mathbf{T}}^{m}\left[  \left(  \psi\circ\overline{\mathbf{R}}%
^{n}\right)  _{-}\right]  =\overline{\lim}_{n}\mu_{\mathbf{T}}^{m}\left[
\left(  \psi\circ\overline{\mathbf{R}}^{n}\right)  _{-}\right]  \ .
\end{equation}
The same considerations also hold for $\mu_{\mathbf{R}_{0}}\left(
\psi\right)  $ and $\left\{  \mu_{\mathbf{T}_{0}}\left[  \left(  \psi
\circ\overline{\mathbf{R}}^{n}\right)  _{\pm}\right]  \right\}  _{n\geq1},$
that is
\begin{align}
0  &  \leq\mu_{\mathbf{R}_{0}}\left(  \psi\right)  =\lim_{n\rightarrow\infty
}\mu_{\mathbf{T}_{0}}\left[  \left(  \psi\circ\mathbf{R}_{0}^{n}\right)
_{+}\right]  =\underline{\lim}_{n}\mu_{\mathbf{T}_{0}}\left[  \left(
\psi\circ\mathbf{R}_{0}^{n}\right)  _{+}\right]  =\\
&  =\lim_{n\rightarrow\infty}\mu_{\mathbf{T}_{0}}\left[  \left(  \psi
\circ\mathbf{R}_{0}^{n}\right)  _{-}\right]  =\overline{\lim}_{n}%
\mu_{\mathbf{T}_{0}}\left[  \left(  \psi\circ\mathbf{R}_{0}^{n}\right)
_{-}\right] \nonumber
\end{align}
(\cite{AP} Section 7.3.4.1). Hence we get
\begin{gather}
\left\vert \mu_{\overline{\mathbf{R}}}^{m}\left(  \psi\right)  -\mu
_{\mathbf{R}_{0}}\left(  \psi\right)  \right\vert =\mu_{\overline{\mathbf{R}}%
}^{m}\left(  \psi\right)  \vee\mu_{\mathbf{R}_{0}}\left(  \psi\right)
-\mu_{\overline{\mathbf{R}}}^{m}\left(  \psi\right)  \wedge\mu_{\mathbf{R}%
_{0}}\left(  \psi\right) \\
=\lim_{n\rightarrow\infty}\mu_{\mathbf{T}}^{m}\left[  \left(  \psi
\circ\overline{\mathbf{R}}^{n}\right)  _{-}\right]  \vee\lim_{n\rightarrow
\infty}\mu_{\mathbf{T}_{0}}\left[  \left(  \psi\circ\mathbf{R}_{0}^{n}\right)
_{-}\right]  -\lim_{n\rightarrow\infty}\mu_{\mathbf{T}}^{m}\left[  \left(
\psi\circ\overline{\mathbf{R}}^{n}\right)  _{+}\right]  \wedge\lim
_{n\rightarrow\infty}\mu_{\mathbf{R}_{0}}\left[  \left(  \psi\circ
\mathbf{R}_{0}^{n}\right)  _{+}\right] \nonumber\\
=\overline{\lim}_{n}\mu_{\mathbf{T}}^{m}\left[  \left(  \psi\circ
\overline{\mathbf{R}}^{n}\right)  _{-}\right]  \vee\overline{\lim}_{n}%
\mu_{\mathbf{T}_{0}}\left[  \left(  \psi\circ\mathbf{R}_{0}^{n}\right)
_{-}\right]  -\underline{\lim}_{n}\mu_{\mathbf{T}}^{m}\left[  \left(
\psi\circ\overline{\mathbf{R}}^{n}\right)  _{+}\right]  \wedge\underline{\lim
}_{n}\mu_{\mathbf{T}_{0}}\left[  \left(  \psi\circ\mathbf{R}_{0}^{n}\right)
_{+}\right]  \ .\nonumber
\end{gather}
But, since the marginal of $\mu_{\mathbf{T}_{0}}$ on $\left(  \Omega
,\mathcal{B}\left(  \Omega\right)  \right)  $ is $\delta_{\bar{0}},$%
\begin{align}
\left\vert \mu_{\overline{\mathbf{R}}}^{m}\left(  \psi\right)  -\mu
_{\mathbf{R}_{0}}\left(  \psi\right)  \right\vert  &  =\overline{\lim}_{n}%
\mu_{\mathbf{T}}^{m}\left[  \left(  \psi\circ\overline{\mathbf{R}}^{n}\right)
_{-}\right]  \vee\overline{\lim}_{n}\mu_{\mathbf{T}_{0}}\left[  \left(
\psi\circ\mathbf{R}^{n}\right)  _{-}\right] \\
&  -\underline{\lim}_{n}\mu_{\mathbf{T}}^{m}\left[  \left(  \psi\circ
\overline{\mathbf{R}}^{n}\right)  _{+}\right]  \wedge\underline{\lim}_{n}%
\mu_{\mathbf{T}_{0}}\left[  \left(  \psi\circ\overline{\mathbf{R}}^{n}\right)
_{+}\right]  \ .\nonumber
\end{align}
Moreover, since $\psi\in\left\{  \phi\in L_{\mathbb{P}_{\lambda}}^{1}\left(
\Omega,C_{b}\left(  \mathcal{M}\right)  \right)  :\phi\geq0\right\}  ,M_{\psi
}:=\sup_{x\in\mathcal{M}}\psi\left(  \cdot,x\right)  \in L^{1}\left(
\Omega,\mathbb{P}_{\lambda}\right)  $ and $0\leq\psi_{-}\leq\psi_{+}\leq
M_{\psi},$ then, by Fatou's Lemma,
\begin{gather}
\overline{\lim}_{n}\mu_{\mathbf{T}}^{m}\left[  \left(  \psi\circ
\overline{\mathbf{R}}^{n}\right)  _{-}\right]  \vee\overline{\lim}_{n}%
\mu_{\mathbf{T}_{0}}\left[  \left(  \psi\circ\overline{\mathbf{R}}^{n}\right)
_{-}\right]  -\underline{\lim}_{n}\mu_{\mathbf{T}}^{m}\left[  \left(
\psi\circ\overline{\mathbf{R}}^{n}\right)  _{+}\right]  \wedge\underline{\lim
}_{n}\mu_{\mathbf{T}_{0}}\left[  \left(  \psi\circ\overline{\mathbf{R}}%
^{n}\right)  _{+}\right] \\
\leq\mu_{\mathbf{T}}^{m}\left[  \overline{\lim}_{n}\left(  \psi\circ
\overline{\mathbf{R}}^{n}\right)  _{-}\right]  \vee\mu_{\mathbf{T}_{0}}\left[
\overline{\lim}_{n}\left(  \psi\circ\overline{\mathbf{R}}^{n}\right)
_{-}\right]  -\mu_{\mathbf{T}}^{m}\left[  \underline{\lim}_{n}\left(
\psi\circ\overline{\mathbf{R}}^{n}\right)  _{+}\right]  \wedge\mu
_{\mathbf{T}_{0}}\left[  \underline{\lim}_{n}\left(  \psi\circ\overline
{\mathbf{R}}^{n}\right)  _{+}\right] \\
=\mu_{\mathbf{T}}^{m}\left[  \overline{\lim}_{n}\left(  \psi\circ
\overline{\mathbf{R}}^{n}\right)  _{-}\right]  \vee\mu_{\mathbf{T}_{0}}\left[
\overline{\lim}_{n}\left(  \psi\circ\overline{\mathbf{R}}^{n}\right)
_{-}\right]  -\mu_{\mathbf{T}_{0}}\left[  \overline{\lim}_{n}\left(  \psi
\circ\overline{\mathbf{R}}^{n}\right)  _{-}\right] \\
+\mu_{\mathbf{T}_{0}}\left[  \overline{\lim}_{n}\left(  \psi\circ
\overline{\mathbf{R}}^{n}\right)  _{-}\right]  -\mu_{\mathbf{T}_{0}}\left[
\underline{\lim}_{n}\left(  \psi\circ\overline{\mathbf{R}}^{n}\right)
_{+}\right] \nonumber\\
+\mu_{\mathbf{T}_{0}}\left[  \underline{\lim}_{n}\left(  \psi\circ
\overline{\mathbf{R}}^{n}\right)  _{+}\right]  -\mu_{\mathbf{T}}^{m}\left[
\underline{\lim}_{n}\left(  \psi\circ\overline{\mathbf{R}}^{n}\right)
_{+}\right]  \wedge\mu_{\mathbf{T}_{0}}\left[  \underline{\lim}_{n}\left(
\psi\circ\overline{\mathbf{R}}^{n}\right)  _{+}\right] \nonumber\\
=\left(  \mu_{\mathbf{T}}^{m}\left[  \overline{\lim}_{n}\left(  \psi
\circ\overline{\mathbf{R}}^{n}\right)  _{-}\right]  -\mu_{\mathbf{T}_{0}%
}\left[  \overline{\lim}_{n}\left(  \psi\circ\overline{\mathbf{R}}^{n}\right)
_{-}\right]  \right)  \vee0+\mu_{\mathbf{T}_{0}}\left[  \overline{\lim}%
_{n}\left(  \psi\circ\overline{\mathbf{R}}^{n}\right)  _{-}\right]
-\mu_{\mathbf{T}_{0}}\left[  \underline{\lim}_{n}\left(  \psi\circ
\overline{\mathbf{R}}^{n}\right)  _{+}\right] \\
+\left(  \mu_{\mathbf{T}_{0}}\left[  \underline{\lim}_{n}\left(  \psi
\circ\overline{\mathbf{R}}^{n}\right)  _{+}\right]  -\mu_{\mathbf{T}}%
^{m}\left[  \underline{\lim}_{n}\left(  \psi\circ\overline{\mathbf{R}}%
^{n}\right)  _{+}\right]  \right)  \vee0\nonumber\\
\leq\left\vert \mu_{\mathbf{T}}^{m}\left[  \overline{\lim}_{n}\left(
\psi\circ\overline{\mathbf{R}}^{n}\right)  _{-}\right]  -\mu_{\mathbf{T}_{0}%
}\left[  \overline{\lim}_{n}\left(  \psi\circ\overline{\mathbf{R}}^{n}\right)
_{-}\right]  \right\vert +\left\vert \mu_{\mathbf{T}_{0}}\left[
\underline{\lim}_{n}\left(  \psi\circ\overline{\mathbf{R}}^{n}\right)
_{+}\right]  -\mu_{\mathbf{T}}^{m}\left[  \underline{\lim}_{n}\left(
\psi\circ\overline{\mathbf{R}}^{n}\right)  _{+}\right]  \right\vert \\
+\left\vert \mu_{\mathbf{T}_{0}}\left[  \overline{\lim}_{n}\left(  \psi
\circ\overline{\mathbf{R}}^{n}\right)  _{-}\right]  -\mu_{\mathbf{T}_{0}%
}\left[  \underline{\lim}_{n}\left(  \psi\circ\overline{\mathbf{R}}%
^{n}\right)  _{+}\right]  \right\vert \ .\nonumber
\end{gather}

Since $\mu_{\mathbf{T}}^{m}$ weakly converges to $\mu_{\mathbf{T}_{0}},$
setting $\overline{\psi}:=\overline{\lim}_{n}\left(  \psi\circ\overline
{\mathbf{R}}^{n}\right)  _{-},\underline{\psi}:=\underline{\lim}_{n}\left(
\psi\circ\overline{\mathbf{R}}^{n}\right)  _{+}$ we have $\overline{\psi
},\underline{\psi}\in\left\{  \phi\in L_{\mathbb{P}_{\lambda}}^{1}\left(
\Omega,C_{b}\left(  I\right)  \right)  :\phi\geq0\right\}  $ and
$\forall\epsilon>0,$ there exists $n_{\epsilon}^{\prime}\left(  \overline
{\psi}\right)  $ such that $\forall m>n_{\epsilon}^{\prime}\left(
\overline{\psi}\right)  ,$\linebreak$\left\vert \mu_{\mathbf{T}}^{m}\left[
\overline{\psi}\right]  -\mu_{\mathbf{T}_{0}}\left[  \overline{\psi}\right]
\right\vert <\epsilon$ as well as there exists $n_{\epsilon}^{\prime\prime
}\left(  \underline{\psi}\right)  $ such that $\forall m>n_{\epsilon}%
^{\prime\prime}\left(  \underline{\psi}\right)  ,\left\vert \mu_{\mathbf{T}%
}^{m}\left[  \underline{\psi}\right]  -\mu_{\mathbf{T}_{0}}\left[
\underline{\psi}\right]  \right\vert <\epsilon.$

On the other hand, $\forall n\geq0,$%
\begin{equation}
\mu_{\mathbf{T}_{0}}\left(  \psi\circ\overline{\mathbf{R}}^{n}\right)  _{\pm
}=\mu_{\mathbf{T}_{0}}\left(  \psi\circ\mathbf{R}_{0}^{n}\right)  _{\pm}%
=\mu_{T_{0}}\left(  \psi_{0}\circ R_{0}^{n}\right)  _{\pm}\ ,
\end{equation}
where $\psi_{0}:=\psi\left(  \cdot,\bar{0}\right)  ,$ so that
\begin{gather}
\left\vert \mu_{\mathbf{T}_{0}}\left[  \overline{\lim}_{n}\left(  \psi
\circ\overline{\mathbf{R}}^{n}\right)  _{-}\right]  -\mu_{\mathbf{T}_{0}%
}\left[  \underline{\lim}_{n}\left(  \psi\circ\overline{\mathbf{R}}%
^{n}\right)  _{+}\right]  \right\vert =\\
\left\vert \mu_{T_{0}}\left[  \overline{\lim}_{n}\left(  \psi_{0}\circ
R_{0}^{n}\right)  _{-}\right]  -\mu_{T_{0}}\left[  \underline{\lim}_{n}\left(
\psi_{0}\circ R_{0}^{n}\right)  _{+}\right]  \right\vert \leq\nonumber\\
\mu_{T_{0}}\left[  \left\vert \overline{\lim}_{n}\left(  \psi_{0}\circ
R_{0}^{n}\right)  _{-}-\underline{\lim}_{n}\left(  \psi_{0}\circ R_{0}%
^{n}\right)  _{+}\right\vert \right]  \ .\nonumber
\end{gather}
Since $\psi_{0}\in C_{b}\left(  \mathcal{M}\right)  $ and $\forall u\in
I,q^{-1}\left(  u\right)  \subset\mathcal{M}$ is compact, by Assumption 1,
$\forall\epsilon>0,\exists\delta_{\epsilon}>0,n_{\epsilon}>0$ such that
$\forall n\geq n_{\epsilon},u\in I,\operatorname*{diam}R_{0}^{n}\left(
q^{-1}\left(  u\right)  \right)  <\delta_{\epsilon}$ and $\forall x,y\in
R_{0}^{n}\left(  q^{-1}\left(  u\right)  \right)  ,$\linebreak$\left\vert
\psi_{0}\left(  x\right)  -\psi_{0}\left(  y\right)  \right\vert <\epsilon.$
Then,
\begin{equation}
\left\vert \mu_{T_{0}}\left[  \overline{\lim}_{n}\left(  \psi_{0}\circ
R_{0}^{n}\right)  _{-}\right]  -\mu_{T_{0}}\left[  \underline{\lim}_{n}\left(
\psi_{0}\circ R_{0}^{n}\right)  _{+}\right]  \right\vert \leq\epsilon\ .
\end{equation}
Hence, $\psi\in\left\{  \phi\in L_{\mathbb{P}}^{1}\left(  \Omega,C_{b}\left(
\mathcal{M}\right)  \right)  :\phi\geq0\right\}  ,\forall m>m_{\epsilon
}\left(  \psi\right)  :=n_{\epsilon}^{\prime}\left(  \bar{\psi}\right)  \vee
n_{\epsilon}^{\prime\prime}\left(  \underline{\psi}\right)  ,$%
\begin{equation}
\left\vert \mu_{\overline{\mathbf{R}}}^{m}\left(  \psi\right)  -\mu
_{\mathbf{R}_{0}}\left(  \psi\right)  \right\vert \leq3\epsilon\ ,
\end{equation}
but decomposing any real-valued function $\psi$ on $\Omega\times\mathcal{M}$
as $\psi=\psi\vee0-\left\vert \psi\wedge0\right\vert ,$ we get that given any
$\psi\in L_{\mathbb{P}_{\lambda}}^{1}\left(  \Omega,C_{b}\left(
\mathcal{M}\right)  \right)  ,\forall\epsilon>0$ $\exists m_{\epsilon}\left(
\psi\right)  $ such that $\forall m>m_{\epsilon}\left(  \psi\right)
,\left\vert \mu_{\overline{\mathbf{R}}}^{m}\left(  \psi\right)  -\mu
_{\mathbf{R}_{0}}\left(  \psi\right)  \right\vert \leq6\epsilon.$
\end{proof}

\begin{lemma}
\label{Lf}If $\mu_{\overline{\mathbf{R}}}^{\varepsilon}$ weakly converges to
$\mu_{\mathbf{R}_{0}},$ then $\mu_{\mathbf{R}}^{\varepsilon}$ weakly converges
to $\mu_{\mathbf{R}_{0}}$ too.
\end{lemma}

\begin{proof}
For any $A\in\mathcal{B}\left(  \mathcal{M}\right)  \otimes\mathcal{F},$ we
denote by $\overline{A}$ its closure and recall that $\mu_{\mathbf{R}%
}^{\varepsilon}\left(  A\right)  =\mu_{\overline{\mathbf{R}}}^{\varepsilon
}\left(  \mathbf{1}_{\mathbf{K}\left(  A\right)  }\right)  .$ Moreover, for
any real-valued Borel function $\psi$ on $\mathcal{M}\times\Omega
,\mu_{\mathbf{R}_{0}}\left(  \psi\right)  =\mu_{\mathbf{R}_{0}}\left(
\psi\circ\mathbf{K}\right)  .$ Hence, defining, for any $B\in\mathcal{B}%
\left(  \mathcal{M}\right)  ,C\in\mathcal{F},\epsilon>0$%
\begin{equation}
\left(  B\times C\right)  _{\epsilon}:=\left\{  \left(  x,\omega\right)
\in\mathcal{M}\times\Omega:\inf_{y\in B}\left\Vert x-y\right\Vert
<\epsilon,\inf_{\omega^{\prime}\in C}\rho\left(  \omega,\omega^{\prime
}\right)  <\varepsilon\right\}
\end{equation}
we set
\begin{align}
L\left(  \mu_{\mathbf{R}}^{\varepsilon},\mu_{\mathbf{R}_{0}}\right)   &
:=\inf\left\{  \epsilon>0:\mu_{\mathbf{R}}^{\varepsilon}\left(  \overline
{B\times C}\right)  \leq\mu_{\mathbf{R}_{0}}\left(  \overline{\left(  B\times
C\right)  _{\epsilon}}\right)  +\epsilon,\forall B\in\mathcal{B}\left(
\mathcal{M}\right)  ,C\in\mathcal{F}\right\} \\
&  =\inf\left\{  \epsilon>0:\mu_{\overline{\mathbf{R}}}^{\varepsilon}\left(
\mathbf{K}\left(  \overline{B\times C}\right)  \right)  \leq\mu_{\mathbf{R}%
_{0}}\left(  \mathbf{K}\left(  \overline{\left(  B\times C\right)  _{\epsilon
}}\right)  \right)  +\epsilon,\forall B\in\mathcal{B}\left(  \mathcal{M}%
\right)  ,C\in\mathcal{F}\right\}  \ .\nonumber
\end{align}
But, for any $B\in\mathcal{B}\left(  \mathcal{M}\right)  ,C\in\mathcal{F},$
\begin{align}
\mathbf{K}\left(  B\times C\right)   &  =\left\{  \left(  x,\omega\right)
\in\mathcal{M}\times\Omega:\left(  \kappa_{\pi\left(  \omega\right)  }%
^{-1}\left(  x\right)  ,\omega\right)  \in B\times C\right\} \\
&  =\left(  \bigcap\limits_{\omega\in C}\kappa_{\pi\left(  \omega\right)
}\left(  B\right)  \right)  \times C\ ,\nonumber
\end{align}
hence, since for any $\eta\in spt\lambda_{\varepsilon},\kappa_{\eta}$ is a
diffeomorphism, $\kappa_{\eta}\left(  \mathcal{B}\left(  \mathcal{M}\right)
\right)  =\mathcal{B}\left(  \mathcal{M}\right)  ,$ i.e. $L\left(
\mu_{\mathbf{R}}^{\varepsilon},\mu_{\mathbf{R}_{0}}\right)  =L\left(
\mu_{\overline{\mathbf{R}}}^{\varepsilon},\mu_{\mathbf{R}_{0}}\right)  .$
Therefore, the distance between $\mu_{\mathbf{R}}^{\varepsilon}$ and
$\mu_{\mathbf{R}_{0}}$ in the L\'{e}vy-Prokhorov metric, namely $LP\left(
\mu_{\mathbf{R}}^{\varepsilon},\mu_{\mathbf{R}_{0}}\right)  :=L\left(
\mu_{\mathbf{R}}^{\varepsilon},\mu_{\mathbf{R}_{0}}\right)  \vee L\left(
\mu_{\mathbf{R}_{0}},\mu_{\mathbf{R}}^{\varepsilon}\right)  ,$ equal that
between $\mu_{\overline{\mathbf{R}}}^{\varepsilon}$ and $\mu_{\mathbf{R}_{0}%
}.$ Since the weak convergence of measures is equivalent to the convergence in
the $LP$ distance we get the thesis.
\end{proof}

The last two results prove the following.

\begin{corollary}
If for any $\eta\in spt\lambda_{\varepsilon},R_{\eta}%
:\mathcal{M\circlearrowleft}$ is continuous and $\mu_{T_{0}}$ is
stochastically stable, then $\mu_{R_{0}}$ is also stochastically stable.
\end{corollary}

\begin{theorem}
\label{Th1}If $\nu_{1}^{\varepsilon}$ weakly converges to $\mu_{T_{0}},$ then
$\mu_{T_{0}}$ is stochastically stable and $\bar{\nu}_{2}^{\varepsilon}$
weakly converges to $\mu_{R_{0}}.$
\end{theorem}

\begin{proof}
By (\ref{QR=TQ}), $\forall\varphi\in L_{\mathbb{P}_{\lambda}}^{1}\left(
\Omega,C_{b}\left(  I\right)  \right)  $ and $n\geq1,$ it follows that
\begin{equation}
\mu_{\overline{\mathbf{R}}}^{m}\left[  \varphi\circ Q\circ\overline
{\mathbf{R}}^{n}\right]  =\mu_{\overline{\mathbf{R}}}^{m}\left[  \varphi
\circ\mathbf{T}^{n}\circ Q\right]  \ .
\end{equation}
Moreover, since $\forall\left(  u,\omega\right)  \in I\times\Omega,$%
\begin{equation}
\left(  \varphi\circ Q\right)  _{-}\left(  u,\omega\right)  =\inf_{x\in
q^{-1}\left(  u\right)  }\varphi\circ Q\left(  x,\omega\right)  =\inf_{x\in
q^{-1}\left(  u\right)  }\varphi\circ q\left(  x\right)  =\varphi\left(
u\right)  \ ,
\end{equation}
as well as
\begin{equation}
\left(  \varphi\circ Q\right)  _{-}\left(  u,\omega\right)  =\sup_{x\in
q^{-1}\left(  u\right)  }\varphi\circ Q\left(  x,\omega\right)  =\sup_{x\in
q^{-1}\left(  u\right)  }\varphi\circ q\left(  x\right)  =\varphi\left(
u\right)  \ ,
\end{equation}
$\forall m\geq1,$ by the invariance of $\mu_{\mathbf{T}}^{m}$ under
$\mathbf{T},$ we get
\begin{align}
\mu_{\overline{\mathbf{R}}}^{m}\left[  \varphi\circ Q\right]   &
=\lim_{n\rightarrow\infty}\mu_{\mathbf{T}}^{m}\left[  \left(  \varphi\circ
Q\circ\overline{\mathbf{R}}^{n}\right)  _{\pm}\right]  =\lim_{n\rightarrow
\infty}\mu_{\mathbf{T}}^{m}\left[  \left(  \varphi\circ\mathbf{T}^{n}\circ
Q\right)  _{\pm}\right] \\
&  =\lim_{n\rightarrow\infty}\mu_{\mathbf{T}}^{m}\left[  \varphi
\circ\mathbf{T}^{n}\right]  =\mu_{\mathbf{T}}^{m}\left[  \varphi\right]
\ .\nonumber
\end{align}
Furthermore, by (\ref{defT}), $\forall\varphi_{0}\in C_{b}\left(  I\right)
,u\in I$ since
\begin{equation}
\left(  \varphi_{0}\circ q\right)  _{-}\left(  u\right)  =\inf_{x\in
q^{-1}\left(  u\right)  }\varphi_{0}\circ q\left(  x\right)  =\varphi
_{0}\left(  u\right)  =\sup_{x\in q^{-1}\left(  u\right)  }\varphi_{0}\circ
q\left(  x\right)  =\left(  \varphi_{0}\circ q\right)  _{+}\left(  u\right)
\ ,
\end{equation}
then
\begin{align}
\mu_{R_{0}}\left[  \varphi_{0}\circ q\right]   &  =\lim_{n\rightarrow\infty
}\mu_{T_{0}}\left[  \left(  \varphi_{0}\circ q\circ R_{0}^{n}\right)  _{\pm
}\right]  =\lim_{n\rightarrow\infty}\mu_{T_{0}}\left[  \left(  \varphi
_{0}\circ T_{0}^{n}\circ q\right)  _{\pm}\right] \label{l0}\\
&  =\lim_{n\rightarrow\infty}\mu_{T_{0}}\left[  \varphi_{0}\circ T_{0}%
^{n}\right]  =\mu_{T_{0}}\left[  \varphi_{0}\right]  \ .\nonumber
\end{align}
Thus, $\forall\varphi\in L_{\mathbb{P}_{\lambda}}^{1}\left(  \Omega
,C_{b}\left(  I\right)  \right)  ,$ setting $\varphi_{0}=\varphi\left(
\cdot,\bar{0}\right)  ,\varphi_{0}\circ q=\varphi\circ Q\left(  \cdot,\bar
{0}\right)  $ and
\begin{equation}
\mu_{T_{0}}\otimes\delta_{\bar{0}}\left[  \varphi\right]  =\mu_{T_{0}}\left[
\varphi_{0}\right]  =\mu_{R_{0}}\left[  \varphi_{0}\circ q\right]  =\mu
_{R_{0}}\otimes\delta_{\bar{0}}\left[  \varphi_{0}\circ q\right]  =\mu_{R_{0}%
}\otimes\delta_{\bar{0}}\left[  \varphi\circ Q\right]  \ .
\end{equation}

Therefore, if $\mu_{\mathbf{T}}^{m}$ weakly converges to $\mu_{\mathbf{T}_{0}%
},$ then
\begin{align}
\lim_{m\rightarrow\infty}\mu_{\overline{\mathbf{R}}}^{m}\left[  \varphi\circ
Q\right]   &  =\lim_{m\rightarrow\infty}\mu_{\mathbf{T}}^{m}\left[
\varphi\right]  =\mu_{\mathbf{T}_{0}}\left[  \varphi\right]  =\mu_{T_{0}%
}\otimes\delta_{\bar{0}}\left[  \varphi\right] \label{l1}\\
&  =\mu_{R_{0}}\otimes\delta_{\bar{0}}\left[  \varphi\circ Q\right]
=\mu_{\mathbf{R}_{0}}\left[  \varphi\circ Q\right]  \ .\nonumber
\end{align}
Clearly, if $\nu_{1}^{m}$ weakly converges to $\mu_{T_{0}},$ since
$\mathbb{P}_{m}$ weakly converges to $\delta_{\bar{0}},$ then $\mu
_{\mathbf{T}}^{m}=\nu_{1}^{m}\otimes\mathbb{P}_{m}$ weakly converges to
$\mu_{\mathbf{T}_{0}}=\mu_{T_{0}}\otimes\delta_{\bar{0}}.$ Hence, $\forall
\bar{\varphi}\in C_{b}\left(  I\right)  ,$ by (\ref{defqq}), since
$\bar{\varphi}\circ\mathbf{q}\in L_{\mathbb{P}_{\lambda}}^{1}\left(
\Omega,C_{b}\left(  I\right)  \right)  ,$ and since $\forall x\in
\mathcal{M},\omega\in\Omega,\bar{\varphi}\circ q\left(  x\right)
=\bar{\varphi}\circ\mathbf{q}\circ Q\left(  x,\omega\right)  ,$ setting
$\varphi=\bar{\varphi}\circ\mathbf{q},$ by (\ref{l1}) we have
\begin{align}
\lim_{m\rightarrow\infty}\bar{\nu}_{2}^{m}\left[  \bar{\varphi}\circ q\right]
&  =\lim_{m\rightarrow\infty}\bar{\nu}_{2}^{m}\otimes\mathbb{P}_{m}\left[
\bar{\varphi}\circ q\right]  =\lim_{m\rightarrow\infty}\bar{\nu}_{2}%
^{m}\otimes\mathbb{P}_{m}\left[  \bar{\varphi}\circ\mathbf{q}\circ Q\right]
=\lim_{m\rightarrow\infty}\mu_{\overline{\mathbf{R}}}^{m}\left[  \bar{\varphi
}\circ\mathbf{q}\circ Q\right] \label{l2}\\
&  =\mu_{\mathbf{R}_{0}}\left[  \bar{\varphi}\circ\mathbf{q}\circ Q\right]
=\mu_{R_{0}}\left[  \bar{\varphi}\circ q\right]  \ .\nonumber
\end{align}

Given $A\in\mathcal{B}\left(  \mathcal{M}\right)  ,$ let
\begin{align}
q\left(  A\right)   &  :=\bigcup\limits_{x\in A}q\left(  x\right)  =\left\{
u\in I:u=q\left(  x\right)  \ ,\ x\in A\right\}  \ ,\\
b\left(  A\right)   &  :=\left\{  x\in\mathcal{M}:q\left(  x\right)  \in
q\left(  A\right)  \right\}  \supseteq A\ .
\end{align}
Moreover, $\forall\epsilon>0$ we set
\begin{equation}
\mathcal{M}\ni x\longmapsto\psi_{A}^{\epsilon}\left(  x\right)  :=\left(
1-\inf_{y\in A}\frac{\left\Vert x-y\right\Vert }{\epsilon}\right)  \vee
0\in\left[  0,1\right]  \ ,
\end{equation}
as well as
\begin{equation}
I\ni u\longmapsto\varphi_{J}^{\epsilon}\left(  x\right)  :=\left(
1-\inf_{v\in J}\frac{\left\vert u-v\right\vert }{\epsilon}\right)  \vee
0\in\left[  0,1\right]  \;,\;J\in\mathcal{B}\left(  I\right)  \ .
\end{equation}
Since
\begin{equation}
\inf_{y\in b\left(  A\right)  }\left\Vert x-y\right\Vert =\inf_{y\in b\left(
A\right)  }\left\vert q\left(  x\right)  -q\left(  y\right)  \right\vert
=\inf_{v\in q\left(  A\right)  }\left\vert q\left(  x\right)  -v\right\vert
\end{equation}
$\forall\epsilon>0$ we get $\psi_{b\left(  A\right)  }^{\epsilon}%
=\varphi_{q\left(  A\right)  }^{\epsilon}\circ q.$

Hence, given $A\in\mathcal{B}\left(  \mathcal{M}\right)  $ and denoting by
$\overline{A}$ its closure, since $\psi_{A}^{\epsilon}\in C_{b}\left(
\mathcal{M}\right)  ,\varphi_{q\left(  A\right)  }^{\epsilon}\in C_{b}\left(
I\right)  ,$ from (\ref{l2}), (\ref{l1}) and (\ref{l0}), $\forall\epsilon>0$
we have
\begin{align}
\overline{\lim}_{m}\bar{\nu}_{2}^{m}\left(  \overline{A}\right)   &
\leq\overline{\lim}_{m}\bar{\nu}_{2}^{m}\left[  \psi_{b\left(  \overline
{A}\right)  }^{\epsilon}\right]  =\overline{\lim}_{m}\bar{\nu}_{2}^{m}\left[
\varphi_{q\left(  \overline{A}\right)  }^{\epsilon}\circ q\right] \\
&  =\overline{\lim}_{m}\mu_{\overline{\mathbf{R}}}^{m}\left[  \varphi
_{q\left(  \overline{A}\right)  }^{\epsilon}\circ\mathbf{q}\circ Q\right]
=\lim_{m\rightarrow\infty}\mu_{\overline{\mathbf{R}}}^{m}\left[
\varphi_{q\left(  \overline{A}\right)  }^{\epsilon}\circ\mathbf{q}\circ
Q\right] \nonumber\\
&  =\mu_{R_{0}}\left[  \varphi_{q\left(  \overline{A}\right)  }^{\epsilon
}\circ q\right]  \ ,\nonumber
\end{align}
that is
\begin{equation}
\overline{\lim}_{m}\bar{\nu}_{2}^{m}\left(  \overline{A}\right)  \leq
\mu_{R_{0}}\left[  \mathbf{1}_{q\left(  \overline{A}\right)  }\circ q\right]
=\mu_{R_{0}}\left(  \overline{A}\right)
\end{equation}
and the thesis follows from Portmanteau theorem and Remark \ref{Rem2}.
\end{proof}

This result together with Lemma \ref{Lf} implies the stochastic stability of
$\mu_{R_{0}}.$

\begin{corollary}
\label{C1}If $\bar{\nu}_{2}^{\varepsilon}$ weakly converges to $\mu_{R_{0}},$
then $\mu_{R_{0}}$ is stochastically stable.
\end{corollary}

\begin{proof}
If $\bar{\nu}_{2}^{\varepsilon}$ weakly converges to $\mu_{R_{0}},$ then by
Remark \ref{Rem1} $\mu_{\overline{\mathbf{R}}}^{\varepsilon}=\bar{\nu}%
_{2}^{\varepsilon}\otimes\mathbb{P}_{\varepsilon}$ weakly converges to
$\mu_{\mathbf{R}_{0}}$ and, by Definition \ref{defSS_TR}, the thesis follows
from Lemma \ref{Lf}.
\end{proof}

\subsection{Stochastic stability of $\mu_{S_{0}}$}

As a corollary of the stochastic stability of $\mu_{R_{0}}$ we have the following.

\begin{proposition}
Let $\mathbf{t}$ be bounded away from zero and integrable w.r.t.
$\mu_{\mathbf{R}}.$ If $\mu_{R_{0}}$ is stochastically stable, then
$\mu_{S_{0}}$ is also stochastically stable.
\end{proposition}

\begin{proof}
Given $\eta\in spt\lambda_{\varepsilon},$ if $f$ is a bounded measurable
function on $\mathfrak{V},$ there exists a bounded measurable function
$\check{f}$ on $\mathcal{V}_{\eta}$ such that, denoting by $\breve{f}$ its
extension on $\mathcal{V}_{\eta}\times\Omega$ by setting
\begin{equation}
\mathcal{V}_{\eta}\times\Omega\ni\left(  x,s,\omega\right)  \longmapsto
\breve{f}\left(  x,s,\omega\right)  :=\check{f}\left(  x,s\right)
\in\mathbb{R}\ ,
\end{equation}
by (\ref{Pp}),
\begin{equation}
\check{f}\left(  \tilde{\pi}_{\eta}\left(  \cdot,\cdot\right)  \right)
=\breve{f}\left(  \tilde{\pi}_{\eta}\left(  \cdot,\cdot\right)  ,\cdot\right)
=f\circ\hat{\pi}_{\bar{\eta}}\left(  \cdot,\cdot,\cdot\right)  \ .
\end{equation}
Then, since the marginal on $\left(  \Omega,\mathcal{B}\left(  \Omega\right)
\right)  $ of $\mu_{\mathbf{R}_{0}}$ is the Dirac mass at $\bar{0},$ by
(\ref{deft_e}),
\begin{align}
\mu_{\mathbf{R}_{0}}\left[  \int_{0}^{\mathbf{t}_{0}}dsf\circ\hat{\pi}\left(
\cdot,\bar{0},s\right)  \right]   &  =\mu_{\mathbf{R}_{0}}\left[  \int
_{0}^{\mathbf{t}_{0}}dsf\circ\hat{\pi}_{\bar{0}}\left(  \cdot,\cdot,s\right)
\right]  =\mu_{\mathbf{R}_{0}}\left[  \int_{0}^{\mathbf{t}_{0}}ds\breve
{f}\left(  \tilde{\pi}_{0}\left(  \cdot,s\right)  ,\bar{0}\right)  \right] \\
&  =\mu_{R_{0}}\left[  \int_{0}^{\tau_{0}}ds\check{f}\circ\tilde{\pi}%
_{0}\left(  \cdot,s\right)  \right] \nonumber
\end{align}
and
\begin{equation}
\mu_{S_{0}}\left[  \check{f}\right]  =\frac{\mu_{R_{0}}\left[  \int_{0}%
^{\tau_{0}}ds\check{f}\circ\tilde{\pi}_{0}\left(  \cdot,s\right)  \right]
}{\mu_{\mathbf{R}_{0}}\left[  \mathbf{t}_{0}\right]  }=\mu_{\mathbf{S}_{0}%
}\left[  f\right]  \ .
\end{equation}

Since $\mathbf{t}\in L_{\mu_{\mathbf{R}}^{\varepsilon}}^{1},\mathbf{t}_{0}\in
L_{\mu_{\mathbf{R}_{0}}}^{1},$ for any $\epsilon>0,$ there exists
$M_{\epsilon}\in\mathbb{N}$ such that, $\forall M>M_{\epsilon},$%
\begin{gather}
\left\vert \mu_{\mathbf{R}}^{\varepsilon}\left(  \mathbf{t}\right)
-\mu_{\mathbf{R}}^{\varepsilon}\left(  \mathbf{t}\wedge M\right)  \right\vert
+\left\vert \mu_{\mathbf{R}_{0}}\left(  \mathbf{t}_{0}\right)  -\mu
_{\mathbf{R}}\left(  \mathbf{t}_{0}\wedge M\right)  \right\vert =\\
\mu_{\mathbf{R}}^{\varepsilon}\left[  \left(  \mathbf{t}-\mathbf{t}\wedge
M\right)  \mathbf{1}_{\left(  M,\infty\right)  }\left(  \mathbf{t}\right)
\right]  +\mu_{\mathbf{R}_{0}}\left[  \left(  \mathbf{t}_{0}-\mathbf{t}%
_{0}\wedge M\right)  \mathbf{1}_{\left(  M,\infty\right)  }\left(
\mathbf{t}_{0}\right)  \right]  \leq\epsilon\ .\nonumber
\end{gather}
Hence, for any bounded measurable function $f$ on $\mathfrak{V},$%
\begin{align}
\mu_{\mathbf{R}}^{\varepsilon}\left[  \int_{0}^{\mathbf{t}}dsf\circ\hat{\pi
}\left(  \cdot,\cdot,s\right)  \right]   &  =\mu_{\mathbf{R}}^{\varepsilon
}\left[  \left(  \int_{0}^{\mathbf{t}}dsf\circ\hat{\pi}\left(  \cdot
,\cdot,s\right)  \right)  \left(  \mathbf{1}_{\left[  0,M\right]  }\left(
\mathbf{t}\right)  +\mathbf{1}_{\left(  M,\infty\right)  }\left(
\mathbf{t}\right)  \right)  \right] \\
&  =\mu_{\mathbf{R}}^{\varepsilon}\left[  \int_{0}^{\mathbf{t}\wedge
M}dsf\circ\hat{\pi}\left(  \cdot,\cdot,s\right)  \right]  +\mu_{\mathbf{R}%
}^{\varepsilon}\left[  \mathbf{1}_{\left(  M,\infty\right)  }\left(
\mathbf{t}\right)  \int_{M}^{\mathbf{t}}dsf\circ\hat{\pi}\left(  \cdot
,\cdot,s\right)  \right] \nonumber
\end{align}
which implies
\begin{equation}
\left\vert \mu_{\mathbf{R}}^{\varepsilon}\left[  \int_{0}^{\mathbf{t}}%
dsf\circ\hat{\pi}\left(  \cdot,\cdot,s\right)  \right]  -\mu_{\mathbf{R}%
}^{\varepsilon}\left[  \int_{0}^{\mathbf{t}\wedge M}dsf\circ\hat{\pi}\left(
\cdot,\cdot,s\right)  \right]  \right\vert \leq\epsilon\sup_{\left(
x,\omega,s\right)  \in\mathfrak{V}}\left\vert f\left(  x,\omega,s\right)
\right\vert \ .
\end{equation}
Therefore, since%
\begin{equation}
\mu_{\mathbf{S}}^{\varepsilon}\left[  f\right]  =\frac{\mu_{\mathbf{R}%
}^{\varepsilon}\left(  \mathbf{t}\wedge M\right)  }{\mu_{\mathbf{R}%
}^{\varepsilon}\left(  \mathbf{t}\right)  }\frac{\mu_{\mathbf{R}}%
^{\varepsilon}\left[  \int_{0}^{\mathbf{t}\wedge M}dsf\circ\hat{\pi}\left(
\cdot,\cdot,s\right)  \right]  }{\mu_{\mathbf{R}}^{\varepsilon}\left(
\mathbf{t}\wedge M\right)  }+\frac{\mu_{\mathbf{R}}^{\varepsilon}\left[
\mathbf{1}_{\left(  M,\infty\right)  }\left(  \mathbf{t}\right)  \int
_{M}^{\mathbf{t}}dsf\circ\hat{\pi}\left(  \cdot,\cdot,s\right)  \right]  }%
{\mu_{\mathbf{R}}^{\varepsilon}\left(  \mathbf{t}\right)  }\ ,
\end{equation}
we obtain
\begin{align}
\left\vert \mu_{\mathbf{S}}^{\varepsilon}\left[  f\right]  -\frac
{\mu_{\mathbf{R}}^{\varepsilon}\left[  \int_{0}^{\mathbf{t}\wedge M}%
dsf\circ\hat{\pi}\left(  \cdot,\cdot,s\right)  \right]  }{\mu_{\mathbf{R}%
}^{\varepsilon}\left(  \mathbf{t}\wedge M\right)  }\right\vert  &
\leq\left\vert 1-\frac{\mu_{\mathbf{R}}^{\varepsilon}\left(  \mathbf{t}\wedge
M\right)  }{\mu_{\mathbf{R}}^{\varepsilon}\left(  \mathbf{t}\right)
}\right\vert \frac{\mu_{\mathbf{R}}^{\varepsilon}\left[  \int_{0}%
^{\mathbf{t}\wedge M}dsf\circ\hat{\pi}\left(  \cdot,\cdot,s\right)  \right]
}{\mu_{\mathbf{R}}^{\varepsilon}\left(  \mathbf{t}\wedge M\right)  }+\\
+\frac{\sup_{\left(  x,\omega,s\right)  \in\mathfrak{V}}\left\vert f\left(
x,\omega,s\right)  \right\vert }{\mu_{\mathbf{R}}^{\varepsilon}\left(
\mathbf{t}\right)  \wedge1}\epsilon &  \leq2\epsilon\frac{\sup_{\left(
x,\omega,s\right)  \in\mathfrak{V}}\left\vert f\left(  x,\omega,s\right)
\right\vert }{\mu_{\mathbf{R}}^{\varepsilon}\left(  \mathbf{t}\right)
\wedge1}\ .\nonumber
\end{align}
Moreover, by the same argument, we also get
\begin{equation}
\left\vert \mu_{\mathbf{R}_{0}}\left[  \int_{0}^{\mathbf{t}_{0}}dsf\circ
\hat{\pi}_{\bar{0}}\left(  \cdot,\cdot,s\right)  \right]  -\mu_{\mathbf{R}%
_{0}}\left[  \int_{0}^{\mathbf{t}_{0}\wedge M}dsf\circ\hat{\pi}_{\bar{0}%
}\left(  \cdot,\cdot,s\right)  \right]  \right\vert \leq\epsilon\sup_{\left(
x,\omega,s\right)  \in\mathfrak{V}}\left\vert f\left(  x,\omega,s\right)
\right\vert
\end{equation}
and
\begin{equation}
\left\vert \mu_{\mathbf{S}_{0}}\left[  f\right]  -\frac{\mu_{\mathbf{R}_{0}%
}\left[  \int_{0}^{\mathbf{t}_{0}\wedge M}dsf\circ\hat{\pi}\left(  \cdot
,\cdot,s\right)  \right]  }{\mu_{\mathbf{R}_{0}}\left(  \mathbf{t}_{0}\wedge
M\right)  }\right\vert \leq2\epsilon\frac{\sup_{\left(  x,\omega,s\right)
\in\mathfrak{V}}\left\vert f\left(  x,\omega,s\right)  \right\vert }%
{\mu_{\mathbf{R}_{0}}\left(  \mathbf{t}_{0}\right)  \wedge1}\ .
\end{equation}

Let $\mathbf{t}^{M}:=\mathbf{t}\wedge M,\mathbf{t}_{0}^{M}:=\mathbf{t}%
_{0}\wedge M$ and let $\left\{  \varepsilon_{m}\right\}  _{m\geq1}$ be any
sequence in $[0,1)$ converging to $0.$ Since $\mu_{\mathbf{R}}^{m}$ weakly
converges to $\mu_{\mathbf{R}_{0}},$ for any $\delta>0,$ there exists
$N_{\delta}>1$ such that, $\forall m\geq N_{\delta},$%
\begin{equation}
\left\vert \mu_{\mathbf{R}}^{m}\left(  \mathbf{t}^{M}\right)  -\mu
_{\mathbf{R}_{0}}\left(  \mathbf{t}^{M}\right)  \right\vert =\left\vert
\mu_{\mathbf{R}}^{m}\left(  \mathbf{t}^{M}\right)  -\mu_{\mathbf{R}_{0}%
}\left(  \mathbf{t}_{0}^{M}\right)  \right\vert \leq\delta\ .
\end{equation}
Moreover, since $\mathbf{t}^{M}$ is bounded, considering the linear map,
\begin{equation}
C_{\Omega}\left(  \mathfrak{V}\right)  \ni f\longmapsto\mathbf{E}_{M}\left(
f\right)  :=\int_{0}^{\mathbf{t}^{M}}dsf\circ\hat{\pi}\left(  \cdot
,\cdot,s\right)  \in L_{\mathbb{P}_{m}}^{1}\left(  \Omega,C_{b}\left(
\mathcal{M}\right)  \right)  \ ,
\end{equation}
from the linear space $C_{\Omega}\left(  \mathfrak{V}\right)  $ of bounded
measurable functions $f$ on $\mathfrak{V}$ such that\linebreak$\forall
\omega\in\Omega,f\left(  \cdot,\omega,\cdot\right)  \in C_{b}\left(
\mathcal{M}_{\tau_{\pi\left(  \omega\right)  }}\right)  $ to $L_{\mathbb{P}%
_{m}}^{1}\left(  \Omega,C_{b}\left(  \mathcal{M}\right)  \right)  ,$ for $m$
large enough, we get\linebreak$\left\vert \mu_{\mathbf{R}}^{m}\left[
\mathbf{E}_{M}\left(  f\right)  \right]  -\mu_{\mathbf{R}_{0}}\left[
\mathbf{E}_{M}\left(  f\right)  \right]  \right\vert \leq\delta.$ Therefore,
for $m$ sufficiently large,
\begin{align}
\left\vert \mu_{\mathbf{S}}^{m}\left[  f\right]  -\mu_{\mathbf{S}_{0}}\left[
f\right]  \right\vert  &  =\left\vert \frac{\mu_{\mathbf{R}}^{m}\left[
\mathbf{E}\left(  f\right)  \right]  }{\mu_{\mathbf{R}}^{m}\left[
\mathbf{t}\right]  }-\frac{\mu_{\mathbf{R}_{0}}\left[  \mathbf{E}\left(
f\right)  \right]  }{\mu_{\mathbf{R}_{0}}\left[  \mathbf{t}_{0}\right]
}\right\vert \\
&  \leq\left\vert \frac{\mu_{\mathbf{R}}^{m}\left[  \mathbf{E}_{M}\left(
f\right)  \right]  }{\mu_{\mathbf{R}}^{m}\left[  \mathbf{t}^{M}\right]
}-\frac{\mu_{\mathbf{R}_{0}}\left[  \mathbf{E}_{M}\left(  f\right)  \right]
}{\mu_{\mathbf{R}_{0}}\left[  \mathbf{t}_{0}^{M}\right]  }\right\vert
+4\epsilon\frac{\sup_{\left(  x,\omega,s\right)  \in\mathfrak{V}}\left\vert
f\left(  x,\omega,s\right)  \right\vert }{\mu_{\mathbf{R}_{0}}\left(
\mathbf{t}_{0}\right)  \wedge\mu\left(  \mathbf{t}\right)  \wedge
1}\ .\nonumber
\end{align}
and
\begin{align}
\left\vert \frac{\mu_{\mathbf{R}}^{m}\left[  \mathbf{E}_{M}\left(  f\right)
\right]  }{\mu_{\mathbf{R}}^{m}\left[  \mathbf{t}^{M}\right]  }-\frac
{\mu_{\mathbf{R}_{0}}\left[  \mathbf{E}_{M}\left(  f\right)  \right]  }%
{\mu_{\mathbf{R}_{0}}\left[  \mathbf{t}_{0}^{M}\right]  }\right\vert  &
\leq\left\vert \frac{\mu_{\mathbf{R}}^{m}\left[  \mathbf{E}_{M}\left(
f\right)  \right]  -\mu_{\mathbf{R}_{0}}\left[  \mathbf{E}_{M}\left(
f\right)  \right]  }{\mu_{\mathbf{R}_{0}}\left[  \mathbf{t}_{0}^{M}\right]
}\right\vert +\frac{\mu_{\mathbf{R}}^{m}\left[  \mathbf{E}_{M}\left(
\left\vert f\right\vert \right)  \right]  }{\mu_{\mathbf{R}}^{m}\left[
\mathbf{t}^{M}\right]  }\left\vert \frac{\mu_{\mathbf{R}}^{m}\left[
\mathbf{t}^{M}\right]  -\mu_{\mathbf{R}_{0}}\left[  \mathbf{t}_{0}^{M}\right]
}{\mu_{\mathbf{R}_{0}}\left[  \mathbf{t}_{0}^{M}\right]  }\right\vert \\
&  \leq\frac{1+\sup_{\left(  x,\omega,s\right)  \in\mathfrak{V}}\left\vert
f\left(  x,\omega,s\right)  \right\vert }{\mu_{\mathbf{R}_{0}}\left[
\mathbf{t}_{0}\right]  \wedge M}\delta\ .\nonumber
\end{align}

\end{proof}

For what concerns the weak convergence of the invariant measure of the flow
$\left(  \overline{\mathbf{S}}^{t},t\geq0\right)  $ to $\mu_{\mathbf{S}_{0}}$
we have the following result whose proof is identical to the preceding one and
so we omit it.

\begin{proposition}
Let $\mathbf{t}$ as in the previous proposition. If $\mu_{\overline
{\mathbf{R}}}$ weakly converges to $\mu_{\mathbf{R}_{0}},$ then $\mu
_{\overline{\mathbf{S}}}^{\varepsilon}$ weakly converges to $\mu
_{\mathbf{S}_{0}}.$
\end{proposition}

\subsection{Stochastic stability of the physical measure for the unperturbed
flow\label{SSPM}}

Here we will show that the stochastic stability of $\mu_{S_{0}}$ will imply
that of the physical measure.

Setting
\begin{equation}
\mathcal{M}\times\mathbb{R}^{+}\ni\left(  x,t\right)  \longmapsto\Psi_{\eta
}\left(  x,t\right)  :=\Phi_{\eta}^{t}\left(  x\right)  \in U\subset
\mathbb{R}^{3}\ ,
\end{equation}
where $U$ can be chosen to be independent of $\eta,$ we define the
diffeomorphism $\chi_{\eta}:\mathcal{V}_{\eta}\longrightarrow U$ relating the
original flow $\left(  \Phi_{\eta}^{t},t\geq0\right)  $ with its associated
suspension semiflow (\ref{ssf_omega-1}), i.e. such that
\begin{equation}
\chi_{\eta}\circ\tilde{\pi}_{\eta}\left(  \cdot,\cdot+t\right)  =\Phi_{\eta
}^{t}\circ\chi_{\eta} \label{diff_omega}%
\end{equation}
(see \cite{AP} par. 7.3.8).

Moreover, by (\ref{sSS}), for $n\geq2,$ we define
\begin{equation}
U\times\Omega\ni\left(  y,\omega\right)  \longmapsto\mathbf{\hat{s}}%
_{n}\left(  y,\omega\right)  :=\mathbf{\hat{s}}_{1}\left(  y,\omega\right)
+\mathbf{s}_{n-1}\left(  \Phi_{\pi\left(  \omega\right)  }^{\mathbf{\hat{s}%
}_{1}\left(  y,\omega\right)  }\left(  y\right)  ,\omega\right)  \in
\overline{\mathbb{R}^{+}}\ , \label{sn}%
\end{equation}
where $\mathbf{\hat{s}}_{1}$ is given in (\ref{s1}) and
\begin{equation}
U\times\Omega\ni\left(  y,\omega\right)  \longmapsto\bar{N}_{t}\left(
y,\omega\right)  :=\max\left\{  n\in\mathbb{Z}^{+}:\mathbf{\hat{s}}_{n}\left(
y,\omega\right)  \leq t\right\}  \in\mathbb{Z}^{+}\mathbb{\ }. \label{Nbar}%
\end{equation}
For any $\omega\in\Omega,$ we define the non autonomous phase field
$\mathbb{R}^{+}\ni t\longmapsto\bar{\phi}_{\omega}\left(  t,\cdot\right)  \in
C^{0}\left(  \mathbb{R}^{3},\mathbb{R}^{3}\right)  ,$ piecewise $C^{r}\left(
\mathbb{R}^{3},\mathbb{R}^{3}\right)  ,r\geq2,$ such that
\begin{align}
\mathbb{R}^{+}\times U  &  \ni\left(  t,y\right)  \longmapsto\bar{\phi
}_{\omega}\left(  t,y\right)  :=\phi_{\pi\left(  \theta^{\bar{N}_{t}\left(
y,\omega\right)  }\omega\right)  }\left(  y\right)  \in\mathbb{R}%
^{3}\label{fi_bar}\\
\phi_{\pi\left(  \theta^{\bar{N}_{t}\left(  y,\omega\right)  }\omega\right)
}  &  :=\phi_{\pi\left(  \omega\right)  }\left(  x\right)  \mathbf{1}%
_{[0,\mathbf{\hat{s}}_{1}\left(  y,\omega\right)  )}\left(  t\right)
+\sum_{n\geq1}\phi_{\pi\left(  \theta^{n}\omega\right)  }\mathbf{1}%
_{[\mathbf{\hat{s}}_{n}\left(  y,\omega\right)  ,\mathbf{\hat{s}}_{n+1}\left(
y,\omega\right)  )}\left(  t\right)
\end{align}
and denote by $\left(  \hat{\Phi}_{\omega}^{t,t_{0}},t>t_{0}\geq0\right)  $
the associated semiflow. Hence, because $\forall\eta\in\left[  0,\varepsilon
\right]  ,\Phi_{\eta}^{t}\left(  U\right)  \subseteq U$ it follows that
$\forall\omega\in\Omega,t>0,\hat{\Phi}_{\omega}^{t,0}\left(  U\right)
\subseteq U.$

Since by (\ref{MOmegat}) any $\mathbf{v}\in\mathfrak{V}$ can be represented as
a vector $\left(  x\left(  \mathbf{v}\right)  ,\omega\left(  \mathbf{v}%
\right)  ,s\left(  \mathbf{v}\right)  \right)  \in\left(  \mathcal{M}%
\times\Omega\right)  _{\mathbf{t}},$ let us consider the map%
\begin{equation}
\mathfrak{V}\ni\mathbf{v\longmapsto V}\left(  \mathbf{v}\right)  :=\left(
\hat{\Phi}_{\omega\left(  \mathbf{v}\right)  }^{s\left(  \mathbf{v}\right)
,0}\left(  x\left(  \mathbf{v}\right)  \right)  ,\omega\left(  \mathbf{v}%
\right)  \right)  \in U\times\Omega\ . \label{V}%
\end{equation}
Notice that, by the definition of $\left(  \hat{\Phi}_{\omega}^{t,0}%
,t\geq0\right)  ,\hat{\Phi}_{\omega\left(  \mathbf{v}\right)  }^{s\left(
\mathbf{v}\right)  ,0}\left(  x\left(  \mathbf{v}\right)  \right)  =\Phi
_{\pi\left(  \omega\left(  \mathbf{v}\right)  \right)  }^{s\left(
\mathbf{v}\right)  }\left(  x\left(  \mathbf{v}\right)  \right)  .$ Setting
\begin{equation}
U\times\Omega\times\mathbb{R}^{+}\ni\left(  u,\omega,t\right)  \longmapsto
X^{t}\left(  u,\omega\right)  :=\left(  \hat{\Phi}_{\omega}^{t,0}\left(
u\right)  ,\theta^{\bar{N}_{t}\left(  u,\omega\right)  }\omega\right)  \in
U\times\Omega\ , \label{X}%
\end{equation}
for $t\geq0,\mathbf{v\in}\mathfrak{V},$ by (\ref{V}), (\ref{Nbar}) and
(\ref{X}) we have
\begin{equation}
X^{t}\left(  \mathbf{V}\left(  \mathbf{v}\right)  \right)  =\left(  \hat{\Phi
}_{\omega\left(  \mathbf{v}\right)  }^{t,0}\left(  \hat{\Phi}_{\omega\left(
\mathbf{v}\right)  }^{s\left(  \mathbf{v}\right)  ,0}\left(  x\left(
\mathbf{v}\right)  \right)  \right)  ,\theta^{\bar{N}_{t}\left(  \hat{\Phi
}_{\omega\left(  \mathbf{v}\right)  }^{s\left(  \mathbf{v}\right)  ,0}\left(
x\left(  \mathbf{v}\right)  \right)  ,\omega\left(  \mathbf{v}\right)
\right)  }\omega\left(  \mathbf{v}\right)  \right)  \ .
\end{equation}
But, by (\ref{s1}), (\ref{sSS}) and (\ref{sn}),
\begin{align}
\mathbf{\hat{s}}_{1}\left(  \hat{\Phi}_{\omega\left(  \mathbf{v}\right)
}^{s\left(  \mathbf{v}\right)  ,0}\left(  x\left(  \mathbf{v}\right)  \right)
,\omega\left(  \mathbf{v}\right)  \right)   &  =\mathbf{\hat{s}}_{1}\left(
\Phi_{\pi\left(  \omega\left(  \mathbf{v}\right)  \right)  }^{s\left(
\mathbf{v}\right)  }\left(  x\left(  \mathbf{v}\right)  \right)
,\omega\left(  \mathbf{v}\right)  \right)  =\mathbf{t}\left(  x\left(
\mathbf{v}\right)  ,\omega\left(  \mathbf{v}\right)  \right)  -s\left(
\mathbf{v}\right) \\
\mathbf{s}_{n}\left(  \Phi_{\pi\left(  \omega\left(  \mathbf{v}\right)
\right)  }^{s\left(  \mathbf{v}\right)  }\left(  x\left(  \mathbf{v}\right)
\right)  ,\omega\left(  \mathbf{v}\right)  \right)   &  =\mathbf{s}_{n}\left(
\mathbf{R}\left(  x\left(  \mathbf{v}\right)  ,\omega\left(  \mathbf{v}%
\right)  \right)  ,\omega\left(  \mathbf{v}\right)  \right)  \ ,\;n\geq1\ ,
\end{align}
hence,
\begin{align}
\mathbf{\hat{s}}_{n}\left(  \hat{\Phi}_{\omega\left(  \mathbf{v}\right)
}^{s\left(  \mathbf{v}\right)  ,0}\left(  x\left(  \mathbf{v}\right)  \right)
,\omega\left(  \mathbf{v}\right)  \right)   &  =\mathbf{\hat{s}}_{n}\left(
\Phi_{\pi\left(  \omega\left(  \mathbf{v}\right)  \right)  }^{s\left(
\mathbf{v}\right)  }\left(  x\left(  \mathbf{v}\right)  \right)
,\omega\left(  \mathbf{v}\right)  \right) \\
&  =\mathbf{t}\left(  x\left(  \mathbf{v}\right)  ,\omega\left(
\mathbf{v}\right)  \right)  -s\left(  \mathbf{v}\right)  +\mathbf{s}%
_{n-1}\left(  \Phi_{\pi\left(  \omega\left(  \mathbf{v}\right)  \right)
}^{s\left(  \mathbf{v}\right)  }\left(  x\left(  \mathbf{v}\right)  \right)
,\omega\left(  \mathbf{v}\right)  \right) \nonumber\\
&  =\mathbf{t}\left(  x\left(  \mathbf{v}\right)  ,\omega\left(
\mathbf{v}\right)  \right)  -s\left(  \mathbf{v}\right)  +\mathbf{s}%
_{n}\left(  \mathbf{R}\left(  x\left(  \mathbf{v}\right)  ,\omega\left(
\mathbf{v}\right)  \right)  ,\omega\left(  \mathbf{v}\right)  \right)
\ ,\nonumber
\end{align}
which implies
\begin{equation}
\bar{N}_{t}\left(  \hat{\Phi}_{\omega\left(  \mathbf{v}\right)  }^{s\left(
\mathbf{v}\right)  ,0}\left(  x\left(  \mathbf{v}\right)  \right)
,\omega\left(  \mathbf{v}\right)  \right)  =\bar{N}_{t}\left(  \Phi
_{\pi\left(  \omega\left(  \mathbf{v}\right)  \right)  }^{s\left(
\mathbf{v}\right)  }\left(  x\left(  \mathbf{v}\right)  \right)
,\omega\left(  \mathbf{v}\right)  \right)  =N_{t}\left(  x\left(
\mathbf{v}\right)  ,\omega\left(  \mathbf{v}\right)  \right)
\end{equation}
and
\begin{align}
\hat{\Phi}_{\omega\left(  \mathbf{v}\right)  }^{t,0}\left(  \hat{\Phi}%
_{\omega\left(  \mathbf{v}\right)  }^{s\left(  \mathbf{v}\right)  ,0}\left(
x\left(  \mathbf{v}\right)  \right)  \right)   &  =\hat{\Phi}_{\omega\left(
\mathbf{v}\right)  }^{t,0}\left(  \Phi_{\pi\left(  \omega\left(
\mathbf{v}\right)  \right)  }^{s\left(  \mathbf{v}\right)  }\left(  x\left(
\mathbf{v}\right)  \right)  \right) \\
&  =\hat{\Phi}_{\omega\left(  \mathbf{v}\right)  }^{s\left(  \mathbf{v}%
\right)  +t,0}\left(  x\left(  \mathbf{v}\right)  \right)  \ .\nonumber
\end{align}
Therefore, by (\ref{defSS}) and (\ref{defSS1}),
\begin{align}
X^{t}\left(  \mathbf{V}\left(  \mathbf{v}\right)  \right)   &  =\left(
\hat{\Phi}_{\omega\left(  \mathbf{v}\right)  }^{s\left(  \mathbf{v}\right)
+t,0}\left(  x\left(  \mathbf{v}\right)  \right)  ,\theta^{N_{t}\left(
x\left(  \mathbf{v}\right)  ,\omega\left(  \mathbf{v}\right)  \right)  }%
\omega\left(  \mathbf{v}\right)  \right) \\
&  =\mathbf{V}\left(  \mathbf{S}^{t}\left(  x\left(  \mathbf{v}\right)
,\omega\left(  \mathbf{v}\right)  ,s\left(  \mathbf{v}\right)  \right)
\right) \nonumber\\
&  =\mathbf{V}\left(  \hat{\pi}\left(  x\left(  \mathbf{v}\right)
,\omega\left(  \mathbf{v}\right)  ,s\left(  \mathbf{v}\right)  +t\right)
\right)  \ ,\nonumber
\end{align}
that is
\begin{equation}
\mathbf{V}\circ\hat{\pi}\left(  \cdot,\cdot,\cdot+t\right)  =X^{t}%
\circ\mathbf{V}\;,\;t\geq0\ . \label{VpiXV}%
\end{equation}

By \cite{AP} Section 7.3.8 $\mu_{0}:=\left(  \Psi_{0}\right)  _{\#}\left(
\mu_{S_{0}}\right)  $ is the physical measure for $\left(  \Phi_{0}^{t}%
,t\geq0\right)  $ whose basin $B\left(  \mu_{0}\right)  $ covers a
neighborhood $V_{0}$ of the attractor of $\left(  \Phi_{0}^{t},t\geq0\right)
$ of full $\lambda^{3}$ measure which is a subset of $\chi_{0}\left(
\mathcal{V}_{0}\right)  \subseteq U.$ In fact, by the definition of
$\mathfrak{V},\forall\eta\in spt\lambda_{\varepsilon},\mathcal{V}_{\eta}%
\times\left\{  \bar{\eta}\right\}  \subset\mathfrak{V},$ and by (\ref{V})
$\mathbf{V}\left(  \mathcal{V}_{\eta}\times\left\{  \bar{\eta}\right\}
\right)  =\chi_{\eta}\left(  \mathcal{V}_{\eta}\right)  \times\left\{
\bar{\eta}\right\}  .$ Hence, setting $\mathcal{U}:=\mathbf{V}\left(
\mathfrak{V}\right)  ,\chi_{\eta}\left(  \mathcal{V}_{\eta}\right)  \subseteq
U_{0}=:p\left(  \mathcal{U}\right)  \subseteq U$ and in particular
$V_{0}\subset U_{0}.$

Let $\mu_{\mathbf{V}}^{\varepsilon}:=\mathbf{V}_{\#}\mu_{\mathbf{S}%
}^{\varepsilon}=\mu_{\mathbf{S}}^{\varepsilon}\circ\mathbf{V}^{-1}.$ By the
invariance of $\mu_{\mathbf{S}}^{\varepsilon}$ under the flow $\left(
\hat{\pi}\left(  \cdot,\cdot,\cdot+t\right)  ,t\geq0\right)  $ and
(\ref{VpiXV}) we get the invariance of $\mu_{\mathbf{V}}^{\varepsilon}$ under
the evolution given by $\left(  X^{t},t\geq0\right)  .$ Indeed, $\forall
A\subseteq\mathcal{U},$%
\begin{align}
\mu_{\mathbf{V}}^{\varepsilon}\left(  X^{t}\left(  A\right)  \right)   &
=\mu_{\mathbf{V}}^{\varepsilon}\left(  X^{t}\circ\mathbf{V}\left(
\mathbf{V}^{-1}\left(  A\right)  \right)  \right)  =\mu_{\mathbf{V}%
}^{\varepsilon}\left(  \mathbf{V}\circ\hat{\pi}\left(  \cdot,\cdot
,\cdot+t\right)  \left(  \mathbf{V}^{-1}\left(  A\right)  \right)  \right) \\
&  =\mu_{\mathbf{S}}^{\varepsilon}\left(  \hat{\pi}\left(  \cdot,\cdot
,\cdot+t\right)  \left(  \mathbf{V}^{-1}\left(  A\right)  \right)  \right)
=\mu_{\mathbf{S}}^{\varepsilon}\left(  \left(  \mathbf{V}^{-1}\left(
A\right)  \right)  \right)  =\mu_{\mathbf{V}}^{\varepsilon}\left(  A\right)
\ .\nonumber
\end{align}
Moreover, we have

\begin{proposition}
If $\mu_{S_{0}}$ is stochastically stable, then, as $\varepsilon$ tends to
$0,\mu_{\mathbf{V}}^{\varepsilon}$ weakly converges to $\mu_{0}\otimes
\delta_{\bar{0}}$ with $\mu_{0}$ the unperturbed physical measure.
\end{proposition}

\begin{proof}
Let $B\subseteq V_{0}\subset U_{0}.$ By (\ref{diff_omega}) $\chi_{0}%
^{-1}\left(  B\right)  \subset\mathcal{V}_{0}.$ Given $C\in\mathcal{F},$ we
set $A:=\chi_{0}^{-1}\left(  B\right)  \times C.$ By (\ref{defSS1}) $\hat{\pi
}\left(  A\right)  \subset\mathfrak{V}$ and by (\ref{Pp})
\begin{align}
\mu_{\mathbf{V}}^{\varepsilon}\left(  \mathbf{V}\circ\hat{\pi}\left(
A\right)  \right)   &  =\mu_{\mathbf{S}}^{\varepsilon}\left[  \hat{\pi}\left(
A\right)  \right]  \underset{\varepsilon\rightarrow0}{\longrightarrow}%
\mu_{\mathbf{S}_{0}}\left[  \hat{\pi}\left(  A\right)  \right]  =\mathbf{1}%
_{C}\left(  \bar{0}\right)  \mu_{S_{0}}\left[  \tilde{\pi}_{0}\circ p\left(
\chi_{0}^{-1}\left(  B\right)  \times\left\{  \bar{0}\right\}  \right)
\right] \\
&  =\mathbf{1}_{C}\left(  \bar{0}\right)  \mu_{S_{0}}\left[  \tilde{\pi}%
_{0}\left(  \chi_{0}^{-1}\left(  B\right)  \right)  \right]  \ .\nonumber
\end{align}
Since $\tilde{\pi}_{0}$ acts as the identity on $\mathcal{M}_{\tau_{0}}$ and
$\chi_{0}^{-1}\left(  B\right)  \subseteq\mathcal{M}_{\tau_{0}}$%
\begin{equation}
\mu_{S_{0}}\left[  \tilde{\pi}_{0}\left(  \chi_{0}^{-1}\left(  B\right)
\right)  \right]  =\mu_{S_{0}}\left[  \chi_{0}^{-1}\left(  B\right)  \right]
=\left(  \chi_{0}\right)  _{\#}\left(  \mu_{S_{0}}\right)  \left(  B\right)
\equiv\mu_{0}\left(  B\right)  \ .
\end{equation}

\end{proof}

\subsubsection{Proof of Theorem \ref{main}}

By construction $\mu_{\mathbf{V}}^{\varepsilon}$ is the physical measure of
$\left(  X^{t},t\geq0\right)  $ that is, for any bounded measurable function
$f$ on $U\times\Omega,\lim_{t\rightarrow\infty}\frac{1}{t}\int_{0}^{t}dsf\circ
X^{s}=\mu_{\mathbf{V}}^{\varepsilon}\left(  f\right)  .$ Moreover, the
projection on $U$ of the evolution $\left(  X^{t},t\geq0\right)  $ provides a
representation of the system evolution $\left(  \mathfrak{u}_{t}%
,t\geq0\right)  $ as it has been already shown in (\ref{X_t}). Therefore, the
thesis follows considering functions $U\times\Omega\ni\left(  y,\omega\right)
\longmapsto f\left(  y,\omega\right)  :=\tilde{f}\left(  y\right)
\in\mathbb{R}$ with $\tilde{f}$ bounded measurable on $U.$

\subsection{Stochastic stability of $\mu_{T_{0}}$\label{SST}}

In this section, to ease the notation, we will simply refer to the unperturbed
map $T_{0}$ as $T$ and consequently note $\mu_{T_{0}}$ as $\mu_{T}.$ Moreover,
for the same reason, since no confusion will arise, we will note $T_{\eta}$
for $\bar{T}_{\eta}.$ Furthermore, since as it is explained in the appendix in
the case $\mathcal{M}=\mathcal{M}^{\prime\prime}$ the invariant measure for
$T_{\eta}$ can be reconstructed from those of $\tilde{T}_{\eta},$ when
considering this case, here, with abuse of notation, we will refer to the
unperturbed map $\tilde{T}$ and to $\tilde{T}_{\eta}$ again as, respectively,
$T$ and $T_{\eta}$ unless differently specified.

As we stated in Section \ref{RDS}, the stochastic perturbation of a
one-dimensional map $T$ is realized through sequences of random
transformations. This means that we will compose maps as $T_{\eta_{k}}%
\circ\cdots\circ T_{\eta_{1}}$ with the $\{\eta_{j}\}_{j\in\mathbb{N}}\in
spt\lambda_{\varepsilon}$ taken independently from each other and with the
same distribution $\lambda_{\varepsilon}.$ This implies that the invariant
measure $\mu_{\mathbf{T}}$ of the skew system (\ref{defTT}) factorizes in the
direct product of $\mathbb{P}_{\varepsilon}:=\lambda_{\varepsilon}%
^{\mathbb{N}}$ times the so-called stationary measure $\nu_{1}^{\varepsilon}$
(see Remark \ref{Rem1}) which will be the stationary measure of the Markov
chain with transition probability
\begin{equation}
\mathcal{Q}(x,A):=\lambda_{\varepsilon}\{\eta\in\left[  -1,1\right]  :T_{\eta
}(x)\in A\}\ .
\end{equation}
where $x$ and $A$ are respectively a point and a Borel subset of the interval.
It is well known that whenever the stationary measure is absolutely continuous
with respect to the Lebesgue measure, its density will be a fixed point of the
random transfer operator which we are going to define together with the
strategy to prove stochastic stability of $\mu_{T}.$

We denote by $\mathcal{L}$ the transfer operator of the unperturbed map $T,$
by $\mathcal{L}_{\varepsilon}$ the random transfer operator defined by the
formula $\mathcal{L}_{\varepsilon}f=\int_{\left[  -1,1\right]  }%
d\lambda_{\varepsilon}\left(  \eta\right)  \mathcal{L}_{\eta}f,$ where $f$
belongs to some Banach space $\mathbb{B}\subset L^{1}:=L^{1}\left(
I,\lambda\right)  $ and by $\mathcal{L}_{\eta}$ is the transfer operator
associated to the perturbed map $T_{\eta}.$ Let us suppose that:

\begin{description}
\item[A1] The unperturbed transfer operator $\mathcal{L}$ verifies the
so-called Lasota-Yorke inequality, namely there exists constants
$0<\varkappa<1,D>0,$ such that for any $f\in\mathbb{B}$ we have
\begin{equation}
\left\Vert \mathcal{L}f\right\Vert _{\mathbb{B}}\leq\varkappa\left\Vert
f\right\Vert _{\mathbb{B}}+D\left\Vert f\right\Vert _{1}\ . \label{LY0}%
\end{equation}

\item[A2] The map $T$ preserve only one absolutely continuous invariant
probability measure $\mu$ with density $h,$ which therefore will be also
ergodic and mixing.

\item[A3] The random transfer operator $\mathcal{L}_{\varepsilon}$ verifies a
similar Lasota-Yorke inequality which, for sake of simplicity, we will assume
to hold with the same parameters $\varkappa$ and $D.$

\item[A4] There exits a measurable function $\left[  -1,1\right]
\ni\varepsilon\longmapsto\upsilon^{\prime}(\varepsilon)\in\mathbb{R}^{+}$
tending to zero when $\varepsilon\rightarrow0$ such that for $f\in
\mathbb{B}:$
\begin{equation}
|||\mathcal{L}f-\mathcal{L}_{\varepsilon}f|||\leq\upsilon^{\prime}%
(\varepsilon).
\end{equation}
where the norm $|||\cdot|||$ above is so defined: $|||L|||:=\sup_{\left\Vert
f\right\Vert _{\mathbb{B}}\leq1}\left\Vert Lf\right\Vert _{1},$ for a linear
operator $L:L^{1}\circlearrowleft.$
\end{description}

Besides, we add two very natural assumptions on the Markov chain given by our
random transformations, namely

\begin{description}
\item[A5] The transition probability $\mathcal{Q}(x,A)$ admits a density
$\mathfrak{q}_{\varepsilon}(x,y),$ namely: $\mathcal{Q}(x,A)=\int
_{A}\mathfrak{q}_{\varepsilon}(x,y)dy;$

\item[A6] $spt\mathcal{Q}(x,\cdot)=B_{\varepsilon}(Tx),$ for any $x$ in the
interval, where $B_{\varepsilon}(z)$ denotes the ball of center $z$ and radius
$\varepsilon.$
\end{description}

Assumptions A1-A3 on the transfer operators together with assumptions A5 and
A6 on the Markov chain defined by the random transformations, by Corollary 1
in \cite{BHV} guarantee that there will be only one absolutely continuous
stationary measure $\mu_{\varepsilon}$ with density $h_{\varepsilon}.$ At this
point, assumption A4 allow us to invoke the perturbation theorem of \cite{KL}
to assert that the norm $|||\cdot|||$ of the difference of the spectral
projections of the operators $\mathcal{L}$ and $\mathcal{L}_{\varepsilon}$
associated with the eigenvalue $1$ goes to zero when $\varepsilon
\rightarrow0.$ Since the corresponding eigenspace have dimension $1,$ we
conclude that $h_{\varepsilon}\rightarrow h$ in the $L^{1}$ norm and we have
proved the stochastic stability in the strong sense.

We will use as $\mathbb{B}$ the Banach space of quasi-H\"{o}lder functions. We
start by defining, for all functions $h\in L^{1}$ and $0<\alpha\leq1$ the
seminorm
\begin{equation}
|h|_{\alpha}:=\sup_{0<\varepsilon_{1}\leq\varepsilon_{0}}\frac{1}%
{\varepsilon_{1}^{\alpha}}\int\text{osc}(h,B_{\varepsilon_{1}}(x))dx\ ,
\label{halfa}%
\end{equation}
where, for any measurable set $A,\text{ osc}(h,A):=\text{Essup}_{x\in
A}h(x)-\text{Essinf}_{x\in A}h(x).$ We say that $h$ belong to the set
$V_{\alpha}\subseteq L^{1}$ if $|h|_{\alpha}<\infty.V_{\alpha}$ does not
depend on $\varepsilon_{0}$ and equipped with the norm
\begin{equation}
\left\Vert h\right\Vert _{\alpha}:=\left\vert h\right\vert _{\alpha
}+\left\Vert h\right\Vert _{1}%
\end{equation}
is a Banach space and from now on $V_{\alpha}$ will denote the Banach space
$\mathbb{B}:=(V_{\alpha},\left\Vert \cdot\right\Vert _{\alpha}).$ Furthermore,
it can be proved that $\mathbb{B}$ is continuously injected into $L^{\infty}$
and in particular $||h||_{\infty}\leq C_{s}||h||_{\alpha}$ where $C_{s}%
=\frac{\max(1,\varepsilon_{0}^{\alpha})}{\varepsilon_{0}^{n}},$ \cite{Sa}. The
value of $\alpha$ could be chosen equal to $1$ thanks to the horizontally
closeness hypothesis given below.

We now describe how the one-dimensional map $T$ is perturbed. From now on we
will suppose that $spt\lambda_{\varepsilon}\subset(-\varepsilon,\varepsilon)$
and choose the maps $T_{\eta}$ with absolutely continuous invariant
distribution $\mu_{\eta}$ in such a way they are close to $T$ in the following sense:

\begin{itemize}
\item denoting by $g=\frac{1}{|T^{\prime}|}$ and $g_{\eta}=\frac{1}{|T_{\eta
}^{\prime}|}$ the potentials of the two maps defined everywhere but in the
discontinuity, or critical, points $x_{0}$ and $x_{0,\eta}$ respectively, we
have that $g$ and $g_{\eta}$ satisfy the H\"{o}lder conditions, with the same
constant and exponent (we can always reduce to this case by choosing
$\varepsilon$ sufficiently small):
\begin{equation}
|g(x)-g(y)|\leq C_{h}|x-y|^{\epsilon}\ ;\ |g_{\eta}(x)-g_{\eta}(y)|\leq
C_{h}|x-y|^{\epsilon}\ ,
\end{equation}
where $(x,y)$ belong to the two domains on injectivity of the maps excluding
the critical points. We will call these domains $I_{1},I_{2}$ and $I_{1,\eta
},I_{2,\eta}$ respectively assuming that the domain labelled with $i=1$ is the leftmost.

\item The branches are \emph{horizontally close}, namely for any $z\in I$ we
have:
\begin{equation}
|T_{j}^{-1}(z)-T_{j,\eta}^{-1}(z)|\leq\upsilon(\varepsilon)\ ;\ |T^{\prime
}(T_{j}^{-1}(z))-T_{\eta}^{\prime}(T_{j,\eta}^{-1}(z))|\leq\upsilon
(\varepsilon),\ j=1,2\ ,
\end{equation}
where $T_{j}^{-1},T_{j,\eta}^{-1}$ denote the inverse branches of the two maps
and in the comparison of the derivatives we exclude $z=1.$ Here and in a few
other forthcoming bounds, where we compare close quantities, we will simply
write $\upsilon(\varepsilon)$ as the error term, meaning that such a function
goes to zero when $\varepsilon\rightarrow0$ and it is bounded as
$\upsilon(\varepsilon)\leq\varepsilon,$ with the explicit form of
$\upsilon(\varepsilon)$ which could change from an inequality to another
\footnote{Of course we could ask for bounds of the type $\upsilon
(\varepsilon)\leq C\varepsilon,$ where $C$ is a constant independent of
$\upsilon;$ the presence of the constant will simply modify some factor in the
next bounds and it will be irrelevant for our purposes.}.
\end{itemize}

With these assumptions, and those listed in Section \ref{assT}, if uniformly
in $\eta\in spt\lambda_{\varepsilon}$ the $L^{\infty}$ norm $g_{\eta}$ is
bounded by a constant in $\left(  0,1\right)  ,$ it follows from Butterley's
work \cite{Bu} that the map $T$ and each $T_{\eta}$ verify a Lasota-Yorke
inequality with the same constants (these constants are in fact explicitly
given and basically depend on the $L^{\infty}$ norm of $g_{\eta}$ and on the
constants $\lambda$ and $C_{\delta}$ appearing Theorems 4.1 and 4.2 in the
just cited Butterley's paper).

\begin{remark}
It is important to stress at this point that the uniform expandingness of our
maps $T_{\eta}$ is essential to prove the quasi-compactness of the associated
transfer operators. Therefore what just stated does not apply directly to the
one-dimensional Lorenz-cusp type map $\tilde{T}$ appearing in our previous
paper \cite{GMPV}. Nevertheless, making use of Theorem 2 in \cite{Pi}, we can
consider in place of the $\tilde{T}_{\eta}$'s the family of uniformly
expanding maps $\left\{  \overline{T}_{\eta}\right\}  _{\eta\in spt\lambda
_{\varepsilon}}$ such that $\overline{T}_{\eta}\circ W=W\circ\tilde{T}_{\eta
},$ with $W$ a given function defined in section \ref{Ltcm} of the appendix.
Indeed, these maps are uniformly expanding, more precisely, by construction,
we have $\inf_{\eta\in spt\lambda_{\varepsilon}}\inf\left\vert \overline
{T}_{\eta}^{\prime}\right\vert >1,$ which implies that the conditions A1 and
A3 given above are met. A2 is also met by the uniqueness of $\mu_{\tilde
{T}_{\eta}}$ which we proved in \cite{GMPV}, since $\mu_{\overline{T}_{\eta}%
}=\mu_{\tilde{T}_{\eta}}\circ W^{-1},$ while the validity of conditions A5 and
A6 follows by direct computation under the assumption of $\varepsilon$ being
sufficiently small.
\end{remark}

We now add two more assumptions for future purposes:

\begin{description}
\item[A7] \textbf{Vertical closeness of the derivatives} For any $\eta\in
spt\lambda_{\varepsilon}$ let $k_{\eta}:=\inf\left\{  k\in\mathbb{N}%
:x_{0,\eta}\in B_{k\eta}\left(  x_{0}\right)  \right\}  $ be the the smallest
integer $k$ for $k\eta$ be the radius of a ball centered in $x_{0}$ containing
the critical point of $T_{\eta}.$ We then assume that there exists a positive
constant $C$ such that
\begin{equation}
\sup_{\eta\in spt\lambda_{\varepsilon}}\sup_{x\in B_{k_{\eta}\eta}^{c}(x_{0}%
)}\{|T_{\eta}^{\prime}(x)-T^{\prime}(x)|\}\leq C\upsilon(\varepsilon)\ .
\end{equation}

\item[A8] \textbf{Translational similarity of the branches} We suppose that,
for any $\eta\in spt\lambda_{\varepsilon},$ the branches $T_{i}%
:=T\upharpoonleft_{I_{i}}$ and $T_{i,\eta}:=T_{\eta}\upharpoonleft_{I_{i,\eta
}}$ corresponding to the same value of the index $i=1,2$ will not intersect
each other, but in $x=0,1.$
\end{description}

The introduction of assumptions A7 and A8, as one can see by looking at Figure
2 below, which is taken from our previous work \cite{GMPV}, are motivated by
the change in the shape of $T_{\eta}$ w.r.t. that of $T$ an additive
perturbation of order $\eta$ to the phase velocity field produces. In
particular, A7, which was also already used in \cite{BR}, requires that
outside a small neighborhood of the abscissa of the cusp of the unperturbed
map $T,$ the derivative of $T$ and of all its perturbations $T_{\eta}$ are
$\varepsilon$ close. Assumption A8 requires that the left (resp. right)
branches of $T$ and of its perturbations $T_{\eta}$ can only meet in $0$
(resp. $1$).

\begin{theorem}
\label{SSST}For any realization of the noise $\eta\in spt\lambda_{\varepsilon
},$ let $T_{\eta}$ satisfy the assumptions A1-A8. Then, $\mu_{T}$ is strongly
stochastically stable.
\end{theorem}

\begin{proof}
If we were able to prove that the transfer operator for $T$ and for $T_{\eta}$
are close in the norm $|||\cdot|||$ uniformly in $\eta,$ we would get desired
result no matter of the probability distribution of the noise $\lambda
_{\varepsilon}.$ We therefore begin to compare the two operators, first of all
we have for any $h\in\mathbb{B}$%
\begin{equation}
(\mathcal{L}h-\mathcal{L}_{\eta}h)(x)=\sum_{i=1,2}h(T_{i}^{-1}x)g(T_{i}%
^{-1}x)-\sum_{i=1,2}h(T_{i,\eta}^{-1}x)g_{\omega}(T_{i,\eta}^{-1}x)
\end{equation}
With the usual adding and subtracting procedure, we can regroup the r.h.s. of
the previous expression in the following blocks:
\begin{equation}
(\mathcal{L}h-\mathcal{L}_{\eta}h)(x)=\sum_{i=1,2}[h(T_{i}^{-1}x)-h(T_{i,\eta
}^{-1}x)]g(T_{i}^{-1}x)+\sum_{i=1,2}h(T_{i,\eta}^{-1}x)[g(T_{i}^{-1}%
x)-g_{\eta}(T_{i,\eta}^{-1}x)].
\end{equation}
We denote with (I) and (II) the first and the second term on the r.h.s.. The
second one can be further decomposed as
\begin{equation}
(II)=\sum_{i=1,2}h(T_{i,\eta}^{-1}x)[g(T_{i}^{-1}x)-g(T_{i,\eta}^{-1}%
x)]+\sum_{i=1,2}h(T_{i,\eta}^{-1}x)[g(T_{i,\eta}^{-1}x)-g_{\eta}(T_{i,\eta
}^{-1}x)]
\end{equation}
and we call (III) and (IV) the two terms on the r.h.s.. We now begin to
estimate them.

\begin{description}
\item[(I)] We have by the horizontal closeness
\begin{equation}
\sum_{i=1,2}|h(T_{i}^{-1}x)-h(T_{i,\eta}^{-1}x)|g(T_{i}^{-1}x)\leq\sum
_{i=1,2}\text{osc}(h,B_{\varepsilon}(T_{i}^{-1}x))g(T_{i}^{-1}x)=\mathcal{L}%
[\text{osc}(h,B_{\varepsilon}(\cdot)]\ .
\end{equation}
By integrating and using duality on the transfer operator we get
\begin{equation}
\int|(I)|dx\leq\int\text{osc}(h,B_{\varepsilon}(x))dx\leq\varepsilon^{\alpha
}|h|_{\alpha}\ .
\end{equation}

\item[(III)] Since $g$ is H\"{o}lder we immediately have:
\begin{equation}
\int|(III)|dx\leq2\varepsilon C_{h}||h||_{\infty}\leq2\varepsilon^{\iota}%
C_{h}C_{s}|h|_{\alpha}\ .
\end{equation}

\item[(IV)] We rewrite the difference of the potential as
\begin{equation}
|g(T_{i,\eta}^{-1}x)-g_{\eta}(T_{i,\eta}^{-1}x)|\leq\frac{|T_{\eta}^{\prime
}(T_{i,\eta}^{-1}x)-T^{\prime}(T_{i,\eta}^{-1}x)|}{|T_{\eta}^{\prime
}(T_{i,\eta}^{-1}x)||T^{\prime}(T_{i,\eta}^{-1}x)|}\ .
\end{equation}
Let $y_{\eta}:=\inf_{x\in B_{k_{\eta}\eta}\left(  x_{0}\right)  }T_{\eta
}\left(  x\right)  .$ Assumption A8 implies $\lim_{\eta\rightarrow0}y_{\eta
}=1.$ Now, we first compute the integral $\int|\mathcal{L}h-\mathcal{L}_{\eta
}h|dx$ removing the interval $[y_{+},1],$ where $y_{+}:=\inf_{\eta\in
spt\lambda_{\varepsilon}}y_{\eta}.$ Clearly the estimate of $(I)$ and $(III)$
remain unchanged and, by the assumption A7, $(IV)$ immediately gives
\begin{equation}
\int|(IV)|dx\leq2C_{s}C\varepsilon|h|_{\alpha}\ .
\end{equation}
Therefore, we are left with the estimate of the error term $\int_{\Delta
}|\mathcal{L}h-\mathcal{L}_{\eta}h|dx,$ where $\Delta:=[y_{+},1].$%
\begin{align}
\int_{\Delta}|\mathcal{L}h-\mathcal{L}_{\eta}h|dx  &  \leq\int\mathcal{L}%
(|h|)\mathbf{1}_{\Delta}dx+\int\mathcal{L}_{\eta}(|h|)\mathbf{1}_{\Delta
}dx\leq\\
\int(|h|)\mathbf{1}_{\Delta}\circ Tdx+\int(|h|)\mathbf{1}_{\Delta}\circ
T_{\eta}dx  &  \leq2C_{s}|h|_{\alpha}[\text{Leb}(T^{-1}\Delta)+\text{Leb}%
(T_{\eta}^{-1}\Delta)]\leq\nonumber\\
&  16C_{s}|h|_{\alpha}\varepsilon\ .\nonumber
\end{align}
By collecting all the bounds just got, we conclude that $||\mathcal{L}%
-\mathcal{L}_{\varepsilon}||_{1}\leq O(\varepsilon)||f||_{\alpha}.$
\end{description}
\end{proof}

The proof we just gave refers to the case where $T$ and its perturbations are
respectively the Lorenz cusp-type map studied in \cite{GMPV}.

The same technique can be used to show the stochastic stability of the
classical Lorenz-type map again under the uniformly expandingness assumption.
In this case we do not need the vertical closeness of the derivatives; instead
we have to add the additional hypothesis that the largest elongations between
$|T(0)-T_{\eta}(0)|$ and $|T(1)-T_{\eta}(1)|$ are of order $\varepsilon$ for
any $\eta$ and moreover $|T_{1}^{-1}(T_{\eta}(0))|$ and $1-|T_{2}^{-1}%
(T_{\eta}(1))|$ are also of order $\varepsilon,$ where the last two quantities
are the size of the intervals whose images contains points that have only one
preimage when we apply simultaneously the maps $T$ and $T_{\eta}.$ Hence they
must be removed when we compare the associate transfer operators. The proof
then follows the same lines of the previous one and therefore is omitted.

\part{The semi-Markov description of the process}

In this part of the paper we will discuss the stochastic stability of the
unperturbed physical measure in the framework of PDMP.

\section{The associated semi-Markov Process in $\mathbb{R}^{3}$\label{SMRE}}

Let $\left\{  \mathfrak{x}_{n}\right\}  _{n\in\mathbb{Z}^{+}}$ be the
(homogeneous) Markov chain on $\left(  \Omega,\mathcal{F},\mathbb{P}\right)  $
with values in $\mathcal{M}$ such that, by (\ref{def_tSS}), for any
$A\in\mathcal{B}\left(  \mathcal{M}\right)  ,n\in\mathbb{N},$%
\begin{equation}
\mathbb{P}\left\{  \omega\in\Omega:\mathfrak{x}_{n}\left(  \omega\right)  \in
A|\mathfrak{F}_{n-1}^{\mathfrak{x}}\right\}  =\mathbb{P}\left\{  \omega
\in\Omega:\Phi_{\pi\left(  \theta^{n}\omega\right)  }^{\mathbf{t}\left(
\mathfrak{x}_{n-1},\theta^{n}\omega\right)  }\left(  \mathfrak{x}%
_{n-1}\right)  \in A|\mathfrak{x}_{n-1}\right\}  \;\mathbb{P}-a.s.\;,
\end{equation}
whose transition probability measure is therefore
\begin{equation}
\mathbb{P}\left\{  \mathfrak{x}_{1}\in dz|\mathfrak{x}_{0}\right\}
=\mathbb{\lambda}_{\varepsilon}\left\{  \eta\in\left[  -1,1\right]  :R_{\eta
}\left(  \mathfrak{x}_{0}\right)  \in dz\right\}  \ .
\end{equation}
Consequently, we define the random sequence $\left\{  \mathfrak{s}%
_{n}\right\}  _{n\in\mathbb{Z}^{+}}$ such that
\begin{align}
\Omega &  \ni\omega\longmapsto\mathfrak{s}_{0}\left(  \omega\right)
:=\mathbf{t}\left(  \mathfrak{x}_{0}\left(  \omega\right)  ,\omega\right)
\ ,\\
\Omega &  \ni\omega\longmapsto\mathfrak{s}_{n+1}\left(  \omega\right)
:=\mathfrak{s}_{n}\left(  \omega\right)  +\mathbf{t}\left(  \mathfrak{x}%
_{n}\left(  \omega\right)  ,\omega\right)  \in\mathbb{R}^{+}\ ,\ n\geq0\ ,
\end{align}
and accordingly the counting process $\left(  \mathbf{N}_{t},t\geq0\right)  $
such that
\begin{equation}
\mathbf{N}_{t}:=\sup\left\{  n\in\mathbb{Z}^{+}:\mathfrak{s}_{n}\leq
t\right\}  \ .
\end{equation}
We remark that for $\varepsilon$ sufficiently small $\lambda_{\varepsilon
}\left\{  \eta\in\left[  -1,1\right]  :\inf_{x\in\mathcal{M}}\tau_{\eta
}\left(  x\right)  >0\right\}  =1$ which imply that for any $t>0,\mathbb{P}%
\left\{  \omega\in\Omega:\mathbf{N}_{t}\left(  \omega\right)  <\infty\right\}
=1.$

The sequence $\left\{  \left(  \mathfrak{x}_{n},\mathbf{t}_{n}\right)
\right\}  _{n\in\mathbb{Z}^{+}}$ such that $\mathbf{t}_{0}:=\mathfrak{s}%
_{0},\mathbf{t}_{n}:=\mathfrak{s}_{n+1}-\mathfrak{s}_{n},n\geq0$ is a Markov
renewal process, since by construction, $\forall A\in\mathcal{B}\left(
\mathcal{M}\right)  ,t>0,n\geq0,$%
\begin{align}
\mathbb{P}\left\{  \mathfrak{x}_{n+1}\in A,\mathbf{t}_{n+1}\leq t|\mathfrak{x}%
_{n},\mathbf{t}_{n}\right\}   &  =\mathbb{P}\left\{  \mathfrak{x}_{n+1}\in
A,\mathbf{t}_{n+1}\leq t|\mathfrak{x}_{n}\right\}  \;\mathbb{P}-a.s.\ ,\\
\mathbb{P}\left\{  \mathfrak{x}_{1}\in A,\mathbf{t}_{1}\leq t|\mathfrak{x}%
_{0}\right\}   &  =\mathbb{\lambda}_{\varepsilon}\left\{  \eta\in\left[
-1,1\right]  :R_{\eta}\left(  \mathfrak{x}_{0}\right)  \in A,\tau_{\eta
}\left(  \mathfrak{x}_{0}\right)  \leq t\right\} \nonumber
\end{align}
and
\begin{equation}
\mathbb{P}\left\{  \mathbf{t}_{n+1}\leq t|\left\{  \mathfrak{x}_{n}\right\}
_{n\in\mathbb{Z}^{+}}\right\}  =\mathbb{P}\left\{  \mathbf{t}_{n+1}\leq
t|\mathfrak{x}_{n},\mathfrak{x}_{n+1}\right\}  \;\mathbb{P}-a.s.\ .
\end{equation}
Therefore $\left(  \mathfrak{x}_{t},t\geq0\right)  $ such that $\mathfrak{x}%
_{t}:=\mathfrak{x}_{\mathbf{N}_{t}}$ is the associated semi-Markov process
\cite{As}, \cite{KS}.

Let us set
\begin{equation}
U\times\Omega\ni\left(  y,\omega\right)  \longmapsto\mathbf{\hat{s}}%
_{1}\left(  y,\omega\right)  :=\inf\left\{  t>0:\Phi_{\pi\left(
\omega\right)  }^{t}\left(  y\right)  \in\mathcal{M}\right\}  \in
\mathbb{R}^{+}\ . \label{s1}%
\end{equation}
Then, we introduce the random process $\left(  \mathfrak{u}_{t}\left(
y_{0}\right)  ,t\geq0\right)  $ started at $y_{0}\in U,$ such that
\begin{align}
\Omega\ni\omega\longmapsto\mathfrak{u}_{t}\left(  y_{0}\right)  \left(
\omega\right)   &  :=\left(  1-\mathbf{1}_{\mathcal{M}}\left(  y_{0}\right)
\right)  \Phi_{\pi\left(  \omega\right)  }^{t}\left(  y_{0}\right)
\mathbf{1}_{[0,\mathbf{\hat{s}}_{1}\left(  y_{0},\omega\right)  )}\left(
t\right)  +\label{u_s}\\
&  +\mathbf{1}_{\{\Phi_{\pi\left(  \omega\right)  }^{\mathbf{\hat{s}}%
_{1}\left(  y_{0},\omega\right)  \left(  1-\mathbf{1}_{\mathcal{M}}\left(
y_{0}\right)  \right)  }\left(  y_{0}\right)  \}}\left(  \mathfrak{x}%
_{0}\right)  \Phi_{\pi\left(  \theta^{\left(  1-\mathbf{1}_{\mathcal{M}%
}\left(  y_{0}\right)  \right)  }\omega\right)  }^{t}\left(  \mathfrak{x}%
_{0}\right)  \mathbf{1}_{[\left(  1-\mathbf{1}_{\mathcal{M}}\left(
y_{0}\right)  \right)  \mathbf{\hat{s}}_{1}\left(  y_{0},\omega\right)
,\mathfrak{s}_{1}\left(  \omega\right)  )}\left(  t\right)  +\nonumber\\
&  +\sum_{n\geq1}\Phi_{\pi\left(  \theta^{n+\left(  1-\mathbf{1}_{\mathcal{M}%
}\left(  y_{0}\right)  \right)  }\omega\right)  }^{t-\mathfrak{s}_{n}\left(
\omega\right)  }\left(  \mathfrak{x}_{n}\right)  \mathbf{1}_{[\mathfrak{s}%
_{n}\left(  \omega\right)  ,\mathfrak{s}_{n+1}\left(  \omega\right)  )}\left(
t\right)  \in U\ .\nonumber
\end{align}
Setting $\left(  \mathfrak{l}_{t},t\geq0\right)  $ such that $\mathfrak{l}%
_{t}:=t-\mathfrak{s}_{\mathbf{N}_{t}},$ we have that $\left(  \mathfrak{u}%
_{t},t\geq0\right)  ,$ with $\mathfrak{u}_{t}\left(  \cdot\right)  =\left(
\Phi_{\pi\circ\theta^{\mathbf{N}_{t}}}^{\mathfrak{l}_{t}}\circ\mathfrak{x}%
_{t}\right)  \left(  \cdot\right)  ,$ is a semi-Markov random evolution
\cite{KS}.

\section{Stochastic stability of the unperturbed physical measure}

The process $\left(  \mathfrak{v}_{t},t\geq0\right)  $ such that
$\mathfrak{v}_{t}:=\left(  \mathfrak{x}_{t},\mathbf{N}_{t},\mathfrak{l}%
_{t}\right)  $ is a homogeneous Markov process as well as the process $\left(
\mathfrak{w}_{t},t\geq0\right)  $ such that $\mathfrak{w}_{t}:=\left(
\mathfrak{x}_{t},\mathfrak{l}_{t}\right)  .$ Moreover $\overline{\mathcal{F}%
}_{t}^{\mathfrak{w}}\subseteq\overline{\mathcal{F}}_{t}^{\mathfrak{v}}$ and it
follows from \cite{Da} Theorem A2.2 that these $\sigma$algebras are both right continuous.

By setting $z=0$ in formula (3.9) in \cite{Al} Corollary 1, (see also
\cite{Al} Theorem 3) we have that for any $x\in\mathcal{M},v\geq0$ and any
measurable set $A\subseteq\mathcal{M},$%
\begin{equation}
\lim_{t\rightarrow\infty}\mathbb{P}\left\{  \mathfrak{x}_{t}\in A,\mathfrak{l}%
_{t}>z|\mathfrak{x}_{0}=x,\mathfrak{l}_{0}=v\right\}  =\frac{\int
_{\mathcal{M}}\nu_{2}\left(  dx\right)  \left[  \mathbf{1}_{A}\left(
x\right)  \int_{z}^{\infty}ds\left(  1-F_{\tau}^{\varepsilon}\left(
s;x\right)  \right)  \right]  }{\int_{\mathcal{M}}\nu_{2}\left(  dx\right)
\left[  \int_{0}^{\infty}ds\left(  1-F_{\tau}^{\varepsilon}\left(  s;x\right)
\right)  \right]  }\;,\;\mathbb{P}\text{-a.s.\ ,}%
\end{equation}
where for any $x\in\mathcal{M},t\geq0,$%
\begin{equation}
F_{\tau}^{\varepsilon}\left(  t;x\right)  :=\mathbb{P}\left\{  \omega\in
\Omega:\mathbf{t}\left(  x,\omega\right)  \leq t\right\}  =\lambda
_{\varepsilon}\left\{  \eta\in\left[  -1,1\right]  :\tau_{\eta}\left(
x\right)  \leq t\right\}
\end{equation}
and (see Remark \ref{Rem2}) $\nu_{2}\in\mathfrak{P}\left(  \mathcal{M}\right)
$ is stationary for the Markov chain $\left\{  \mathfrak{x}_{n}\right\}
_{n\in\mathbb{Z}^{+}}.$

\begin{proposition}
For any bounded measurable function $f$ on $U$ and any $y_{0}\in U,$%
\begin{equation}
\lim_{t\rightarrow\infty}\frac{1}{t}\int_{0}^{t}dsf\circ\mathfrak{u}%
_{s}\left(  y_{0}\right)  =\frac{\int_{\left[  -1,1\right]  }\lambda
_{\varepsilon}\left(  d\eta\right)  \int_{\mathcal{M}}\nu_{2}\left(
dx\right)  \int_{0}^{\tau_{\eta}\left(  x\right)  }dsf\left(  \Phi_{\eta}%
^{s}\left(  x\right)  \right)  }{\int_{\mathcal{M}}\nu_{2}\left(  dx\right)
\left[  \int_{0}^{\infty}ds\left(  1-F_{\tau}^{\varepsilon}\left(  s;x\right)
\right)  \right]  }\;,\;\mathbb{P}\text{-a.s.}%
\end{equation}

\end{proposition}

\begin{proof}
Given any bounded measurable function $f$ on $U,$ by (\ref{u_s})
\begin{align}
\int_{0}^{t}dsf\circ\mathfrak{u}_{s}\left(  y_{0}\right)   &  =\left(
1-\mathbf{1}_{\mathcal{M}}\left(  y_{0}\right)  \right)  \int_{0}%
^{\mathbf{\hat{s}}_{1}\left(  y_{0},\cdot\right)  }dsf\left(  \Phi_{\pi}%
^{s}\left(  y_{0}\right)  \right)  +\\
&  +\mathbf{1}_{\{\Phi_{\pi}^{\mathbf{\hat{s}}_{1}\left(  y_{0},\cdot\right)
\left(  1-\mathbf{1}_{\mathcal{M}}\left(  y_{0}\right)  \right)  }\left(
y_{0}\right)  \}}\left(  \mathfrak{x}_{0}\right)  \int_{\mathbf{\hat{s}}%
_{1}\left(  y_{0},\cdot\right)  \left(  1-\mathbf{1}_{\mathcal{M}}\left(
y_{0}\right)  \right)  }^{\mathfrak{s}_{1}}dsf\left(  \Phi_{\pi\circ
\theta^{\left(  1-\mathbf{1}_{\mathcal{M}}\left(  y_{0}\right)  \right)  }%
}^{s-\mathbf{\hat{s}}_{1}\left(  y_{0},\cdot\right)  }\left(  \mathfrak{x}%
_{0}\right)  \right)  +\nonumber\\
&  +\sum_{n=1}^{\mathbf{N}_{t}-1}\int_{\mathfrak{s}_{n}}^{\mathfrak{s}_{n+1}%
}dsf\left(  \Phi_{\pi\circ\theta^{n+\left(  1-\mathbf{1}_{\mathcal{M}}\left(
y_{0}\right)  \right)  }}^{s-\mathfrak{s}_{n}}\left(  \mathfrak{x}_{n}\right)
\right)  +\int_{\mathfrak{s}_{\mathbf{N}_{t}}}^{t}dsf\left(  \Phi_{\pi
\circ\theta^{\mathbf{N}_{t}+\left(  1-\mathbf{1}_{\mathcal{M}}\left(
y_{0}\right)  \right)  }}^{s-\mathfrak{s}_{\mathbf{N}_{t}}}\left(
\mathfrak{x}_{t}\right)  \right)  \ .\nonumber
\end{align}
By definition the process $\left(  \mathfrak{u}_{t},t\geq0\right)  $ is
semi-regenerative with imbedded Markov renewal process $\left\{  \left(
\mathfrak{x}_{n},\mathbf{t}_{n}\right)  \right\}  _{n\in\mathbb{N}},$ that is
$\left(  \mathfrak{u}_{t},t\geq0\right)  $ is regenerative with imbedded
renewal process $\left\{  \mathfrak{s}_{n}\right\}  _{n\geq1}.$ Indeed,
$\forall n\geq1$ the post-process $\left(  \left(  \mathfrak{u}%
_{t+\mathfrak{s}_{n}},t\geq0\right)  ,\left\{  \mathbf{t}_{n+k}\right\}
_{k\geq1}\right)  $ is independent of the random vector $\left(
\mathbf{\hat{s}}_{1}\left(  y_{0},\cdot\right)  ,\mathfrak{s}_{1}%
,..,\mathfrak{s}_{n}\right)  $ (\cite{As} Section VII.5). It is enough to
restrict ourselves to the nondelayed case, that is $y_{0}\in\mathcal{M},$
since $\mathbb{E}\left[  \mathbf{\hat{s}}_{1}\left(  y_{0},\cdot\right)
\right]  ,\sup_{x\in\mathcal{M}}\lambda_{\varepsilon}\left(  \tau_{\eta
}\left(  x\right)  \right)  <\infty.$ By (\ref{def_tSS}) and (\ref{sSS})
\begin{align}
\lim_{n\rightarrow\infty}\frac{\mathfrak{s}_{n}}{n}  &  =\lim_{n\rightarrow
\infty}\frac{1}{n}\sum_{k=1}^{n}\mathbf{t}\left(  \mathfrak{x}_{n}%
,\cdot\right)  =\lim_{n\rightarrow\infty}\frac{1}{n}\sum_{k=1}^{n}\tau_{\pi
}\left(  \mathbf{R}^{k}\left(  y_{0},\cdot\right)  \right) \\
&  =\mathbb{P}\otimes\nu_{2}\left[  \tau_{\pi}\right]  =\int\nu_{2}\left(
dx\right)  \left[  \int_{0}^{\infty}ds\left(  1-F_{\tau}^{\varepsilon}\left(
s;x\right)  \right)  \right]  \;,\;\mathbb{P}\text{-a.s.}\ .\nonumber
\end{align}
Moreover, by renewal theory (see e.g. \cite{As} Section V)
\begin{equation}
\lim_{t\rightarrow\infty}\frac{t}{\mathbf{N}_{t}}=\nu_{2}\left[  \int
_{0}^{\infty}ds\left(  1-F_{\tau}^{\varepsilon}\left(  s;\cdot\right)
\right)  \right]  \;,\;\mathbb{P}\text{-a.s.}\ ,
\end{equation}
therefore,
\begin{align}
\lim_{t\rightarrow\infty}\left\vert \int_{\mathfrak{s}_{\mathbf{N}_{t}}}%
^{t}dsf\left(  \Phi_{\pi\circ\theta^{\mathbf{N}_{t}+\left(  1-\mathbf{1}%
_{\mathcal{M}}\left(  y_{0}\right)  \right)  }}^{s-\mathfrak{s}_{\mathbf{N}%
_{t}}}\left(  \mathfrak{x}_{t}\right)  \right)  \right\vert  &  \leq
\lim_{t\rightarrow\infty}\left\Vert f\right\Vert _{\infty}\frac{\mathfrak{l}%
_{t}}{t}=\\
&  =\lim_{t\rightarrow\infty}\left\Vert f\right\Vert _{\infty}\left(
1-\frac{\mathfrak{s}_{\mathbf{N}_{t}}}{\mathbf{N}_{t}}\frac{\mathbf{N}_{t}}%
{t}\right)  =0\;,\;\mathbb{P}\text{-a.s.}\ ,\nonumber
\end{align}
and the thesis follows from \cite{As} Theorem VI.3.1.
\end{proof}

Defining
\begin{equation}
\mu_{\varepsilon}\left(  f\right)  :=\frac{\int_{\left[  -1,1\right]  }%
\lambda_{\varepsilon}\left(  d\eta\right)  \int_{\mathcal{M}}\nu_{2}\left(
dx\right)  \int_{0}^{\tau_{\eta}\left(  x\right)  }ds}{\int\nu_{2}\left(
dx\right)  \left[  \int_{0}^{\infty}ds\left(  1-F_{\tau}^{\varepsilon}\left(
s;x\right)  \right)  \right]  }f\circ\Phi_{\eta}^{s}\left(  x\right)  \ ,
\end{equation}
by the stochastic stability of $\mu_{R_{0}},$ since for any bounded
real-valued measurable function $\varphi$ on $\mathcal{M}\times\mathbb{R}%
^{+},$%
\begin{gather}
\lim_{\varepsilon\rightarrow0}\frac{1}{\nu_{2}\left[  \int_{0}^{\infty
}ds\left(  1-F_{\tau}^{\varepsilon}\left(  s;\cdot\right)  \right)  \right]
}\int_{\mathcal{M}}\nu_{2}^{\varepsilon}\left(  dx\right)  \int_{0}%
^{\tau_{\eta}\left(  x\right)  }ds\varphi\left(  x,s\right)  =\\
=\int_{\mathcal{M}}\mu_{R_{0}}\left(  dx\right)  \int_{0}^{\tau_{0}\left(
x\right)  }ds\frac{1}{\mu_{R_{0}}\left[  \tau_{0}\right]  }\varphi\left(
x,s\right)  =\mu_{S_{0}}\left(  \varphi\right)  \ ,\nonumber
\end{gather}
we get
\begin{equation}
\lim_{\varepsilon\rightarrow0}\mu_{\varepsilon}\left(  f\right)  =\mu_{S_{0}%
}\left(  f\circ\Phi_{0}^{\cdot}\right)  =\int_{\mathcal{M}}\mu_{R_{0}}\left(
dx\right)  \int_{0}^{\tau_{0}\left(  x\right)  }ds\frac{1}{\mu_{R_{0}}\left[
\tau_{0}\right]  }f\circ\Phi_{0}^{s}\left(  x\right)  \ ,
\end{equation}
that is the proof of the following result.

\begin{theorem}
If $\nu_{2}^{\varepsilon}$ weakly converges to $\mu_{R_{0}},$ then
$\mu_{\varepsilon}$ weakly converges to the unperturbed physical measure.
\end{theorem}

\begin{remark}
This last result provides another proof of the stochastic stability of the
physical measure already given in Section \ref{SSPM}. Notice that, by
(\ref{u_s}) and by the definition $\left(  \hat{\Phi}_{\omega}^{t,t_{0}%
},t>t_{0}\geq0\right)  $ given at the beginning of that section, for any,
$u_{0}\in U,\omega\in\Omega,$ the associated trajectory $\left\{  \left(
u,t\right)  \in U\times\mathbb{R}^{+}:u=\mathfrak{u}_{t}\left(  u_{0}\right)
\left(  \omega\right)  \right\}  $ of $\left(  \mathfrak{u}_{t}\left(
u_{0}\right)  ,t\geq0\right)  ,$ that is the process $\left(  \mathfrak{u}%
_{t},t\geq0\right)  $ started at $u_{0},$ coincides with $\hat{\Phi}_{\omega
}^{t,0}\left(  u_{0}\right)  .$
\end{remark}

Therefore we are left with the proof of the existence of $\nu_{2}%
^{\varepsilon}$ and of its weak convergence to $\mu_{R_{0}}$ in the limit of
$\varepsilon$ tending to $0,$ i.e. of the stochastic stability of the
invariant measure for the unperturbed Poincar\'{e} map $R_{0}.$

We show that in this framework the existence of the invariant measure
$\bar{\nu}_{2}^{\varepsilon}$ for the transition operator $P_{\overline{R}},$
and its weak converge to $\mu_{R_{0}}$ can be proven following the same
argument which led to the existence and the strong stochastic stability of
$\nu_{1},$ the invariant measure for the transition operator $P_{T},$ given in
Section \ref{SST}.

Since $\mathcal{M}$ is foliated by the invariant stable foliation of the
unperturbed flow and that the leaves of the foliation can be rectified because
the regularity of the foliation is higher that $C^{1},$ any connected
component of $\mathcal{M}$ can be represented as
\begin{equation}
\mathcal{O}\ni\left(  u,v\right)  \longmapsto\mathbf{r}\left(  u,v\right)
:=\left(  y_{1}\left(  u,v\right)  ,y_{2}\left(  u,v\right)  ,y_{3}\left(
u,v\right)  \right)  \in\mathbb{R}^{3}\ , \label{M_loc_ch}%
\end{equation}
where $\mathcal{O}$ is a regular open subset of $\mathbb{R}^{2}$ and
$\mathbf{r}\in C^{1}\left(  \mathcal{O},\mathbb{R}^{3}\right)  \cap C\left(
\overline{\mathcal{O}},\mathbb{R}^{3}\right)  $ is such that, setting $\bar
{I}:=\left\{  u\in\mathbb{R}:\exists v\in\mathbb{R}\ s.t.\ \left(  u,v\right)
\in\mathcal{O}\right\}  ,\forall u\in\bar{I},\mathbf{r}\left(  u,\cdot\right)
\cap\mathcal{M}$ is an invariant stable leaf. Making the identification of
$\mathcal{M}$ with $\overline{\mathcal{O}}$ and of $I$ with $\bar{I},$ we also
identify $q:\mathcal{M}\longrightarrow I$ with $\tilde{q}:\mathcal{O}%
\longrightarrow\bar{I}\footnote{If $\bar{\iota}:\bar{I}\longrightarrow I,$
then $\bar{\iota}\circ\tilde{q}=q\circ\mathbf{r}.$}$ as well as, for any $%
\eta\in spt\lambda_\varepsilon,$ the map $\bar{R}_\eta:\mathcal{M}%
\circlearrowleft$ defined in (\ref{Rk=kR}) with the skew-product
\begin{equation}
\mathcal{O}\ni\left(  u,v\right)  \longmapsto\left(  \bar{T}_{\eta}\left(
u\right)  ,\Upsilon_{\eta}\left(  u,v\right)  \right)  \in\mathcal{O}^{\prime
}\;,\;\mathcal{O}^{\prime}\subseteq\mathcal{O\ }.
\end{equation}

Hence, denoting by $\overline{\mathcal{O}}\ni\left(  u,v\right)
\longmapsto\mathbf{m}\left(  u,v\right)  \in\mathbb{R}^{+}$ the
Radon-Nikod\'{y}m derivative w.r.t. $\lambda^{2}$ of the uniform probability
distribution $\lambda_{\mathcal{M}}$ on $\mathcal{M},$ if $\bar{h}\in
L^{1}\left(  \mathcal{M},\lambda_{\mathcal{M}}\right)  ,$ let $h:=\bar{h}%
\circ\mathbf{r}\in L^{1}\left(  \overline{\mathcal{O}},\mathbf{m}\lambda
^{2}\right)  .$

\begin{proposition}
If, for any $\eta\in spt\lambda_{\varepsilon},\mathcal{L}_{\eta}$ satisfies
the Lasota-Yorke inequality (\ref{LY0}), $T_{0}$ preserves only one invariant
measure a.c.w.r.t. $\lambda$ and the transition operator $P_{\overline{R}}$
satisfies the assumption A5 given in section \ref{SST}, then $\mu_{R_{0}}$ is
strongly stochastically stable.
\end{proposition}

\begin{proof}
Let us set $\mathfrak{M}:=\mathfrak{M}\left(  \mathcal{M}\right)  .$ For any
$\mu\in\mathfrak{M},g\in M_{b}\left(  \mathcal{M}\right)  $ and any
sub$\sigma$algebra $\mathcal{B}^{\prime}$ of $\mathcal{B}\left(
\mathcal{M}\right)  ,$%
\begin{align}
\mu\left(  g\right)   &  =\hat{\mu}_{+}\left(  \mathcal{M}\right)  \bar{\mu
}_{+}\left(  g\right)  -\mu_{-}\left(  \mathcal{M}\right)  \hat{\mu}%
_{-}\left(  g\right)  =\mu_{+}\left(  \mathcal{M}\right)  \hat{\mu}_{+}\left(
\hat{\mu}_{+}\left(  g|\mathcal{B}_{\mu_{+}}^{\prime}\right)  \right)
-\mu_{-}\left(  \mathcal{M}\right)  \hat{\mu}_{-}\left(  \hat{\mu}_{-}\left(
g|\mathcal{B}_{\mu_{-}}^{\prime}\right)  \right) \\
&  =\mu\left(  \mathcal{E}_{\mu}\left(  g|\mathcal{B}^{\prime}\right)
\right)  \ ,\nonumber
\end{align}
where $\mathcal{B}_{\mu_{\pm}}^{\prime}$ is the trace $\sigma$algebra of
$\mathcal{B}^{\prime}$ on $spt\mu_{\pm},$ namely $\left\{  A\subseteq
spt\mu_{\pm}:\exists B\in\mathcal{B}^{\prime}\ s.t.\ A=B\cap spt\mu_{\pm
}\right\}  $ and, since $\mu_{\pm}\left(  \hat{\mu}_{\mp}\left(
g|\mathcal{B}_{\mu_{\mp}}^{\prime}\right)  \right)  =0$ because $spt\hat{\mu
}_{\pm}\left(  g|\mathcal{B}_{\mu_{\pm}}^{\prime}\right)  \subseteq
spt\mu_{\pm},$%
\begin{equation}
\mathcal{E}_{\mu}\left(  g|\mathcal{B}^{\prime}\right)  :=\hat{\mu}_{+}\left(
g|\mathcal{B}_{\mu_{+}}^{\prime}\right)  +\hat{\mu}_{-}\left(  g|\mathcal{B}%
_{\mu_{-}}^{\prime}\right)  \ .
\end{equation}
Given $\mu\in\mathfrak{M}$ and $\mathcal{B}^{\prime}$ sub$\sigma$algebra of
$\mathcal{B}\left(  \mathcal{M}\right)  ,$ for any $g\in M_{b}\left(
\mathcal{M}\right)  \mathfrak{,}$%
\begin{equation}
\left\vert \mathcal{E}_{\mu}\left(  g|\mathcal{B}^{\prime}\right)  \right\vert
\leq\hat{\mu}_{+}\left(  \left\vert g\right\vert |\mathcal{B}_{\mu_{+}%
}^{\prime}\right)  +\hat{\mu}_{-}\left(  \left\vert g\right\vert
|\mathcal{B}_{\mu_{-}}^{\prime}\right)  =\mathcal{E}_{\mu}\left(  \left\vert
g\right\vert |\mathcal{B}^{\prime}\right)  \leq2\left\Vert g\right\Vert
_{\infty}\ .
\end{equation}
Hence, $\mathcal{E}_{\mu}\left(  \cdot|\mathcal{B}^{\prime}\right)  $ is a
bounded positivity preserving linear operator from $M_{b}\left(
\mathcal{M}\right)  $ to\linebreak$\left\{  g\in M_{b}\left(  \mathcal{M}%
\right)  :g\text{ is }\mathcal{B}^{\prime}\text{-measurable}\right\}  .$

If $\mathcal{B}^{\prime}=\mathcal{B}_{I}:=q^{-1}\left(  \mathcal{B}\left(
I\right)  \right)  ,$ for any $\mu\in\mathfrak{M},g\in M_{b}\left(
\mathcal{M}\right)  ,$ there exists $\varphi_{\mu,g}\in M_{b}\left(  I\right)
$ such that $\mathcal{E}_{\mu}\left(  g|\mathcal{B}_{I}\right)  =\varphi
_{\mu,g}\circ q\ \mu-a.e..$ In particular, for any $g\in M_{b}\left(
\mathcal{M}\right)  $ such that $g=f\circ q$ with $f\in M_{b}\left(  I\right)
,\varphi_{\mu,g}=f$ for any $\mu\in\mathfrak{M}.$

Let $\mathbb{M}$ be the set of $\mu\in\mathfrak{M}$ such that, for any $f\in
M_{b}\left(  I\right)  ,\mu\left(  f\circ q\right)  =\lambda\left(  h_{\mu
}f\right)  ,$ with $h_{\mu}\in L^{1}\left(  I,\lambda\right)  .$ Clearly, if
$\mathfrak{M}^{\sim}:\mathfrak{M}/\sim$ is the set of equivalence classes of
the elements of $\mathfrak{M}$ w.r.t. the equivalence relation $\sim$ on
$\mathfrak{M}$ such that, for any $\mathcal{B}_{I}$-measurable $g\in
M_{b}\left(  \mathcal{M}\right)  ,$%
\begin{equation}
\mu\sim\nu\Longleftrightarrow\mu\left(  g\right)  =\nu\left(  g\right)  \ ,
\end{equation}
$\mathbb{M}$ is the subset of $\mathfrak{M}^{\sim}$ whose elements are a.c.
w.r.t. $\lambda.$ Since $\mathbf{1}_{\mathcal{M}}=\mathbf{1}_{I}\circ q,$ for
any $\mu$ in $\mathbb{M},\left\Vert \mu\right\Vert =\left\vert \mu\right\vert
\left(  \mathbf{1}_{\mathcal{M}}\right)  =\left\Vert h_{\mu}\right\Vert
_{L^{1}\left(  I,\lambda\right)  },$ hence $\forall\mu,\nu\in\mathbb{M}%
,\left\Vert \mu-\nu\right\Vert =\left\Vert h_{\mu}-h_{\nu}\right\Vert
_{L^{1}\left(  I,\lambda\right)  }.$ Therefore, if $\left\{  \mu_{n}\right\}
_{n\geq1}$ is a Cauchy sequence, then $\left\{  h_{\mu_{n}}\right\}  _{n\geq
1}$ is a Cauchy sequence in $L^{1}\left(  I,\lambda\right)  $ which implies
that $\mathbb{M}$ is a Banach space.

Let $\mathbb{B}_{1}$ be the Banach space $\left\{  \mu\in\mathbb{M}:h_{\mu}%
\in\mathbb{B}\right\}  .$ Then, if $\forall\eta\in spt\lambda_{\varepsilon},$%
\begin{align}
\left\Vert \left(  \bar{R}_{\eta}\right)  _{\#}\mu\right\Vert _{\mathbb{B}%
_{1}}  &  =\left\Vert \mathcal{L}_{\eta}h_{\mu}\right\Vert _{\mathbb{B}}%
\leq\varkappa\left\Vert h_{\mu}\right\Vert _{\mathbb{B}}+D\left\Vert h_{\mu
}\right\Vert _{L^{1}\left(  I,\lambda\right)  }\\
&  =\varkappa\left\Vert \mu\right\Vert _{\mathbb{B}_{1}}+D\left\Vert
\mu\right\Vert \ ,\nonumber
\end{align}
with $\varkappa$ and $D$ as in (\ref{LY0}),%
\begin{equation}
\left\Vert \mu P_{\overline{R}}\right\Vert _{\mathbb{B}}=\left\Vert
\mathcal{L}_{\varepsilon}h_{\mu}\right\Vert _{\mathbb{B}}\leq\varkappa
\left\Vert \mu\right\Vert _{\mathbb{B}_{1}}+D\left\Vert \mu\right\Vert \ .
\end{equation}
Moreover, for any $\mu\in\mathbb{B}_{1},$%
\begin{equation}
\left\Vert \left(  \bar{R}_{0}\right)  _{\#}\mu-\mu P_{\overline{R}%
}\right\Vert =\left\Vert \left(  \mathcal{L}_{0}-\mathcal{L}_{\varepsilon
}\right)  h_{\mu}\right\Vert _{L^{1}\left(  I,\lambda\right)  }\leq O\left(
\varepsilon\right)  \left\Vert h_{\mu}\right\Vert _{\mathbb{B}}=O\left(
\varepsilon\right)  \left\Vert \mu\right\Vert _{\mathbb{B}_{1}}\ .
\end{equation}

Therefore, all the assumptions A1-A6 in section \ref{SST} are satisfied and
the thesis follows from Corollary 1 in \cite{BHV} and Lemma \ref{Lf}.
\end{proof}

\subsection{Constant additive random type forcing\label{CF}}

We consider the special case of random perturbations of $\left(  \Phi_{0}%
^{t},t\geq0\right)  $ previously analysed realized by the addition to the
unperturbed phase vector field of a constant random term, namely
\begin{equation}
\phi_{\eta}:=\phi_{0}+\eta H\ ,\;\eta\in spt\lambda_{\varepsilon}\ ,
\end{equation}
with $H$ as in (\ref{ex}).

We will show that in this particular case the stochastic stability of the
unperturbed physical measure will follow directly from that of the
Poincar\'{e} map defined on a given Poincar\'{e} surface.

In \cite{PP} it has been shown that the Casimir function for the (+)
Lie-Poisson brackets associated to the $so\left(  3\right)  $ algebra formula
is a Lyapunov function for the ODE system (\ref{L1}). Namely, assuming
additive perturbations of the phase vector field as those given in (\ref{ex})
we can by rewrite formula (35) of \cite{PP} in our notation so that, for any
realization of the noise $\eta\in spt\lambda_{\varepsilon},$ by \cite{GMPV}
Section 2.1 we get
\begin{equation}
\left(  C\circ\Phi_{\eta}^{t}\right)  \left(  y\right)  \leq C\left(
y\right)  e^{-t\min\left(  1,\zeta,\beta\right)  }+\frac{\left\Vert H_{\eta
}\right\Vert ^{2}}{\left(  \min\left(  1,\zeta,\beta\right)  \right)  ^{2}%
}\left(  1+e^{-t\min\left(  1,\zeta,\beta\right)  }\right)  \ ,
\end{equation}
where $\mathbb{R}^{3}\ni y\longmapsto C\left(  y\right)  :=\left\langle
y,y\right\rangle =\left\Vert y\right\Vert ^{2}\in\mathbb{R}^{+}$ and $H_{\eta
}:=\eta H+H_{0}\in\mathbb{R}^{3},$ with $H_{0}:=\left(  0,0,-\beta\left(
\zeta+\gamma\right)  \right)  .$ Hence, choosing $t=\tau_{\eta}\left(
y\right)  $ we obtain
\begin{equation}
C\circ R_{\eta}\left(  y\right)  \leq a_{\varepsilon}C\left(  y\right)
+K_{\varepsilon}\left(  1+a_{\varepsilon}\right)  \ , \label{LY}%
\end{equation}
where
\begin{align}
a_{\varepsilon}  &  :=e^{-\min\left(  1,\zeta,\beta\right)  \inf_{\eta\in
spt\lambda_{\varepsilon}}\inf_{u\in\mathcal{M}}\tau_{\eta}\left(  u\right)
}\in\left(  0,1\right)  \;,\\
K_{\varepsilon}  &  :=\frac{\sup_{\eta\in spt\lambda_{\varepsilon}}\left\Vert
H_{\eta}\right\Vert ^{2}}{\left(  \min\left(  1,\zeta,\beta\right)  \right)
^{2}}>0\;.
\end{align}
Moreover, for any $\varsigma>0,$ (\ref{LY}) implies
\begin{align}
\left(  1+\varsigma C\right)  \circ R_{\eta}\left(  y\right)   &
\leq1+\varsigma a_{\varepsilon}C\left(  y\right)  +\varsigma K_{\varepsilon
}\left(  1+a_{\varepsilon}\right) \nonumber\\
&  =a_{\varepsilon}\left(  1+\varsigma C\left(  y\right)  \right)  +\bar
{K}_{\varepsilon}\ , \label{LY1}%
\end{align}
where $\bar{K}_{\varepsilon}:=\left(  1-a_{\varepsilon}\right)  +\varsigma
K_{\varepsilon}\left(  1+a_{\varepsilon}\right)  ,$ which entails for $P_{R}$
the weak drift condition
\begin{equation}
P_{R}\left(  1+\varsigma C\right)  \left(  y\right)  \leq a_{\varepsilon
}\left(  1+\varsigma C\left(  y\right)  \right)  +\bar{K}_{\varepsilon}\ .
\label{wd}%
\end{equation}

\begin{lemma}
$P_{R}$ admits an invariant probability measure.
\end{lemma}

\begin{proof}
Let $\mathbb{B}_{0}$ be the dual space of $C\left(  \mathcal{M}\right)  $ and
$\mathbb{B}_{\varsigma}$ be the dual space of $C_{\varsigma}\left(
\mathcal{M}\right)  $: the Banach space of real-valued functions on
$\mathcal{M}$ such that $\sup_{x\in\mathcal{M}}\frac{\left\vert \psi\left(
x\right)  \right\vert }{1+\varsigma C\left(  x\right)  }<\infty.\mathbb{B}%
_{\varsigma}\subseteq\mathbb{B}_{0}$ and (\ref{LY1}), (\ref{wd}) are
respectively equivalent to the Doeblin-Fortet conditions, namely, for any
$\mu\in\mathbb{B}_{\varsigma}$%
\begin{align}
\left\Vert \left(  R_{\eta}\right)  _{\#}\mu\right\Vert _{\varsigma}  &  \leq
a_{\varepsilon}\left\Vert \mu\right\Vert _{\varsigma}+\bar{K}_{\varepsilon
}\left\Vert \mu\right\Vert _{0}\ ,\\
\left\Vert \mu P_{R}\right\Vert _{\varsigma}  &  \leq a_{\varepsilon
}\left\Vert \mu\right\Vert _{\varsigma}+\bar{K}_{\varepsilon}\left\Vert
\mu\right\Vert _{0}\ , \label{wd1}%
\end{align}
where $\left\Vert \cdot\right\Vert _{0},\left\Vert \cdot\right\Vert
_{\varsigma}$ denote the norm of $\mathbb{B}_{0}$ and $\mathbb{B}_{\varsigma
}.$

Let $\mu\in\mathbb{B}_{\varsigma}$ such that $\left\Vert \mu\right\Vert
_{0}=1.$ By (\ref{wd1}) $P_{R}:\mathbb{B}_{\varsigma}\circlearrowleft$ and
$\forall n\geq1,$%
\begin{equation}
\left\Vert \mu P_{R}^{n}\right\Vert _{\varsigma}\leq a_{\varepsilon}%
^{n}\left\Vert \mu\right\Vert _{\varsigma}+\bar{K}_{\varepsilon}%
\frac{1-a_{\varepsilon}^{n}}{1-a_{\varepsilon}}\leq\left(  a_{\varepsilon}%
^{n}+\frac{\bar{K}_{\varepsilon}}{1-a_{\varepsilon}}\right)  \left\Vert
\mu\right\Vert _{\varsigma}\ . \label{wd2}%
\end{equation}
Moreover, since $\mathcal{M}$ is compact $\mathbb{B}_{0}$ is
tight\footnote{Anyway, if $\mathcal{M}$ were not compact, the tightness of the
sequence $\left\{  \mu_{n}\right\}  _{n\in\mathbb{N}}$ such that $\mu_{n}=\mu
P_{R}^{n},\mu\in\mathbb{B}_{\varsigma},$ would follow by (\ref{wd2}) since
$\forall\epsilon>0,\exists L_{\epsilon}>0$ s. t.$\forall L>L_{\epsilon},$
\[
\mu_{n}\left\{  \left(  1+\varsigma C\right)  >L\right\}  \leq\frac{1+\bar
{K}_{\varepsilon}}{L}<\epsilon\ .
\]
See also Lemma 4 in \cite{GHL}.}. Therefore, setting $\mu_{0}:=\mu$ and for
$k\geq1$ $\mu_{k}:=\mu P_{R}^{k},$ the sequence $\left\{  \nu_{n}\right\}
_{n\in\mathbb{Z}^{+}}$ such that $\nu_{0}:=\mu,\nu_{n}:=\frac{1}{n}\sum
_{k=0}^{n-1}\mu_{k},n\geq1,$ admits a weakly convergent subsequence whose
limit $\nu$ is $P_{R}$ invariant since, $\forall\psi\in C\left(
\mathcal{M}\right)  \subseteq C_{\varsigma}\left(  \mathcal{M}\right)  ,$%
\begin{equation}
\nu_{n}\left(  P_{R}\psi\right)  =\nu_{n}\left(  \psi\right)  +\frac{\mu
_{n+1}\left(  \psi\right)  -\mu\left(  \psi\right)  }{n}\ ,
\end{equation}
but
\begin{align}
\left\vert \mu_{n+1}\left(  \psi\right)  -\mu\left(  \psi\right)  \right\vert
&  \leq\left(  \left\Vert \mu_{n+1}\right\Vert _{\varsigma}+\left\Vert
\mu\right\Vert _{\varsigma}\right)  \sup_{x\in\mathcal{M}}\frac{\left\vert
\psi\left(  x\right)  \right\vert }{1+\varsigma C\left(  x\right)  }\\
&  \leq\left(  2+\frac{\bar{K}_{\varepsilon}}{1-a_{\varepsilon}}\right)
\left\Vert \mu\right\Vert _{\varsigma}\left\Vert \psi\right\Vert _{\infty
}\ .\nonumber
\end{align}

\end{proof}

The stochastic stability of $\mu_{R_{0}}$ then follows from Corollary
\ref{C1}, via Theorem \ref{Th1} and Theorem \ref{SSST}.

\part{Appendix}

Here we give examples of the cross-section $\mathcal{M}$ and of the maps
$T_{\eta}$ and $R_{\eta}$ discussed in the paper, as well as some comments on
the results achieved in our previous paper \cite{GMPV}. We also present the
proof of Proposition \ref{Prop1}.

\section{The Poincar\'{e} section $\mathcal{M}$}

Although what stated in Part I and Part II of the paper are not directly
affected by a particular choice of $\mathcal{M},$ to set up the problem in a
way easy to visualize we found useful to refer to the following examples.

Let us consider (\ref{L1}) with the parameter $\gamma,\zeta,\beta$ defining
the classical Lorenz flow and let $c_{0}:=\left(  0,0,-\left(  \gamma
+\zeta\right)  \right)  $ be the hyperbolic equilibrium point of (\ref{L1}).
If $O:\mathbb{R}^{3}\circlearrowleft$ is such that $O^{\text{t}}D\Phi_{0}%
^{t}\left(  c_{0}\right)  O$ is diagonal, we can distinguish between two cases:

\begin{enumerate}
\item in the first case we choose $\mathcal{M}\equiv\mathcal{M}^{\prime},$
where
\begin{equation}
\mathcal{M}^{\prime}:\mathcal{=}\left\{  y\in\mathbb{R}^{3}:\left\vert \left(
O^{\text{t}}y\right)  _{1}\right\vert ,\left\vert \left(  O^{\text{t}%
}y\right)  _{2}\right\vert \leq\frac{1}{2},\left(  O^{\text{t}}y\right)
_{3}=y_{3}=1-\left(  \gamma+\zeta\right)  \right\}  \ ;
\end{equation}

\item in the second, we choose $\mathcal{M}$ to be the Poincar\'{e} section
for the Lorenz'63 flow given in (\ref{L1}) constructed in \cite{GMPV}, namely
$\mathcal{M}:=\mathcal{M}^{\prime\prime},$ where
\begin{align}
\mathcal{M}^{\prime\prime}  &  :=\left\{  y\in\mathbb{R}^{3}:\left\vert
O^{\text{t}}y_{1}\right\vert ,\left\vert O^{\text{t}}y_{2}\right\vert
\leq\frac{1}{2},y_{3}\in\left[  -\left(  \gamma+\zeta\right)  ,1-\left(
\gamma+\zeta\right)  \right]  \ ;\right. \label{M2}\\
&  \left.  \left\langle \phi_{0}\left(  y\right)  ,\nabla\left\Vert
y\right\Vert ^{2}\right\rangle =0\ ,\ \left\langle \phi_{0}\left(  y\right)
,\nabla\left\langle \phi_{0}\left(  y\right)  ,\nabla\left\Vert y\right\Vert
^{2}\right\rangle \right\rangle \leq0\right\}  \ ,\nonumber
\end{align}
with $\phi_{0}$ given by (\ref{L1}), which is given by the union of two
$C^{2}$ compact manifolds $\mathcal{M}_{1},\mathcal{M}_{2}$ intersecting at
$c_{0}$ only and such that, if
\begin{equation}
\mathbb{R}^{3}\ni\left(  y_{1},y_{2},y_{3}\right)  \longmapsto\mathbf{P}%
\left(  y_{1},y_{2},y_{3}\right)  :=\left(  -y_{1},-y_{2},y_{3}\right)  \ ,
\label{P}%
\end{equation}
$\mathbf{P}\mathcal{M}_{1}=\mathcal{M}_{2}.$
\end{enumerate}

\subsection{The Poincar\'{e} map for $\mathcal{M}^{\prime\prime}\label{M"}$}

Since no confusion will arise, here we will drop the subscript $0$ to refer to
the unperturbed one-dimentional maps.

In Section 2.2.2 in \cite{GMPV} we showed that the Poincar\'{e} surface
$\mathcal{M}^{\prime\prime}$ defined in (\ref{M2}) is foliated by curves given
by the intersection of the spheres $\left\{  y\in\mathbb{R}^{3}:\left\Vert
y\right\Vert ^{2}=\mathfrak{r}\right\}  ,\mathfrak{r}\in\left[  \mathfrak{r}%
^{\ast},y_{3}^{2}\left(  c_{0}\right)  \right]  ,$ for some $\mathfrak{r}%
^{\ast}>0,$ with the surface
\begin{equation}
\left\{  y\in\mathbb{R}^{3}:\left\langle \phi_{0}\left(  y\right)
,\nabla\left\Vert y\right\Vert ^{2}\right\rangle =0,\left\langle \phi
_{0}\left(  y\right)  ,\nabla\left\langle \phi_{0}\left(  y\right)
,\nabla\left\Vert y\right\Vert ^{2}\right\rangle \right\rangle \leq0\right\}
\ ,
\end{equation}
where $\phi_{0}$ is defined in (\ref{L1}). By (\ref{P}), $\mathbf{P}$ defines
an equivalence relation between the points of $\mathcal{M}^{\prime\prime}$ and
we can identify $\mathcal{M}_{1}$ with the set $\mathcal{M}_{\mathbf{P}}$ of
the corresponding equivalence classes. Moreover, we can identify the interval
$\left[  \mathfrak{r}^{\ast},y_{3}^{2}\left(  c_{0}\right)  \right]  $ with
the collection of the equivalence classes of the points of $\mathcal{M}_{1},$
and so of $\mathcal{M}_{\mathbf{P}},$ having the same squared Euclidean
distance from the origin, i.e. those beloning to the same leaf of the just
mentioned foliation which we denote by $\mathfrak{C}.$ In \cite{PM} it has
been shown by numerical simultations that $\mathfrak{C}$ is invariant
exhibiting an automorphism $\hat{T}:\left[  \mathfrak{r}^{\ast},y_{3}%
^{2}\left(  c_{0}\right)  \right]  \circlearrowleft.$ By construction, the
Lorenz-type cusp map of the interval given in \cite{GMPV} fig.1, which we
denote by $\tilde{T},$ is the representation of $\hat{T}$ as a map of the
interval $\left[  0,1\right]  .$ Furthermore, if $c_{i}$ is the critical point
of $\phi_{0}$ different from $c_{0}$ having minimal Euclidean distance from
the component $\mathcal{M}_{i},i=1,2,$ in Section B of \cite{PM} it has also
been shown that the $k$-th branch of the induced map of $\tilde{T}$ on
$\left[  u_{0},1\right]  ,$ with $u_{0}:=\tilde{T}^{-1}\left(  1\right)  ,$
refers to trajectories of the system started at $\mathcal{M}_{i}$ that wind
$k$ times around $c_{j},i\neq j,$ before returning on $\mathcal{M}_{i},$ while
the trajectories of the points of $\mathcal{M}_{i}$ winding just around
$c_{i}$ before returning on $\mathcal{M}_{i}$ correspond to the branch
$\tilde{T}\upharpoonleft_{\left[  0,u_{0}\right]  }$ of $\tilde{T}$ (see
\cite{PM} fig.11). Therefore, from these last observations, the map $T$ (i.e.
$\bar{T}_{\eta}:\left[  -1,1\right]  \circlearrowleft$ in (\ref{Tbar}) for
$\eta=0$) can be reconstructed from $\tilde{T}$ and consequently also its
invariant measure. As a matter of fact, describing $\mathcal{M}_{1}$ as in
(\ref{M_loc_ch}), setting $\mathcal{O}\ni\left(  u,v\right)  \longmapsto
\mathbf{\bar{P}}\left(  u,v\right)  :=\left(  \mathbf{p}\left(  u\right)
,\mathbf{p}\left(  v\right)  \right)  ,$ with $\mathbb{R}\ni w\longmapsto
\mathbf{p}\left(  w\right)  :=-w\in\mathbb{R},$ and identifying the
unperturbed Poincar\'{e} map $R_{0}:\mathcal{M}^{\prime\prime}\circlearrowleft
$ with the skew-product $\mathcal{O}\bigvee\mathbf{\bar{P}}\mathcal{O}%
\ni\left(  u,v\right)  \longmapsto\left(  \bar{T}_{0}\left(  u\right)
,\Upsilon_{0}\left(  u,v\right)  \right)  \in\mathcal{O}\bigvee\mathbf{\bar
{P}}\mathcal{O},$ it follows that $\mathbf{P}\circ R_{0}=R_{0}\circ
\mathbf{P},$ hence, since $\mathbf{P}$ is an involution, $\tilde{T}%
=\mathbf{p}\circ\bar{T}_{0}\circ\mathbf{p\upharpoonleft}_{\left[  0,1\right]
}$ and, setting $\overline{\Upsilon}:=\mathbf{p}\circ\Upsilon_{0}%
\circ\mathbf{\bar{P}},$ we get the map $\hat{R}_{0}:\mathcal{M}_{\mathbf{P}%
}\circlearrowleft,$ which can be identified with the continuous skew-product
map $\mathcal{O}\ni\left(  u,v\right)  \longmapsto\left(  \tilde{T}\left(
u\right)  ,\overline{\Upsilon}\left(  u,v\right)  \right)  \in\mathcal{O}.$
The same considerations apply to perturbations of the phase velocity field
that preserves the same symmetry of the system under $\mathbf{P}$ (see
\cite{GMPV} Example 8). In this case rather than (\ref{Tbar}) we would have
had
\begin{align}
\left[  -1,1\right]  \ni u\longmapsto T_{\eta}\left(  u\right)   &
:=\mathbf{1}_{\left[  -1,-u_{0,\eta}\right]  }\left(  u\right)  \tilde
{T}_{\eta}\left(  -u\right)  -\mathbf{1}_{\left[  -u_{0,\eta},0\right]
}\left(  u\right)  \tilde{T}_{\eta}\left(  -u\right)  +\\
&  +\mathbf{1}_{\left[  0,u_{0,\eta}\right]  }\left(  u\right)  \tilde
{T}_{\eta}\left(  u\right)  -\mathbf{1}_{\left[  u_{0,\eta},1\right]  }\left(
u\right)  \tilde{T}_{\eta}\left(  u\right)  \in\left[  -1,1\right] \nonumber
\end{align}

On the other hand, if the perturbed phase velocity field $\phi_{\eta}$ is not
invariant under $\mathbf{P},$ the maps of the interval $\tilde{T}_{1}$ and
$\tilde{T}_{2},$ representing respectively the automorphisms, associated with
the pertubed flow, of the collections of the equivalence classes of the points
of $\mathcal{M}_{1}$ and $\mathcal{M}_{2}$ belonging to the leaves of
$\mathfrak{C},$ can be thought as perturbations of $\tilde{T}$ fitting into
the perturbing scheme given in Section \ref{SST}, if $\eta$ is sufficiently
small (see \cite{GMPV} Example 9).

\section{The one-dimensional map $T_{\eta}$\label{assT}}

In \cite{AMV} and \cite{HM} it has been proven that, in the case we choose
$\mathcal{M}:=\mathcal{M}^{\prime},$ identifying $I$ with $\left[  -\frac
{1}{2},\frac{1}{2}\right]  $ and, with abuse of notation, still denoting by
$\bar{T}_{\eta}:\left[  -\frac{1}{2},\frac{1}{2}\right]  \backslash\left\{
0\right\}  \longrightarrow\left[  -\frac{1}{2},\frac{1}{2}\right]  $ the
corresponding transitive, piecewise continuous map of the interval, there
exists $\alpha\in\left(  0,1\right)  ,G_{\eta}\in C^{\epsilon\alpha}\left(
\left[  -\frac{1}{2},\frac{1}{2}\right]  \right)  $ such that $\bar{T}_{\eta}$
is locally $C^{1+\alpha}$ on $\left[  -\frac{1}{2},\frac{1}{2}\right]
\backslash\{0\}$ and
\begin{equation}
\left[  -\frac{1}{2},\frac{1}{2}\right]  \backslash\left\{  0\right\}  \ni
u\longmapsto\bar{T}_{\eta}^{\prime}\left(  u\right)  :=\left\vert u\right\vert
^{-1+\alpha}G_{\eta}\left(  u\right)  \in\left[  -\frac{1}{2},\frac{1}%
{2}\right]  \ .
\end{equation}
Moreover, $\bar{T}_{\eta}\left(  0^{\mp}\right)  =\pm\frac{1}{2}.$ Namely, in
this case, $\bar{T}_{\eta}$ is the classical Lorenz-type map (see e.g. Fig.
3.24 in \cite{AP} for a sketch).

\begin{figure}[ptbh]
\centering
\resizebox{0.75\textwidth}{!}{\includegraphics{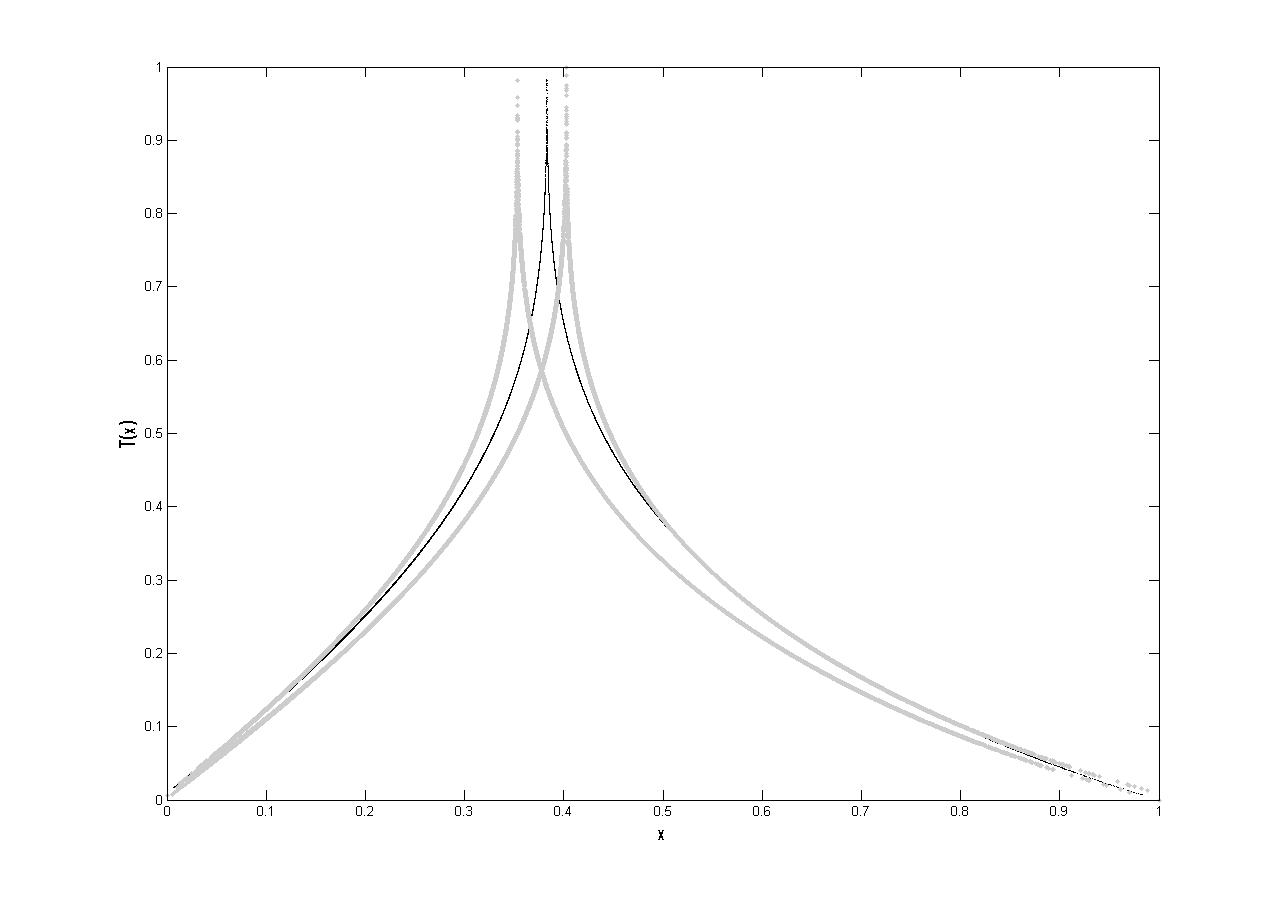}}
%If not, use
%\vspace{5cm}       % Give the correct figure height in cm
%Give a unique label
\caption{experimental plots of the unperturbed map $\tilde{T}_{0}$ (in black)
and of its perturbations (in grey).}%
\label{fig:2}%
\end{figure}

In the case $\mathcal{M}:=\mathcal{M}^{\prime\prime},\Gamma_{0}=\left\{
c_{0}\right\}  .$ Hence, we identify $I$ with $\left[  -1,1\right]  $ and,
again with abuse of notation, we denote by $\bar{T}_{\eta}:\left[
-1,1\right]  \circlearrowleft$ the map
\begin{align}
\left[  -1,1\right]  \ni u\longmapsto\bar{T}_{\eta}\left(  u\right)   &
:=\mathbf{1}_{\left[  -1,-u_{0,\eta}^{2}\right]  }\left(  u\right)  \tilde
{T}_{\eta,2}\left(  -u\right)  -\mathbf{1}_{\left[  -u_{0,\eta}^{2},0\right]
}\left(  u\right)  \tilde{T}_{\eta,2}\left(  -u\right)  +\label{Tbar}\\
&  +\mathbf{1}_{\left[  0,u_{0,\eta}^{1}\right]  }\left(  u\right)  \tilde
{T}_{\eta,1}\left(  u\right)  -\mathbf{1}_{\left[  u_{0,\eta}^{1},1\right]
}\left(  u\right)  \tilde{T}_{\eta,2}\left(  u\right)  \in\left[  -1,1\right]
\ ,\nonumber
\end{align}
where, for $i=1,2,\tilde{T}_{\eta,i}:\left[  0,1\right]  \circlearrowleft$ is
a transitive, continuous Lorenz-like cusp map of the interval of the type
studied in \cite{GMPV}, with two branches and a point $u_{0,\eta}^{i}%
\in\left[  0,1\right]  $ such that $\tilde{T}_{\eta,i}\left(  \left(
u_{0,\eta}^{i}\right)  ^{-}\right)  =\tilde{T}_{\eta,i}\left(  \left(
u_{0,\eta}^{i}\right)  ^{+}\right)  =1.$

In fact, in \cite{PM}, the paper that inspired our previous work \cite{GMPV},
the authors showed that the invariant measure for $\bar{T}_{\eta}$ can be
deduced directly from those of the $\tilde{T}_{\eta,i}$'s, whose local
behaviour is therefore the following (compare formulas (52)-(55) in
\cite{GMPV} and Figure 2):%
\begin{equation}
\tilde{T}_{\eta,i}\left(  u\right)  =\left\{
\begin{array}
[c]{ll}%
a_{\eta,i}u+b_{\eta,i}u^{1+c_{\eta,i}}+o(u^{1+c_{\eta,i}})\ ;\ a_{\eta
,i}\ ,\ c_{\eta,i}>1,b_{\eta,i}>0 & u\rightarrow0^{+}\\
1-A_{\eta,i}(u_{0,\eta}-u)^{B_{\eta,i}}+o((u_{0,\eta}-u)^{B_{\eta,i}%
})\ ;\ A_{\eta,i}>0,B_{\eta,i}\in\left(  0,1\right)  & u\rightarrow\left(
u_{0,\eta}^{i}\right)  ^{-}\\
1-A_{\eta,i}^{\prime}(u-u_{0,\eta})^{B_{\eta,i}^{\prime}}+o((u-u_{0,\eta
})^{B_{\eta,i}^{\prime}})\ ;\ A_{\eta,i}^{\prime}>0,B_{\eta,i}^{\prime}%
\in\left(  0,1\right)  & u\rightarrow\left(  u_{0,\eta}^{i}\right)  ^{+}\\
a_{\eta,i}^{\prime}\left(  1-u\right)  +b_{\eta,i}^{\prime}\left(  1-u\right)
^{1+c_{\eta,i}^{\prime}}+o(\left(  1-u\right)  ^{1+c_{\eta,i}^{\prime}%
})\ ;\ a_{\eta,i}^{\prime}\in\left(  0,1\right)  ,b_{\eta,i}^{\prime
}>0,c_{\eta,i}^{\prime}>1 & u\rightarrow1^{-}%
\end{array}
\right.  \ . \label{a1T}%
\end{equation}

We remark that to prove the stochastic stability of the invariant measure for
the evolution defined by the unperturbed map $T_{0}$ we needed supplementary
assumptions on $T_{0};$ see Section \ref{SST}.

In particular, in the case $\mathcal{M}:=\mathcal{M}^{\prime\prime},$ by
construction the stochastic stability of $T_{0}$ will follow from that of
$\tilde{T}_{0}.$

\section{Existence of invariant measures for the Lorenz-type cusp
map\label{Ltcm}}

In our previous paper \cite{GMPV} the one-dimensional Lorenz-cusp type map $T$
($\tilde{T}$ in the present paper) had a branch with first derivative less
than one on a open set but still bounded from below by a positive number. We
were unable to show that the derivative became globally larger that one for a
suitable power of the map and therefore we proceeded differently to prove the
statistical stability of the unperturbed invariant measure; namely we induced
and we showed that on a (lot of) induced set(s), the derivative of the first
return map was uniformly larger than one.

Anyway, the existence of an invariant measure for $T$ follows combining
Theorem 2 in \cite{Pi} and the results in section 4.2 of \cite{Bu} since one
can check by direct computation that the map
\begin{equation}
I\ni u\longmapsto\overline{T}\left(  u\right)  :=W\circ T\circ W^{-1}\left(
u\right)  \in I\ ,
\end{equation}
where $W$ is the distribution function associated to the probability measure
on $\left(  \left[  0,1\right]  ,\mathcal{B}\left(  \left[  0,1\right]
\right)  \right)  $ with density
\begin{equation}
\left[  0,1\right]  \ni x\longmapsto W^{\prime}\left(  x\right)
:=N_{\bar{\gamma},\bar{\beta}}e^{-\bar{\gamma}x}x^{\bar{\beta}}\left(
1-x\right)  ^{\bar{\beta}}%
\end{equation}
(see formulas (83) and (84) in \cite{GMPV}) for suitably chosen parameters
$\bar{\gamma},\bar{\beta}>0$ is such that $\inf\left\vert \overline{T}%
^{\prime}\right\vert >1.$

In particular, by (\ref{a1T}), for any $\eta\in spt\lambda_{\varepsilon},$
setting $B_{\eta}^{\ast}:=B_{\eta}\vee B_{\eta}^{\prime}$ and choosing
$0<\bar{\beta}<\inf_{\eta\in spt\lambda_{\varepsilon}}\frac{1}{B_{\eta}^{\ast
}}-1,\bar{\gamma}>\sup_{\eta\in spt\lambda_{\varepsilon}}\frac{\bar{\beta}%
+1}{1-x_{0,\eta}}\log\frac{1}{a_{\eta}^{\prime}},$ for any $\eta\in
spt\lambda_{\varepsilon},$ we get $\inf_{\eta\in spt\lambda_{\varepsilon}}%
\inf\left\vert \overline{T}_{\eta}^{\prime}\right\vert >1.$ H\"{o}lder
continuity of $\frac{1}{\overline{T}_{\eta}^{\prime}}$ follows from (\ref{H}).

\section{Statistical stability for Lorenz-like cusp maps}

We take the chance to rectify an incorrect statement we made in \cite{GMPV}
about the regularity properties of the one-dimensional map $T.$

Therefore, in this section, we will use the same notation we used in
\cite{GMPV}.

In that paper we state that the map $T$ was $C^{1+\iota},$ for some $\iota
\in\left(  0,1\right)  ,$ on the union of the two sets $(0,x_{0}),(x_{0},1),$
where the map was $1$ to $1.$ This is incorrect. What is true is that $T^{-1}$
is $C^{1+\iota},$ for some $\iota\in\left(  0,1\right)  ,$ on each open
interval $(0,x_{0}),(x_{0},1).$ Indeed, by the result in \cite{AM}, the stable
foliation for the classical Lorenz flow is $C^{1+\alpha}$ for some $\alpha
\in\left(  0.278,1\right)  ,$ which means, by (54) and (55) in \cite{GMPV},
that, for any $x\in(0,x_{0}),T^{^{\prime}}\left(  x\right)  =\left\vert
x_{0}-x\right\vert ^{1-B^{\prime}}\left[  1+G_{1}\left(  x\right)  \right]  $
with $G_{1}\in C^{\alpha B^{\prime}}(0,x_{0})$ and, for any $x\in
(x_{0},1),T^{\prime}\left(  x\right)  =\left\vert x-x_{0}\right\vert
^{1-B}\left[  1+G_{2}\left(  x\right)  \right]  $ with $G_{2}\in C^{\alpha
B}(x_{0},1).$ In particular this implies that for any couple of points $x,y$
belonging either to $(0,x_{0})$ or to $(x_{0},1)$
\begin{equation}
|T^{\prime}(x)-T^{\prime}(y)|\leq C_{h}\left\vert T^{\prime}(x)\right\vert
\left\vert T^{\prime}(y)\right\vert \left\vert x-y\right\vert ^{\iota}\ ,
\label{H}%
\end{equation}
where $\iota\in(0,1-B^{\ast}],$ with $B^{\ast}:=B\vee B^{\prime},$ and the
constant $C_{h}$ is independent of the location of $x$ and $y.$\footnote{In
\cite{APPV} section 5.3 is stated that the H\"{o}lder continuity of $\frac
{1}{T^{\prime}}$ on any domain $I_{i}$ of bijectivity of $T$ \ follows from
the H\"{o}lder continuity of $T^{\prime}\upharpoonleft_{I_{i}}.$ This cannot
be true in general, as one can see looking at the expression of $T^{\prime}$
given in \cite{HM} Proposition 2.6 for the geometric Lorenz flow. On the other
hand, in this and in similar cases the H\"{o}lder continuity of $\frac
{1}{T^{\prime}}\upharpoonleft_{I_{i}}$ can be directly proved (see also
\cite{AP} section 7.3.2).}

We now detail the modifications that these corrections induce on some of the
proofs of the results given in \cite{GMPV}, all the statements of our results
remaining unchanged.

\begin{description}
\item[Distortion] The proof of the boundedness of the distortion was sketched
in the footnote (1) of \cite{GMPV} by using arguments given in \cite{CHMV}. In
particular, in the initial formula (5) in \cite{CHMV} we need now to replace
the term $\left\vert \frac{D^{2}T(\xi)}{DT(\xi)}\right\vert |T^{q}\left(
x\right)  -T^{q}\left(  y\right)  |,$ where $\xi$ is a point between
$T^{q}\left(  x\right)  $ and $T^{q}\left(  y\right)  ,$ with $\frac
{1}{|DT(\xi)|}C_{h}|DT(T^{q}\left(  x\right)  )||DT(T^{q}\left(  y\right)
)||T^{q}\left(  x\right)  -T^{q}\left(  y\right)  |^{\iota}$ which is smaller
than $C_{h}\left(  |DT(T^{q}\left(  x\right)  )|\vee|DT(T^{q}\left(  y\right)
)|\right)  |T^{q}\left(  x\right)  -T^{q}\left(  y\right)  |^{\iota}$ by
monotonicity of $\left\vert DT\right\vert .$ The key estimate (11) in
\cite{CHMV} will reduce in our case to the bound of the quantity $\sup_{\xi
\in\lbrack b_{i+1},b_{i}]}|DT\left(  \xi\right)  ||b_{i}-b_{i+1}|.$ By using
for $DT$ the expressions given in the formulas (54) and (55) of \cite{GMPV},
and for the $b_{i}$ the scaling given in formula (75) of the same paper, we
immediately get that the above quantity is of order $\frac{1}{(\alpha^{\prime
})^{i}},$ which is enough to pursue the argument about the estimate of the
distortion presented in \cite{CHMV}.

\item[Perturbation] In order to prove the statistical stability of the
invariant measure $\mu_{T}$ for the evolution given by the map $T,$ the
perturbed map $T_{\epsilon}$ must satisfy at least the same regularity
properties required for $T.$ Therefore, in \cite{GMPV} Section 3.2:

\begin{itemize}
\item Assumption A should be replaced by the assumption that there exists
$\iota_{\epsilon}\in\left(  0,1\right)  $\ such that $T\upharpoonleft
_{(0,x_{\epsilon,0})},T\upharpoonleft_{(x_{\epsilon,0},1)}$ are $C^{1+\iota
_{\epsilon}}$ rather than assuming the stronger requirement that $T_{\epsilon
}$ is $C^{1+\iota_{\epsilon}}$ on $(0,x_{\epsilon,0})\cup(x_{\epsilon,0},1);$

\item Assumption C should be replaced by the requirement that the
multiplicative H\"{o}lder constant $C_{h}^{\epsilon}$ of $D\left(
T_{\epsilon}^{-1}\right)  $ will converge to $C_{h}$ when $\epsilon
\rightarrow0.$
\end{itemize}

\item We have then to modify the bounds (92), (99) and (114) in \cite{GMPV}
which are all of the form $|DT_{\epsilon}(a)-DT_{\epsilon}(a_{\epsilon})|,$
with $a$ $\epsilon$-close to $a_{\epsilon}.$ We have $|DT_{\epsilon
}(a)-DT_{\epsilon}(a_{\epsilon})|\leq C_{h}^{\epsilon}|DT_{\epsilon
}(a)||DT_{\epsilon}(a_{\epsilon})||a-a_{\epsilon}|.$ By the continuity and the
monotonicity of $DT_{\epsilon}$ we can replace $a_{\epsilon}$ in
$|DT_{\epsilon}(a_{\epsilon})|$ with $a$ or with another given point between
$a$ and $x_{0};$ finally we use the limit (88) in Assumption B to conclude.
\end{description}

\section{Proof of Proposition \ref{Prop1}}

\begin{proof}
The invariance of $\mu_{\overline{\mathbf{R}}}$ under $\overline{\mathbf{R}}$
follows by (\ref{mu_R}), since
\begin{equation}
\mu_{\overline{\mathbf{R}}}\left(  \psi\circ\overline{\mathbf{R}}\right)
:=\lim_{n\rightarrow\infty}\int\mu_{\mathbf{T}}\left(  du,d\omega\right)
\inf_{x\in q^{-1}\left(  u\right)  }\psi\circ\overline{\mathbf{R}}%
^{n+1}\left(  x,\omega\right)  =\mu_{\overline{\mathbf{R}}}\left(
\psi\right)  \ . \label{invmu_R}%
\end{equation}
Hence, since
\begin{align}
\lim_{n\rightarrow\infty}\int\mu_{\mathbf{T}}\left(  du,d\omega\right)
\inf_{x\in q^{-1}\left(  u\right)  }\psi\circ\overline{\mathbf{R}}^{n}\left(
x,\omega\right)   &  \leq\lim_{n\rightarrow\infty}\int\mu_{\mathbf{T}}\left(
du,d\omega\right)  \left(  \left(  \mathbf{1}_{q^{-1}\left(  u\right)  }\circ
p\right)  \psi\right)  \circ\overline{\mathbf{R}}^{n}\left(  x,\omega\right)
\\
&  \leq\lim_{n\rightarrow\infty}\int\mu_{\mathbf{T}}\left(  du,d\omega\right)
\sup_{x\in q^{-1}\left(  u\right)  }\psi\circ\overline{\mathbf{R}}^{n}\left(
x,\omega\right)  \ ,\nonumber
\end{align}
it is enough to prove that
\begin{equation}
\lim_{n\rightarrow\infty}\int\mu_{\mathbf{T}}\left(  du,d\omega\right)
\inf_{x\in q^{-1}\left(  u\right)  }\psi\circ\overline{\mathbf{R}}^{n}\left(
x,\omega\right)  =\lim_{n\rightarrow\infty}\int\mu_{\mathbf{T}}\left(
du,d\omega\right)  \sup_{x\in q^{-1}\left(  u\right)  }\psi\circ
\overline{\mathbf{R}}^{n}\left(  x,\omega\right)  \ . \label{mu+=mu-}%
\end{equation}
By (\ref{defQ}), (\ref{ass1}) and the definition of $\bar{R}_{\pi\left(
\omega\right)  },\forall\omega\in\Omega,$%
\begin{equation}
\overline{\mathbf{R}}\left(  Q^{-1}\left(  u,\omega\right)  \right)  \subset
Q^{-1}\left(  \mathbf{T}\left(  u,\omega\right)  \right)  \ . \label{contr}%
\end{equation}
Therefore,
\begin{align}
\sup_{x\in q^{-1}\left(  u\right)  }\psi\circ\overline{\mathbf{R}}%
^{n+k}\left(  x,\omega\right)   &  =\sup_{\left(  x,\omega^{\prime}\right)
\in Q^{-1}\left(  u,\omega\right)  }\psi\circ\overline{\mathbf{R}}%
^{n+k}\left(  x,\omega^{\prime}\right) \label{supRn-}\\
&  \leq\sup_{\left(  x,\omega^{\prime}\right)  \in Q^{-1}\left(
\mathbf{T}^{k}\left(  u,\omega\right)  \right)  }\psi\circ\overline
{\mathbf{R}}^{n}\left(  x,\omega^{\prime}\right) \nonumber\\
&  =\sup_{\left(  x,\omega^{\prime}\right)  \in\left\{  \left(  y,\omega
^{\prime\prime}\right)  \in\mathcal{M}\times\Omega\ :\ Q\left(  y,\omega
^{\prime\prime}\right)  =\mathbf{T}^{k}\left(  u,\omega\right)  \right\}
}\psi\circ\overline{\mathbf{R}}^{n}\left(  x,\omega^{\prime}\right) \nonumber
\end{align}
\ and
\begin{align}
\inf_{x\in q^{-1}\left(  u\right)  }\psi\circ\overline{\mathbf{R}}%
^{n+k}\left(  x,\omega\right)   &  =\inf_{\left(  x,\mathbf{\omega}^{\prime
}\right)  \in Q^{-1}\left(  u,\omega\right)  }\psi\circ\overline{\mathbf{R}%
}^{n+k}\left(  x,\omega^{\prime}\right) \label{infRn+}\\
&  \geq\inf_{\left(  x,\omega^{\prime}\right)  \in Q^{-1}\left(
\mathbf{T}^{k}\left(  u,\omega\right)  \right)  }\psi\circ\overline
{\mathbf{R}}^{n}\left(  x,\omega^{\prime}\right) \nonumber\\
&  =\inf_{\left(  x,\omega^{\prime}\right)  \in\left\{  \left(  y,\omega
^{\prime\prime}\right)  \in\mathcal{M}\times\Omega\ :\ Q\left(  y,\omega
^{\prime\prime}\right)  =\mathbf{T}^{k}\left(  u,\omega\right)  \right\}
}\psi\circ\overline{\mathbf{R}}^{n}\left(  x,\omega^{\prime}\right)
\ .\nonumber
\end{align}
Hence, by the invariance of $\mu_{\mathbf{T}}$ under $\mathbf{T},$%
\begin{gather}
\int\mu_{\mathbf{T}}\left(  du,d\omega\right)  \sup_{x\in q^{-1}\left(
u\right)  }\psi\circ\overline{\mathbf{R}}^{n+k}\left(  x,\omega\right)
\leq\int\mu_{\mathbf{T}}\left(  du,d\omega\right)  \sup_{\left(
x,\omega^{\prime}\right)  \in\left\{  \left(  y,\omega^{\prime\prime}\right)
\in\mathcal{M}\times\Omega\ :\ Q\left(  y,\omega^{\prime\prime}\right)
=\mathbf{T}^{k}\left(  u,\omega\right)  \right\}  }\psi\circ\overline
{\mathbf{R}}^{n}\left(  x,\omega^{\prime}\right) \label{mupsi+-}\\
=\int\left(  \mathbf{T}_{\#}^{k}\mu_{\mathbf{T}}\right)  \left(
du,d\omega\right)  \sup_{\left(  x,\omega^{\prime}\right)  \in\left\{  \left(
y,\omega^{\prime\prime}\right)  \in\mathcal{M}\times\Omega\ :\ Q\left(
y,\omega^{\prime\prime}\right)  =\left(  u,\omega\right)  \right\}  }\psi
\circ\overline{\mathbf{R}}^{n}\left(  x,\omega^{\prime}\right) \nonumber\\
=\int\mu_{\mathbf{T}}\left(  du,d\omega\right)  \sup_{\left(  x,\omega
^{\prime}\right)  \in Q^{-1}\left(  u,\omega\right)  }\psi\circ\overline
{\mathbf{R}}^{n}\left(  x,\omega^{\prime}\right) \nonumber\\
=\int\mu_{\mathbf{T}}\left(  du,d\omega\right)  \sup_{x\in q^{-1}\left(
u\right)  }\psi\circ\overline{\mathbf{R}}^{n}\left(  x,\omega\right) \nonumber
\end{gather}
so that the sequence $\left\{  \int\mu_{\mathbf{T}}\left(  du,d\omega\right)
\sup_{x\in q^{-1}\left(  u\right)  }\psi\circ\overline{\mathbf{R}}^{n}\left(
x,\omega\right)  \right\}  _{n\geq1}$ is decreasing. On the other hand,
\begin{gather}
\int\mu_{\mathbf{T}}\left(  du,d\omega\right)  \inf_{x\in q^{-1}\left(
u\right)  }\psi\circ\overline{\mathbf{R}}^{n+k}\left(  x,\omega\right)
\geq\int\mu_{\mathbf{T}}\left(  du,d\omega\right)  \inf_{\left(
x,\omega^{\prime}\right)  \in\left\{  \left(  y,\omega^{\prime\prime}\right)
\in\mathcal{M}\times\Omega\ :\ Q\left(  y,\omega^{\prime\prime}\right)
=\mathbf{T}^{k}\left(  u,\omega\right)  \right\}  }\psi\circ\overline
{\mathbf{R}}^{n}\left(  x,\omega^{\prime}\right) \label{mupsi-+}\\
=\int\left(  \mathbf{T}_{\#}^{k}\mu_{\mathbf{T}}\right)  \left(
du,d\omega\right)  \inf_{\left(  x,\omega^{\prime}\right)  \in\left\{  \left(
y,\omega^{\prime\prime}\right)  \in\mathcal{M}\times\Omega\ :\ Q\left(
y,\omega^{\prime\prime}\right)  =\left(  u,\omega\right)  \right\}  }\psi
\circ\overline{\mathbf{R}}^{n}\left(  x,\omega^{\prime}\right) \nonumber\\
=\int\mu_{\mathbf{T}}\left(  du,d\omega\right)  \inf_{\left(  x,\omega
^{\prime}\right)  \in Q^{-1}\left(  u,\omega\right)  }\psi\circ\overline
{\mathbf{R}}^{n}\left(  x,\omega^{\prime}\right) \nonumber
\end{gather}
so that $\left\{  \int\mu_{\mathbf{T}}\left(  du,d\omega\right)  \inf_{x\in
q^{-1}\left(  u\right)  }\psi\circ\overline{\mathbf{R}}^{n}\left(
x,\omega\right)  \right\}  _{n\geq1}$ is increasing. Since $\forall\omega
\in\Omega,\psi\left(  \cdot,\omega\right)  \in C_{b}\left(  \mathcal{M}%
\right)  $ and $\forall u\in I,q^{-1}\left(  u\right)  \subset\mathcal{M}$ is
compact, by (\ref{contr}), $\forall\varepsilon^{\prime}>0,\exists
\delta_{\varepsilon^{\prime}}>0,n_{\varepsilon^{\prime}}>0$ such that $\forall
n\geq n_{\varepsilon^{\prime}},\omega\in\Omega,u\in I,\operatorname*{diam}%
p\left(  \overline{\mathbf{R}}^{n}\left(  Q^{-1}\left(  u,\omega\right)
\right)  \right)  <\delta_{\varepsilon^{\prime}}$ and $\forall\left(
x,\omega^{\prime}\right)  ,\left(  y,\omega^{\prime}\right)  \in
\overline{\mathbf{R}}^{n}\left(  Q^{-1}\left(  u,\omega\right)  \right)
,$\linebreak$\left\vert \psi\left(  x,\omega^{\prime}\right)  -\psi\left(
y,\omega^{\prime}\right)  \right\vert <\varepsilon^{\prime},$ therefore
\begin{align}
&  \left\vert \int\mu_{\mathbf{T}}\left(  du,d\omega\right)  \sup_{x\in
q^{-1}\left(  u\right)  }\psi\circ\overline{\mathbf{R}}^{n}\left(
x,\omega\right)  -\int\mu_{\mathbf{T}}\left(  du,d\omega\right)  \inf_{x\in
q^{-1}\left(  u\right)  }\psi\circ\overline{\mathbf{R}}^{n}\left(
x,\omega\right)  \right\vert \\
&  \leq\int\mu_{\mathbf{T}}\left(  du,d\omega\right)  \left\vert \sup_{\left(
x,\omega^{\prime}\right)  \in Q^{-1}\left(  u,\omega\right)  }\psi
\circ\overline{\mathbf{R}}^{n}\left(  x,\omega^{\prime}\right)  -\inf_{\left(
x,\omega^{\prime}\right)  \in Q^{-1}\left(  u,\omega\right)  }\psi
\circ\overline{\mathbf{R}}^{n}\left(  x,\omega^{\prime}\right)  \right\vert
\leq\varepsilon^{\prime}\ ,\nonumber
\end{align}
that is (\ref{mu+=mu-}) holds.

Thus, the map
\begin{equation}
L_{\mathbb{P}}^{1}\left(  \Omega,C_{b}\left(  \mathcal{M}\right)  \right)
\ni\psi\longmapsto\hat{\mu}\left(  \psi\right)  :=\lim_{n\rightarrow\infty
}\int\mu_{\mathbf{T}}\left(  du,d\omega\right)  \left(  \left(  \mathbf{1}%
_{q^{-1}\left(  u\right)  }\circ p\right)  \psi\right)  \circ\overline
{\mathbf{R}}^{n}\left(  x,\omega\right)  \in\mathbb{R}%
\end{equation}
is a non negative linear functional such that $\hat{\mu}\left(  1\right)  =1$
and, by (\ref{mu+=mu-}),
\begin{equation}
\hat{\mu}\left(  \psi\right)  =\lim_{n\rightarrow\infty}\int\mu_{\mathbf{T}%
}\left(  du,d\omega\right)  \inf_{x\in q^{-1}\left(  u\right)  }\psi
\circ\overline{\mathbf{R}}^{n}\left(  x,\omega\right)  \ .
\end{equation}
Moreover, $\Omega$ is compact under the product topology, then the space of
quasi-local continuous functions $C_{\infty}\left(  \Omega,C_{b}\left(
\mathcal{M}\right)  \right)  $\footnote{$C_{\infty}\left(  \Omega,C_{b}\left(
\mathcal{M}\right)  \right)  $ is the uniform closure of the set of local
(also called cylinder) functions on $\Omega$ with values in $C_{b}\left(
\mathcal{M}\right)  .$ Since $\Omega$ is compact
\[
C_{\infty}\left(  \Omega,C_{b}\left(  \mathcal{M}\right)  \right)  =C\left(
\Omega,C_{b}\left(  \mathcal{M}\right)  \right)  =C_{K}\left(  \Omega
,C_{b}\left(  \mathcal{M}\right)  \right)
\]
the last term being the Banach space of continuous $C_{b}\left(
\mathcal{M}\right)  $-valued functions on $\Omega$ with compact support, which
is dense in $L_{\mathbb{P}}^{1}\left(  \Omega,C_{b}\left(  \mathcal{M}\right)
\right)  .$} is dense in $L_{\mathbb{P}}^{1}\left(  \Omega,C_{b}\left(
\mathcal{M}\right)  \right)  ,$ therefore, by the Riesz-Markov-Kakutani
theorem there exists a unique Radon measure $\mu_{\overline{\mathbf{R}}}$ on
$\left(  \mathcal{M}\times\Omega,\mathcal{B}\left(  \mathcal{M}\right)
\otimes\mathcal{F}\right)  $ such that $\mu_{\overline{\mathbf{R}}}=\hat{\mu
}\upharpoonleft_{C_{K}\left(  \Omega,C_{b}\left(  \mathcal{M}\right)  \right)
}.$

The injectivity of the correspondence $\mu_{\mathbf{T}}\longmapsto
\mu_{\overline{\mathbf{R}}}$ follows from the fact that,\linebreak%
\ $\forall\varphi\in L_{\mathbb{P}}^{1}\left(  \Omega,C_{b}\left(  I\right)
\right)  ,\varphi\circ Q\in L_{\mathbb{P}}^{1}\left(  \Omega,C_{b}\left(
\mathcal{M}\right)  \right)  $ and
\begin{gather}
\int\mu_{\mathbf{T}}\left(  du,d\omega\right)  \inf_{x\in q^{-1}\left(
u\right)  }\varphi\circ Q\circ\overline{\mathbf{R}}^{n}\left(  x,\omega
\right)  =\int\mu_{\mathbf{T}}\left(  du,d\omega\right)  \inf_{x\in
q^{-1}\left(  u\right)  }\varphi\circ\mathbf{T}^{n}\circ Q\left(
x,\omega\right) \\
=\int\mu_{\mathbf{T}}\left(  du,d\omega\right)  \inf_{x\in q^{-1}\left(
u\right)  }\varphi\circ\mathbf{T}^{n}\left(  q\left(  x\right)  ,\omega
\right)  =\mu_{\mathbf{T}}\left(  \varphi\circ\mathbf{T}^{n}\right)
=\mu_{\mathbf{T}}\left(  \varphi\right)  \ .\nonumber
\end{gather}
Therefore, if there exist $\mu_{\mathbf{T}}^{\prime}$ invariant under
$\mathbf{T}$ such that
\begin{equation}
\mu_{\overline{\mathbf{R}}}\left(  \psi\right)  :=\lim_{n\rightarrow\infty
}\int\mu_{\mathbf{T}}^{\prime}\left(  du,d\omega\right)  \inf_{x\in
q^{-1}\left(  u\right)  }\psi\circ\overline{\mathbf{R}}^{n}\left(
x,\omega\right)  \ ,
\end{equation}
then $\mu_{\mathbf{T}}^{\prime}\left(  \varphi\right)  =\mu_{\mathbf{T}%
}\left(  \varphi\right)  ,$ hence $\mu_{\mathbf{T}}^{\prime}=\mu_{\mathbf{T}%
}.$

The proof of the ergodicity of $\mu_{\overline{\mathbf{R}}}$ under the
hypotesis of the ergodicity of $\mu_{\mathbf{T}}$ is identical to that of
Corollary 7.25 in Section 7.3.4 of \cite{AP}.
\end{proof}


\begin{thebibliography}{99999}                                                                                            %


\bibitem[Al]{Al}Alsmeyer G. \emph{The Markov Renewal Theorem and Related
Results} Markov Proc. Rel. Fields \textbf{3}, 103--127 (1997).

\bibitem[Ar]{Ar}Arnold L. \emph{Random Dynamical Systems }Springer (2003).

\bibitem[As]{As}Asmussen S., \emph{Applied Probability and Queues, II edition
}Springer (2003).

\bibitem[ABS]{ABS}V.S. Afraimovic, V.V. Bykov, Sili'nikov L.P. \emph{The
origin and structure of the Lorenz attractor} Dokl. Akad. Nauk SSSR
\textbf{234}, no. 2, 336--339 (1977).

\bibitem[AM]{AM}Ara\'{u}jo V., Melbourne I. \emph{Existence and smoothness of
the stable foliation for sectional hyperbolic attractors} Bull. London Math.
Soc. \textbf{49} 351--367 (2017).

\bibitem[AMV]{AMV}Ara\'{u}jo V., Melbourne I., Varandas P. \emph{Rapid Mixing
for the Lorenz Attractor and Statistical Limit Laws for Their Time-1 Maps}
Commun. Math. Phys. \textbf{340}, 901--938 (2015).

\bibitem[AP]{AP}Ara\'{u}jo V., Pacifico M. J. \emph{Three-dimesional flows
}Springer (2010).

\bibitem[APPV]{APPV}Ara\'{u}jo V., Pacifico M. J., Pujals E. R., Viana, M.
\emph{Singular-hyperbolic attractors are chaotic} Trans. Amer. Math. Soc.
\textbf{361} n. 5, 2431--2484 (2008).

\bibitem[AS]{AS}Alves, J. F.,Soufi, M. \emph{Statistical stability of
geometric Lorenz attractors} Fundamenta Mathematicae \textbf{224}, 219--231 (2014).

\bibitem[BHV]{BHV}Bahsoun W., Hu H.-Y. Vaienti S. \emph{Pseudo-orbits,
stationary measures and metastability }Dyn. Syst. \textbf{29} n. 3 322--336 (2014).

\bibitem[BR]{BR}Bahsoun W., Ruziboev M. \emph{On the stability of statistical
properties for the Lorenz attractors with }$C^{1+\alpha}$ \emph{stable
foliation} Ergodic Theor. and Dyn. Sys. \textbf{39}, n.12, 3169--3184 (2019).

\bibitem[Bu]{Bu}Butterley O. \emph{Area expanding} $C^{1+\alpha}$
\emph{Suspension Semiflows} Commun. Math. Phys. \textbf{325} n.2, 803--820 (2014).

\bibitem[CHMV]{CHMV}Cristadoro, G-P., Haydn, N., Marie, Ph., Vaienti, S.
\emph{Statistical properties of intermittent maps with unbounded derivative}
Nonlinearity 23, 1071--1095 (2010).

\bibitem[CMP]{CMP}S. Corti, F. Molteni, T. N. Palmer \emph{Signature of recent
climate change in frequencies of natural atmospheric circulation regimes}
Letters to Nature \textbf{398}, 799--802 (1999).

\bibitem[CSG]{CSG}Chekroun, M. D., Simonnet E., Ghil M. \emph{Stochastic
climate dynamics: random attractors and time-independent invariant measures}
Phisica D \textbf{240} n.21, 1685--1700 (2011).

\bibitem[Da]{Da}M. H. A. Davis \emph{Markov Models and Optimization} Springer (1993).

\bibitem[GHL]{GHL}D. Guibourg, L. Herv\'{e}, J. Ledoux \emph{Quasi-compactness
of Markov kernels on weighted-supremum spaces and geometrical ergodicity }arXiv:1110.3240v5

\bibitem[GL]{GL}Galatolo S., Lucena R. \emph{Spectral gap and quantitative
statistical stability for systems with contracting fibers and Lorenz-like
maps} Discrete Contin. Dyn. Syst. \textbf{40}, n.3, 1309--1360 (2020).

\bibitem[GMPV]{GMPV}Gianfelice M., Maimone F., Pelino V., Vaienti S. \emph{On
the recurrence and robust properties of the Lorenz'63 model} Commun. Math.
Phys. \textbf{313}, 745--779 (2012).

\bibitem[GW]{GW}Gukenheimer J., Williams R.F. \emph{Structural stability of
Lorenz attractors} Inst. Hautes Etudes Sci. Publ. Math. \textbf{50}, 59--72 (1979).

\bibitem[HM]{HM}Holland M., Melbourne I. \emph{Central limit theorems and
invariance principles for Lorenz attractors} J. London Math. Soc. \textbf{2},
76, 345364 (2007).

\bibitem[HS]{HS}Hirsch M. W., Smale S. \emph{Differential Equations, Dynamical
Systems, and Linear Algebra} Academic press (1978).

\bibitem[Ke]{Ke}Keller H. \emph{Attractors and bifurcations of the stochastic
Lorenz system} Report 389, Institut f\"{u}r Dynamische Systeme,
Universit\"{a}t Bremen (1996).

\bibitem[Ki]{Ki}Kifer Y. \emph{Random Perturbations of Dynamical Systems}
Birkh\"{a}user (1988).

\bibitem[KL]{KL}Keller G., Liverani C. \emph{Stability of the spectrum for
transfer operators} Ann. Scuola Norm. Sup. Pisa Cl. Sci. (4) \textbf{28} n. 1,
141--152 (1999).

\bibitem[KS]{KS}Korolyuk V., Swishchuk A. \emph{Semi-Markov Random Evolutions
}Springer (1995).

\bibitem[Lo]{Lo}Lorenz E. N. \emph{Deterministic Nonperiodic Flow} J. Atmos.
Sci., vol. 20, 130--141 (1963).

\bibitem[Me]{Me}Metzger R. J. \emph{Stochastic Stability for Contracting
Lorenz Maps and Flows} Comm. Math. Phys. \textbf{212}, 277--296 (2000).

\bibitem[MT]{MT}Meyn S. Tweedie R. L. \emph{Markov Chains and Stochastic
Stability, Second Edition} Cambridge University Press (2009).

\bibitem[NVKDF]{NVKDF}Nevo G., Vercauteren N., Kaiser A., Dubrulle B., Faranda
D. \emph{A statistical-mechanical approach to study the hydrodynamic stability
of stably stratified atmospheric boundary layer} Phys. Rev. Fluids \textbf{2},
084603 (2017).

\bibitem[Pa]{Pa}Palmer T. N. \emph{A Nonlinear Dynamical Perspective on
Climate Prediction} Journal of Climate \textbf{12} n.2, 575--591 (1999).

\bibitem[Pi]{Pi}Pianigiani G. \emph{Existence of invariant measures for
piecewise continuous transformations} Annales Polonici Mathematici
\textbf{XL}, 3945 (1981).

\bibitem[PM]{PM}Pelino V., Maimone F. \emph{Energetics, skeletal dynamics, and
long term predictions on Kolmogorov-Lorenz systems} Physical Review E,
\textbf{76}, 046214 (2007).

\bibitem[PP]{PP}Pasini A., Pelino V. \emph{A unified view of Kolmogorov and
Lorenz systems} Phys. Lett. A \textbf{275}, 435--445 (2000).

\bibitem[Sa]{Sa}Saussol, B. \emph{Absolutely continuous invariant measures for
multidimensional expanding maps }Israel Journal of Mathematics \textbf{116}
223--248 (2000).

\bibitem[Sc]{Sc}Schmallfu\ss , B. \emph{The random attractor of the stochastic
Lorenz system} Z. angew. Math. Phys. \textbf{48} 951--975 (1997).

\bibitem[Su]{Su}Sura P. \emph{A general perspective of extreme events in
weather and climate} Atmospheric Research \textbf{101} 1--21 (2011).

\bibitem[Tu]{Tu}W. Tucker \emph{A rigorous ODE solver and Smale's 14th
problem} Foundations of Computational Mathematics, \textbf{2:1} 53--117 (2002).

\bibitem[Vi]{Vi}Viana M. \emph{Stochastic Dynamics of Deterministic Systems
}IMPA notes (1997).
\end{thebibliography}
\end{document}